\documentclass[a4paper,11pt]{article}
\usepackage[utf8]{inputenc} 
\usepackage[english]{babel} 

\usepackage{amsmath}   
\usepackage{amssymb}   
\usepackage{amsthm}    
\usepackage{amsfonts}  
\usepackage{mathrsfs}  
\usepackage{esint}
\usepackage{mathtools}
\usepackage{microtype}
\usepackage{authblk}

\usepackage[shortlabels, inline]{enumitem}

\usepackage{marginnote}
\usepackage{lmodern}
\theoremstyle{plain}
\newtheorem{theorem}{Theorem}[section]
\newtheorem{corollary}[theorem]{Corollary}
\newtheorem{lemma}[theorem]{Lemma}
\newtheorem{proposition}[theorem]{Proposition}
\newtheorem{definition}[theorem]{Definition}

\theoremstyle{remark}
\newtheorem{remark}[theorem]{Remark}
\newtheorem{example}[theorem]{Example}

\newtheorem{claim}[theorem]{Claim}

\usepackage{xcolor}
\usepackage[pdfpagelabels]{hyperref}
\definecolor{green}{HTML}{2ECC71}
\definecolor{blue}{HTML}{3498DB}
\definecolor{red}{HTML}{E74C3C}
\hypersetup{%
  colorlinks,
  linktocpage,
  urlcolor  = blue,
  citecolor = green,
  linkcolor = red}

\DeclarePairedDelimiterX{\hsp}[2]{\lparen}{\rparen}{#1 \mid #2}
\DeclarePairedDelimiterX{\pairing}[2]{\langle}{\rangle}{#1 \mid #2}

\newcommand{\ttabs}[1]{\lvert#1\rvert}							
\newcommand{\tabs}[1]{\big\lvert#1\big\rvert}						
\newcommand{\abs}[1]{\left\lvert#1\right\rvert}						

\newcommand{\tnorm}[1]{\big\lVert#1\big\rVert}						
\newcommand{\norm}[1]{\left\lVert#1\right\rVert}					

\newcommand{\ttset}[1]{\{#1\}}									
							
\newcommand{\set}[1]{\left\{#1\right\}}							

\newcommand{\tseq}[1]{{\big(#1\big)}}
\newcommand{\seq}[1]{\left(#1\right)}								

\newcommand{\tparen}[1]{\big({#1}\big)}							
\newcommand{\paren}[1]{\left(#1\right)}							

\newcommand{\tbraket}[1]{\big[#1\big]}							
\newcommand{\braket}[1]{\left[#1\right]}							

\newcommand{\ttscalar}[2]{\langle #1 \, \big |\, #2\rangle}				
			
\newcommand{\scalar}[2]{\left\langle #1 \,\middle |\, #2\right\rangle}		

\newcommand{\av}[1]{\left\langle #1 \right\rangle}

\newcommand{\comma}{\,\,\mathrm{,}\;\,}

\newcommand{\fstop}{\,\,\mathrm{.}}

\newcommand{\emparg}{{\,\cdot\,}}								
\newcommand{\diff}{\mathop{}\!\mathrm{d}}						
\renewcommand{\paragraph}[1]{\smallskip\noindent\emph{#1.}\;\,}

\def\N{{\mathbb N}}    
\def\R{{\mathbb R}}

\def\M{{M}}
\def\g{{g}}
\newcommand{\vol}{\mathsf{vol}}

\DeclareMathOperator{\scal}{\mathsf{scal}}

\DeclareMathOperator{\dd}{d}

\DeclareMathOperator{\Pol}{\mathsf{P}}

\def\Leb{\mathcal{L}}
\def\C{\mathcal{C}}

\def\Hil{{H}} 

\newcommand{\ee}{e}

\begin{document}

\title{
\sc Conformally Invariant Random Fields,\\Liouville Quantum Gravity measures,\\and Random Paneitz Operators\\ on Riemannian Manifolds of Even Dimension\\[3cm]}

\author[1]{Lorenzo Dello Schiavo} 
\author[2]{Ronan Herry}
\author[3]{Eva Kopfer}
\author[3]{Karl-Theodor Sturm}
\affil[1]{Institute of Science and Technology Austria \authorcr lorenzo.delloschiavo{@}ist.ac.at \vspace{.2cm}}
\affil[2]{IRMAR, Université de Rennes 1 \authorcr ronan.herry{@}univ-rennes1.fr \vspace{.2cm}}
\affil[3]{%
Institute for Applied Mathematics --
University of Bonn
\authorcr
eva.kopfer@iam.uni-bonn.de
\authorcr
sturm@iam.uni-bonn.de
}

\maketitle

\tableofcontents

\section*{Introduction} 

Conformally invariant random objects on the complex plane or on Riemannian surfaces are a central topic of current research and play a fundamental role in many mathematical theories.
The last two decades have seen an impressive wave of 
fascinating constructions, deep insights and spectacular results for various
conformally (quasi-) invariant random objects, most prominently the
Gaussian Free Field, the Liouville quantum gravity measure, the Brownian map, and the SLE curves.

In this paper, we use ideas from conformal geometry in higher dimension 
to establish the foundations for a mathematical theory of conformally invariant random fields and Liouville Quantum Gravity on compact Riemannian manifolds of even dimension.

\paragraph{\bf Co-polyharmonic Gaussian Fields}
We construct conformally quasi-invariant random Gaussian fields $h$ on admissible Riemannian manifolds $(M,g)$ of arbitrary even dimension.
The covariance kernels of these centered Gaussian fields, naively interpreted as
 $k_g(x,y)={\mathbf E}\big[h(x)\,h(y)\big]$, exhibits a logarithmic divergence
 \begin{equation}
 \Bigg|k_g(x,y)-\log\frac1{d(x,y)}\Bigg|\le C\,.
 \end{equation}
As for the Gaussian Free Field, these random fields, called \emph{co-polyharmonic Gaussian fields}, are not classical functions on $M$ but rather 
elements in the Sobolev space  $H^{-s}_g\coloneqq H^{-s}(M,g)$  for any $s>0$.
By construction, they annihilate constants, that is~$\scalar{h}{\mathbf 1}_g=0$, where $\scalar{\emparg}{\emparg}_g$ denotes the pairing between $H^{-s}_g$ and $H^s_g$.

We prove (Thm.\ \ref{t:ConformalChangeField}) that co-polyharmonic Gaussian fields are conformally quasi-invariant:
let~$h_g$ denote the co-polyharmonic Gaussian field for $(M,g)$ and $h_{g'}$ that for~$(M,g')$ with $g'=e^{2\varphi}g$ and $\varphi$ smooth,
then 
\begin{equation*}
h_{g'}\stackrel{\rm (d)}= h_g-C,
\end{equation*}
where $C$ is an appropriate random variable that ensures that the right-hand side annihilates constants. 
To get rid of this additive correction term, one can consider the
`random variable' 
 $\hbar_g=h_g + a$, called  \emph{ungrounded co-polyharmonic Gaussian field} with $h_g$ as above and $a$ distributed according to the Lebesgue measure on $\R$. Then
\begin{equation}
{\hbar}_{g'}\stackrel{\rm (d)}= {\hbar}_g.
\end{equation}
Here and in all the paper we use $\stackrel{\rm (d)}=$ to indicate that two random variables have the same law.

\paragraph{\bf Co-polyharmonic operators}
For a given closed manifold $(M,g)$ of even dimension $n$, the covariance kernel $k_g$ is the integral kernel for 
the inverse of the operator $\mathsf p_g\coloneqq a_n\Pol_g$  on the `grounded' $L^2$-space
$\mathring L^2_g\coloneqq \big\{u\in L^2(\M,\vol_g): \, \int u\,d\vol_g=0\big\}$.
Here, $a_n\coloneqq
2(4\pi)^{-n/2}/\Gamma(n/2)$ and
\begin{equation}\label{eq:IntroPaneitz}
\Pol_g=(-\Delta_g)^{n/2} + \text{low order terms}
\end{equation}
denotes the \emph{co-polyharmonic operator}
or \emph{Graham--Jenne--Mason--Sparling operator  of maximal order}.
The operator~$\mathsf{P}_{g}$ plays the role of a conformally invariant power of the Laplacian and has been first defined in \cite{GJMS92}.
For $n=2$, the non-negative operator~$\mathsf{P}_{g}$ is just $-\Delta_g$, the negative of the Laplacian, and for $n=4$ it is the celebrated \emph{Paneitz operator} \cite{Paneitz}.

The  co-polyharmonic Gaussian field $h$ on $(M,g)$ can easily be constructed in terms of the eigenbasis $(\psi_j)_{j\in\N_0}$ of $\mathsf p_g$: with  $(\nu_j)_{j\in\N_0}$ the corresponding eigenvalues and any sequence $(\xi_j)_{j\in\N}$ of independent standard normal random variables, then (Prop.\ \ref{t:approx-cph})
\begin{align}\label{intro-approx}
h=\lim_{\ell\to\infty} h_\ell\,,\qquad
h_\ell(x)\coloneqq  \sum_{j=1}^\ell \frac{1}{\sqrt{\nu_j}}\psi_j(x)\, \xi_j\,,
\end{align}
with convergence in quadratic mean.

\paragraph{\bf Liouville Quantum Gravity measures} 
We then define the \emph{plain Liouville Quantum Gravity measure} $\mu^{\gamma h}_g$ on $M$ 
for every parameter $\gamma\in\R$ with $|\gamma|<\sqrt{2n}$ as a random finite measure.
Employing Kahane's idea of \emph{Gaussian multiplicative chaos} (see \cite{Kah85}), we define (Thm.~\ref{t:existence-liouville}) the measure $\mu^{\gamma h}_g$ as the almost sure limit (in the usual sense of weak convergence of measures) of the sequence $(\mu^{\gamma h_{\ell}}_g)_{\ell\in\N}$ of finite measures on $M$ given by
\begin{equation*}
d\mu_g^{\gamma h_{\ell}}(x)\coloneqq 
e^{\gamma h_\ell(x)-\frac{\gamma^2}2k_{\ell}(x,x)}\,d \vol_\g(x),
\end{equation*} 
with $h_\ell$ as in  \eqref{intro-approx} and $k_\ell(x,x)\coloneqq {\mathbf E}[h_\ell(x)^2]=\sum_{j=1}^\ell\frac1{\nu_j}{\psi_j(x)^2}$.

We establish that almost surely the measure $\mu^{\gamma h}_g$ is a finite measure on $M$ with full topological support, and for every $s>\gamma^2/4$, it does not charge sets of vanishing $H^s$-capacity (Thm.~\ref{t:PolarNegligible}).
In particular, it does not charge sets of vanishing $H^{n/2}$-capacity since $|\gamma|<\sqrt{2n}$ throughout.
If, moreover, $|\gamma|<2$ then, almost surely, $\mu^{\gamma h}_g$ does not charge sets of vanishing $H^1$-capacity. 

 The plain Liouville Quantum Gravity measure has the following crucial quasi-invariance property (Thm.~\ref{t:invariance-lqg}): if~$\g'=e^{2\varphi}\g$
 then
\begin{equation}\label{eq:invariance-liouville-intro}
\mu^{\gamma h'}_{g'}\stackrel{\rm (d)}= \ee^{F}\,\mu_g^{\gamma h},
\end{equation}
 {where~$F$ is the random variable explicitly given in~\eqref{eq:ConformalFactorPlainLQG}}.
We also study the 
   \emph{(`adjusted') Liouville Quantum Gravity measures}, 
   denoted by 
   $\bar{\mu}$, which is equal, up to a deterministic multiplicative weight
 to the plain Liouville Quantum Gravity measure defined above, shares many properties with it, and
 exhibits a simpler quasi-invariance property (Thm.\ \ref{t:invariance-adjusted}):
  \begin{equation*}
    \bar{\mu}^{\gamma h'}_{g'} \stackrel{\rm (d)}= e^{-\gamma \xi}\, e^{(n + \gamma^2/2) \varphi} \,\bar{\mu}^{\gamma h}_{g}\comma
  \end{equation*}
  where $\xi\coloneqq\langle h\rangle_{g'}$ is a normal random variable. Passing from the grounded to the ungrounded co-polyharmonic field this finally reads as
 \begin{equation}
    \bar{\mu}^{\gamma \hbar'}_{g'} \stackrel{\rm (d)}=  e^{(n + \gamma^2/2) \varphi} \,\bar{\mu}^{\gamma \hbar}_{g}\fstop
  \end{equation}

\paragraph{\bf Random quadratic forms} 
With respect to the Liouville Quantum Gravity measure, we can define a variety of random objects which play a fundamental role in  geometric analysis, spectral theory, and probabilistic potential theory.

Restricting to the range $\gamma \in (-2,2)$, we construct (Thm.\ \ref{t:liouville-bm}) a \emph{random Dirichlet form} on $L^2(M,\mu_g^{\gamma h})$ by:
\begin{equation*}
  {\mathcal E}_g^h(u,u)\coloneqq  \int_\M |\nabla u|^2\,d\vol_\g, \qquad{\mathcal D}({\mathcal E}_g^h)\coloneqq  H^1(\M)\cap L^2(M,\mu_g^{\gamma h}) \fstop
\end{equation*}
The associated reversible and continuous Markov process is the \emph{Liouville Brownian motion} (see \cite{GarRhoVar14,GarRhoVar16} and~\cite{Ber15} for two independent 
constructions on the plane).
It is obtained from the standard Brownian motion on $(M,g)$ through time change.
The  new time scale is given as the right inverse of the additive functional
\begin{equation*}
A^h_t=\lim_{\ell\to\infty}\int_0^t \exp\Big(\gamma\, h_\ell(X_s)-\frac{\gamma^2}2 k_\ell(X_s,X_s)\Big)ds \fstop
\end{equation*}
In dimension $n > 2$, however, this Liouville Brownian motion has no canonical invariance property under conformal transformations simply because its generator is not conformally invariant.
 
To obtain conformally quasi-invariant random objects 
in higher dimensions, our starting point, in Theorem \ref{t:liouville-co-polyharmonic}, is the random co-polyharmonic form 
\begin{align*}
\mathfrak{E}_g^h (u,v)\coloneqq  \int_\M u\, \Pol_g v \, d\vol_g \comma \qquad  \mathcal{D}(\mathfrak{E}_g^h)\coloneqq  H^{n/2}(M)\cap L^2(M,\mu_g^{\gamma h})\comma
\end{align*}
(rather than the random Dirichlet form) which in the full range $\gamma\in (-\sqrt{2n},\sqrt{2n})$ is, almost surely, a well-defined non-negative closed symmetric bilinear form on~$L^2(M,\mu_g^{\gamma h})$. 
It allows us to define \emph{random co-polyharmonic operators} $\Pol^h_g$.
The associated random co-polyharmonic heat flow $e^{-t \Pol^h_g}$ is the gradient flow for the 
quadratic functional $\frac{1}{2}\mathfrak{E}_g^h$ in the random landscape $L^2(M,\mu_g^{\gamma h})$ (Prop.\ \ref{t:gradient-flow}).

In Theorem \ref{t:invariance-random-co-polyharmonic}, we show that the random co-polyharmonic operators share the fundamental quasi-invariance property
\begin{equation*}
\Pol^{h'}_{g'}\stackrel{\rm (d)}= \ee^{-F}\,\Pol_g^h\comma
\end{equation*}
with $F$ as in \eqref{eq:invariance-liouville-intro}.

\paragraph{\bf Polyakov--Liouville measure} 
Finally, we propose  an ansatz for a Liouville Conformal Field Theory on closed manifolds of arbitrary even dimension.
Our approach, based on Branson's $Q$-curvature,
provides a rigorous meaning to the  Polyakov--Liouville measure $\boldsymbol{\pi}_g$, informally given as 
\begin{equation*}
  \boldsymbol{\pi}_g(dh) = \exp(-S_g(h)) dh
\end{equation*}
with the (non-existing) uniform distribution $dh$ on the set of fields 
and the action, {{heuristically} considered {first} in~\cite[Eqn.~(1.1)]{LevyOz},}
\begin{equation*}
  S_{\g}(h) \coloneqq   \int_M\Big( \tfrac{1}2\,h\mathsf p_gh +\Theta\,Q_g h + {m} e^{\gamma h}\Big) d \vol_g\comma
\end{equation*}
where  $m,\Theta,\gamma>0$ are parameters (subjected to some restrictions). 
To rigorously introduce the \emph{(`adjusted') Polyakov--Liouville measure} $\boldsymbol{\bar\pi}_g$, we define it as
\[
\boldsymbol{\bar\pi}_{g}\coloneqq \sqrt{\frac{\vol_g(M)}{\det'(\frac1{2\pi}\mathsf p_g)}}\cdot  \boldsymbol{\bar\nu}^{*}_{g}
\]
with
\begin{equation}
d{\boldsymbol{\bar\nu}}^*_g(h+a)\coloneqq \exp\Big(-\Theta\scalar{h+a}{Q_g}_g- m \,e^{\gamma a}\bar\mu^{\gamma h}_{g}(M)\Big)\, da\,d{\boldsymbol{\nu}}_g(h)\comma
\end{equation}
where $\boldsymbol{\nu}_g$ denotes the law of the co-polyharmonic Gaussian field,
informally understood as $d \boldsymbol{\nu}_g(h)=\frac1{Z_g}\exp\big(-\frac{1}2\,\scalar{h}{\mathsf p_gh}\big)\,dh$
with $Z_g\coloneqq\sqrt{{\vol_g(M)}/{\det'(\frac1{2\pi}\mathsf p_g)}}$, and where $\bar\mu_{g}^{\gamma h}$ denotes the adjusted Liouville Quantum Gravity measure.
We prove (Thm.~\ref{fin-adj-pol-liou}) that for admissible manifolds of negative total $Q$-curvature, the measures $\boldsymbol{\bar\nu}_g^*$ and $\boldsymbol{\bar\pi}_g$ are  finite. 
Moreover, for the particular choice  $\Theta\coloneqq a_n\big(\frac n \gamma+\frac\gamma2\big)$, 
 the measure $\boldsymbol{\bar\nu}_g^*$
 is \emph{quasi-invariant modulo shift} with $T \colon h \mapsto h - (\frac{n}\gamma+\frac\gamma2) \varphi$  and \emph{A-type conformal anomaly}
\begin{equation}
\frac{d\boldsymbol{\bar\nu}^*_{g'}}{dT_*\boldsymbol{\bar\nu}^*_{g}} \ = \ 
\exp\bigg(
\bigg(\frac n \gamma+\frac\gamma2\bigg)^2\, \bigg[\frac12
\mathfrak{p}_g(\varphi,\varphi)+a_n\int_M \varphi\,Q_g\,d\vol_g\bigg]\bigg)
\end{equation}
(Thm.~\ref{conf-adj-pol-liou}).
In dimensions $2$ and $4$, we also determine the \emph{`full' conformal anomaly} for transformations of the 
Polyakov--Liouville measure $\boldsymbol{\bar\pi}_g$, confirming the result from \cite{GuiRhoVar19} in the case $n=2$ and providing a new formula in the case $n=4$:
  \begin{align*}
\frac{d\boldsymbol{\bar\pi}_{g'}}{dT_*\boldsymbol{\bar\pi}_{g}} \ = \ &  \exp\left(\left[\frac{7}{45}+\Big(\frac{n}{\gamma} + \frac{\gamma}{2}\Big)^{2}\right]\cdot \left[\frac12\mathfrak{p}_{g}(\varphi,\varphi) + {a_n} \int \varphi\, Q_{g}\, d \vol_{g} \right] \right)\\
 \cdot&\exp\left( \frac1{45 \pi^2} \left[-\int \scal_{g'}^2 d\vol_{g'}+\int \scal_{g}^2 d\vol_{g}\right]\right)
 \cdot \exp\left(\frac1{1440 \pi^2}\int \varphi |W|^2d\vol_g\right)\fstop
  \end{align*}

\paragraph{\bf Admissible manifolds} 
 Co-polyharmonic Gaussian fields do \emph{not} exist on {every}  Riemannian manifold.
A closed (i.e.~compact and without boundary) even-dimensional Riemannian manifold~$(M,g)$ is called \emph{admissible} if~$\Pol_g>0$ on~$\mathring L^2_g$.
Admissibility is a conformal invariance.
All closed, non-negatively curved Einstein manifolds are admissible, and so are all closed hyperbolic manifolds with spectral gap $\lambda_1>\frac{n(n-2)}4$.
Of course, all closed 2-dimensional Riemannian manifolds are admissible.

We prove
(Thm.\ \ref{est-k-log}, see also \cite[Lemma 2.1]{Ndiaye}) 
that for every admissible manifold, the inverse of~$\mathsf p_g\coloneqq a_n\Pol_g$ on~$\mathring L^2_g$ has an integral kernel $k_g$ which annihilates constants and satisfies
 \begin{equation}
\abs{k_g(x,y)-\log\frac{1}{d(x,y)}}\leq C \fstop
 \end{equation}
%
%

\paragraph{\bf The two-dimensional case}
Even in the case of surfaces, our approach provides new insights for the study of two-dimensional random objects.
It applies to closed Riemannian surfaces of arbitrary genus and thus some of our results are new in the two-dimensional setting.
%
%
%
In particular, we present a detailed discussion of  the difference between plain and adjusted LQG measures as well as novel conformal transformation formulas for the plain LQG measure (Thm.~\ref{t:invariance-adjusted}, Cor.~\ref{extend-conf-meas}, \ref{t:invariance-lqg-lebesgue}). Moreover, considering the  plain LQG measure (rather than the adjusted one) proves crucial for obtaining the new result on eigenbasis approximation 
\begin{equation}
\mu_g^{\gamma h_\ell}\to \mu_g^{\gamma h}
\end{equation}
 (Thm.~\ref{t:EigenfunctionApproxGMC}). 
 
 In addition to these novel results, our approach recovers many of the famous results concerning the Gaussian Free Field and the associated Liouville Quantum Gravity measure in dimension~$2$, and for the first time it provides an \emph{intrinsic} Riemannian, conformally quasi-invariant extension to higher dimensions.

\paragraph{\bf Probabilistic context}
In dimension~$2$, conformally invariant random objects appear naturally in the study of continuum statistical models.
The celebrated \emph{Gaussian Free Field} arises as the scaling limit of various discrete models of random surfaces, for instance discrete Gaussian Free Fields  or harmonic crystals~\cite{She07}.
A planar conformally invariant random object of fundamental importance is the \emph{Schramm--Loewner evolution} \cite{Lawler05,Schramm07,Lawler18}.
It plays a central role in many problems in statistical physics and satisfies some conformal invariance.
The Schramm-Loewner evolution and the two-dimensional Gaussian Free Field are deeply related.
For instance, level curves of the Discrete Gaussian Free Field  converge to~$\mathsf{SLE}_{4}$~\cite{SchShe09}, and zero contour lines of the Gaussian Free Field are well-defined random curves distributed according to~$\mathsf{SLE}_{4}$~\cite{SchShe13}.
The work~\cite{MSImaginary1}, and subsequent works in its series, thoroughly study the relation between the Schramm-Loewner evolution and Gaussian free field on the plane.
Motivated by Polyakov's informal formulation of Bosonic string theory \cite{Pol81a,Pol81b}, the papers \cite{DS11,DKRV16,GuiRhoVar19} construct mathematically the \emph{Liouville Quantum Gravity} on some surfaces and study its conformal invariance properties.
Formally speaking, the Liouville Quantum Gravity is a random surface obtained by random conformal transform of the Euclidean metric, where the conformal weight is the Gaussian Free Field.
Since the Gaussian Free Field is only a distribution, we do not obtain a random Riemannian manifold but rather a random metric measure space.
The aforementioned works construct the random measure based on a renormalization procedure due to Kahane \cite{Kah85}.
This renormalization depends on a roughness parameter $\gamma$ and works only for $|\gamma| < 2$.
In~\cite{MSLQG1} and subsequent work in its series, J. Miller and S. Sheffield prove that for the value $\gamma = \sqrt{8/3}$ the Liouville Quantum Gravity coincides with the Brownian map, that is a random metric measure space arising as a universal scaling limit of random trees and random planar graphs (see \cite{LeGallMiermont12,LeGall19} and the references therein).
More recently, \cite{DDDF20,GM21} establish the existence of the Liouville Quantum Gravity metric for $\gamma \in (0,2)$.
We also note that the case where $\gamma$ is complex valued is studied in \cite{GHPR20,Pfeffer21}.

\paragraph{\bf Geometric context}
Despite the fact that  the main attention of the probability community has focused so far on the two-dimensional case, (non-random) conformal geometry in dimensions $n>2$ is a fascinating field of research.
Earlier results by \cite{Trudinger68,Aubin76,Schoen84} completely solve the Yamabe problem \cite{Yamabe60} on compact manifolds: every compact Riemannian manifold is conformally equivalent to a manifold with constant scalar curvature.
In the general case, despite ground-breaking results by \cite{ES86} using the conformal Laplacian, a complete picture is still far from reach.
On surfaces, the works \cite{OPS88,OPS89} initiate an approach to the problem based on Polyakov's variational formulation for the determinant of $\Delta_g$ \cite{Pol81a,Pol81b}: they show that constant curvature metrics have maximum determinant.
In dimension $4$, \cite{BransonOrsted91} derives an equivalent of Polyakov's formula for a conformal version of $\Delta_g^{2}$, known as the Paneitz operator and \cite{ChangYang95} finds extremal metrics associated to some functionals of the conformal Laplacian and the Paneitz operator.
\cite{GJMS92} constructs higher order equivalent of Paneitz operators, that is conformally invariant powers of $\Delta_g$, based on \cite{FeffermanGraham85}, see also \cite{GZ03}.
In particular in dimension 4, remarkable spectral properties, sharp functional inequalities and rigidity results have been derived in \cite{
ChangYang95}, \cite{Gursky99}, \cite{ChaGurYan02}, and \cite{ChaGurYan03}. See also \cite{DjaHebLed00} for various such results in higher dimensions.

\paragraph{\bf Higher dimensional random geometry}
Conformally (quasi-)invariant extension  for any of these random objects to higher dimensions were also discussed in 
\cite{LevyOz} and \cite{Cercle}.
Indeed, until we finished and circulated a first version of our paper, we were not aware of any of these contributions.
The ansatz of B.\ Cerclé  \cite{Cercle} is similar to ours, limited, however, to the sphere in $\R^{n+1}$ and relying on an \emph{extrinsic} approach, based on stereographic projections of the Euclidean space,  whereas ours is an intrinsic, Riemannian approach. 
In particular, our approach also applies to huge classes of manifolds with negative total $Q$-curvature, a necessary condition for finiteness of the partition function and for well-definedness of the (normalized) Polyakov--Liouville measure.
The approach by T.\ Levy \& Y.\ Oz \cite{LevyOz} does not provide mathematical results or insights. It is more on a heuristic level, not taking care, however, of the necessary positivity of the respective GJMS operators, and not addressing any details of the necessary re-normalization procedure.
Our intrinsic Riemannian approach also has the advantage that it canonically provides approximations by discrete polyharmonic fields and associated Liouville measures \cite{LDS-RH-EK-KTS2}.
Our construction of Liouville Brownian motion in higher dimensions and random GJMS operators is not anticipated so far, even not for the sphere or other particular cases.

Beyond that, the authors in \cite{DGZ24} carry out the first major step towards a Liouville Quantum Gravity metric on Euclidean space with arbitrary dimension. They prove tightness of the exponential metrics for log-correlated Gaussian fields on $\R^n$, generalizing the result in \cite{DDDF20}. 

\paragraph{\bf Acknowledgements}
The authors are grateful to Masha Gordina for helpful references, and to Nathana\"el Berestycki, Baptiste Cercl\'e, and Ewain Gwynne for valuable comments on the first circulated version of this paper. 
They also would like to thank Sebastian Andres, Peter Friz, and Yizheng Yuan for pointing out an erroneous formulation  in the previous version of Thm.~\ref{t:liouville-bm}.
Moreover, KTS would like to express his thanks to Sebastian Andres, Matthias Erbar, Martin Huesmann, and Jan Mass for stimulating discussions on previous attempts to this project.

LDS gratefully acknowledges financial support from the European Research Council (grant agreement No 716117, awarded to J.~Maas), from the Austrian Science Fund (FWF) project \href{https://doi.org/10.55776/ESP208}{{10.55776/ESP208}}, and from the Austrian Science Fund (FWF) project \href{https://doi.org/10.55776/F65}{{10.55776/F65}}.

RH, EK, and KTS gratefully acknowledge funding by the Deutsche Forschungsgemeinschaft through the project `Random Riemannian Geometry' within the SPP 2265 `Random Geometric Systems', 
 through the Hausdorff Center for Mathematics (project ID 390685813), and through  project B03  within the CRC 1060 (project ID 211504053).
 RH and KTS also  gratefully acknowledges financial support from the European Research Council through the ERC AdG `RicciBounds' (grant agreement 694405).

Data sharing not applicable to this article as no datasets were generated or analyzed during the current study.

\section{Co-polyharmonic operators on even-dimensional manifolds}
Throughout the sequel, without explicitly mentioning it, all manifolds under consideration are assumed to be smooth, connected and closed (i.e.~compact and without boundary).

\subsection{Riemannian manifolds and conformal classes}
Given a closed Riemannian manifold $(M,g)$, we denote its dimension by $n$, its volume measure by $\mathsf{vol}=\mathsf{vol}_g$, its scalar curvature by $\mathsf{scal}$ or by $\mathsf{R}$, its Ricci curvature tensor by $\mathsf{Ric} =\{\mathsf{Ric}_{ij} : i,j = 0,\dots, n\}$, and its Laplace--Beltrami operator by $\Delta=\Delta_g$, the latter being a negative operator.
The spectral gap (or in other words, the first non-trivial eigenvalue) of $-\Delta_g$ on $(M,g)$ is denoted by $\lambda_1 > 0$.

For $u\in L^1(M,\mathsf{vol}_g)$, we set $\langle u \rangle_g \coloneqq  \frac1{\mathsf{vol}_g(M)} \int_M u\,d \mathsf{vol}_g$ and $\pi_{g}(u) 
 \coloneqq  u-\langle u\rangle_g$. We use the short hand notation $L^2_g \coloneqq L^2(M,\mathsf{vol}_g)$ and define the \emph{grounded $L^2$ space} by
\begin{equation*}
    \mathring{L}^{2}_g
    \coloneqq  \{ u \in L^2_g : \langle u \rangle_{g} = 0 \}.
\end{equation*}

\begin{definition}\label{d:ConformalEquiv}
\begin{enumerate}[$(i)$, wide]
\item\label{i:d:ConformalEquiv:1} Two Riemannian metrics $g$ and $g'$ on a manifold $M$ are \emph{conformally equivalent} if there exists a (`weight') function $\varphi\in\mathcal{C}^\infty(M)$ such that $g'=e^{2\varphi}\,g$.
The class of metrics which are conformally equivalent to a given metric $g$ is denoted by~$[g]$. 

\item\label{i:d:ConformalEquiv:2} Two Riemannian manifolds
$(M,g)$ and $(M',g')$ are \emph{conformally equivalent} if there exists a $\mathcal{C}^\infty$-diffeomorphism  $\Phi:M\to M'$ and a function $\varphi\in\mathcal{C}^\infty(M)$ such that the pull back of $g'$ is conformally equivalent to $g$ with weight $\varphi$, that is
\[
\Phi^*\g'=\ee^{2\varphi}\,\g \fstop
\]
In other words, if $(M',g')$ is isometric to $(M,g'')$, and  $g''$ and $g$ are conformally equivalent.
The class of Riemannian manifolds which are conformally equivalent to a given Riemannian manifold $(M,g)$ is denoted by $[(M,g)]$. 

\item\label{i:d:ConformalEquiv:3} A family of operators $A_g$ on a family of conformally equivalent Riemannian manifolds $(M,g)$ is called \emph{conformally quasi-invariant} if for every pair $(M,g)$ and $(M',g')$ of conformally equivalent manifolds and associated maps $\Phi$ and $\varphi$ as in~\ref{i:d:ConformalEquiv:2} there exists a function $f_\varphi$ on $M$ such that
\begin{equation}e^{f_\varphi}\cdot (A_{e^{2\varphi}g}u)\circ\Phi=A_g(u\circ\Phi)\comma \qquad u\in\mathcal{C}^\infty(M) \fstop
\end{equation}
In conformal geometry, such an operator is usually called \emph{conformally covariant}.
However, in this paper, the notion \emph{covariance} is already used for the key quantity for characterizing probabilistic dependencies.
\end{enumerate}
\end{definition}

The study of conformal mappings as in~\ref{i:d:ConformalEquiv:2} above is of particular interest in dimension~2 as powerful uniformization results are available.
For instance, Riemann's mapping theorem \cite{OsgoodRMapping} states that every non-empty simply connected open strict subset of $\mathbb{C}$ is conformally equivalent to the open unit disk.
More generally, the uniformization theorem \cite{PoincareUniformisation} asserts that every simply connected Riemann surface is conformally equivalent either to the sphere, the plane, or the disc (each of them equipped with its standard metric).
 
In contrast, the class of conformal mappings in higher dimensions is very limited.
According to Liouville's theorem \cite{LiouvilleConformal}, conformal mappings of Euclidean domains in dimension~$\ge 3$ can be expressed as a finite number of compositions of translations, homotheties, orthonormal transformations, and inversions.

\begin{example}
  Let $(M',g')$ be the complex plane and $(M,g)$ be the 2-sphere without north pole ${\mathsf{n}}$, regarded as a punctured Riemann sphere.
  Then they are conformally equivalent in the sense of Definition \ref{d:ConformalEquiv} \ref{i:d:ConformalEquiv:2}.
  The conformal map  $\Phi$  is given by the  stereographic projection
(that is, for all $x$ on the sphere $\Phi(x)$ is the stereographic projection of the point $x$), and the weight $\varphi$ is given by
$\varphi(x)=-2\log\big(\sqrt2 \sin\big({d_{\mathbb{S}^2}({\mathsf{n}},x)}/2\big)\big)$.
This example, however, does not fit the setting of this work in two respects:
\begin{enumerate*}[$(1)$]
\item the manifold $M'$ is non-compact, 

\item the weight~$\varphi$ is  non-smooth on the completion of $M$ (it has a singularity at the north pole).
\end{enumerate*} 
\end{example}

\subsection{Co-polyharmonic operators} 
Henceforth, $n$ denotes an \emph{even} number and $(M,g)$ is a \emph{closed} Riemannian manifold of dimension~$n$.

Our interest is primarily in the case $n \ge 4$. The case $n=2$ is widely studied with celebrated, deep and fascinating results. It serves here as a guideline. In this case, most of the following   constructions and results are (essentially) well-known.

The fundamental object for our subsequent considerations are the 
\emph{co-polyharmonic operators} $\mathsf{P}_g$, also
called  \emph{conformally covariant powers of the Laplacian}  or \emph{Graham--Jenne--Mason--Sparling operators  of maximal order} (i.e.~of order $n/2$)  as introduced in \cite{GJMS92}. 
The \emph{co}-polyharmonic operators are \emph{com}panions of the polyharmonic operators $(-\Delta_g)^{n/2}$, coming with \emph{cor}rection terms which make them \emph{con}formally invariant.
The construction of the co-polyharmonic operators~$\mathsf{P}_{g}$ is quite involved.
We outline this construction in Section~\ref{construction-pol}.
Before we get into that, let us first summarize the crucial properties of the operators~$\mathsf{P}_{g}$ that is relevant for the sequel.
We stress that, together with the sign convention~$\Delta_g\leq 0$, our definition~\eqref{eq:IntroPaneitz} implies that~$\Pol_g$ always has non-negative principal part.

\begin{theorem}\label{Pol-basic} For every closed manifold $(M,g)$ of even dimension $n$, 
\begin{enumerate}[$(i)$]
\item\label{i:Pol-basic:1} the co-polyharmonic operator $\mathsf{P}_{g}$ is a 
differential operator of order $n$,
\item\label{i:Pol-basic:2} the leading order is $(-\Delta_g)^{n/2}$, the zeroth order vanishes,
\item\label{i:Pol-basic:3} the coefficients are $\mathcal{C}^\infty$ functions of the curvature tensor and its derivatives,
\item\label{i:Pol-basic:4} it is symmetric and extends to a self-adjoint operator, the Friedrichs extension, denoted by the same symbol, on $L^2(M,\mathsf{vol}_g)$ with domain $\mathcal H^{n}(M,\vol_g)$,
\item\label{i:Pol-basic:5} it is conformally quasi-invariant: if $g'=e^{2\varphi}g$ for some $\varphi\in\mathcal{C}^\infty(M)$, then
\begin{equation}
\mathsf{P}_{g'}=e^{-n\varphi}\mathsf{P}_g.
\end{equation}
\end{enumerate}
More generally, assume that $(\M,\g)$ and $(\M',\g')$ are conformally equivalent with $\C^\infty$-diffeo\-morphism  $\Phi:\M\to\M'$ and  weight $\varphi\in\C^\infty(\M)$ such that $\Phi^*\g'=\ee^{2\varphi} \g$.
Then
\begin{enumerate}[$(i)$, resume]
  \item\label{i:Pol-basic:6} for all $u \in \mathcal{C}^{\infty}(M')$:
\begin{equation}
\left(\Pol_{\M',\g'}u\right)\circ\Phi= e^{-n\varphi}\Pol_{\M,\g}\left(u\circ \Phi\right).
\end{equation}
\end{enumerate}
\end{theorem}

\begin{proof} Most properties are due to 
\cite{GJMS92}, and re-stated in \cite{GZ03};
self-adjointness is proven in \cite[Corollary, p. 91]{GZ03}.
\end{proof}

\begin{remark}
\begin{enumerate}[$(a)$, wide]
\item Some authors work directly with a Laplacian defined as a non-negative operator (for instance, \cite{GZ03}).
Other authors work with the usual Laplacian and consider the operator $\mathsf{P}_g$ with leading term~$\Delta_g^{n/2}$ (for instance, \cite{Bra95,Gov06,Juhl13}); this would correspond to $(-1)^{n/2}\mathsf{P}_g$ in our convention.

\item In general, no closed expressions exist for the operators $\mathsf{P}_g$. 
However, recursive formulas  for the expression of $\mathsf{P}_g$ are known and a priori allow to explicitly compute~$\mathsf{P}_g$ for any even $n$, \cite{Juhl13}.
As the dimension increases, these formulas become more and more involved, as the complexity of lower-order terms grows exponentially with $n$.
\end{enumerate}
 \end{remark}
 
 \begin{proposition}\label{einstein-poly}
 The most prominent cases are:
\begin{enumerate}[$(i)$, wide]
\item\label{i:einstein-poly:1} If $n=2$, then $\mathsf{P}_g=-\Delta_g$.

\item\label{i:einstein-poly:2} If $n=4$, then $\mathsf{P}_g=\Delta_g^2 + \mathsf{div}\left(2\mathsf{Ric}_g-\frac23\mathsf{scal}_g\right)\nabla$ is the celebrated \emph{Paneitz operator}, see~\cite{Paneitz}.
Here  the curvature term $2\mathsf{Ric}_g-\frac{2}{3}\mathsf{scal}_g$ should be viewed as an endomorphism of the
tangent bundle, acting on the gradient of a function.
In coordinates: 
\begin{equation*}
  \mathsf{P}_g u=\sum_{i,j}\nabla_i\left[\nabla^i\nabla^j+2 \mathsf{Ric}^{ij}_g-\frac23\mathsf{scal}_g\cdot g^{ij}\right] \nabla_j u, \qquad \forall u \in \mathcal{C}^{\infty}(M).
\end{equation*}

\item\label{i:einstein-poly:3} If $(M,g)$ is an \emph{Einstein manifold} with $\mathsf{Ric}_g=kg$ (for some $k\in\mathbb{R}$) and even dimension~$n$, then
\begin{equation}
  \mathsf{P}_g=\prod_{j=1}^{n/2}\left[-\Delta_g+\frac{k}{n-1}\nu_j^{(n)}\right]
\end{equation}
with 
$\nu_j^{(n)} \coloneqq  \frac n2\left(\frac n2-1\right)-j(j-1)=\left(\frac{n-1}2\right)^2-\left(\frac{2j-1}2\right)^2$ for $j=1,\ldots,n/2$.
\end{enumerate}
\end{proposition}
\begin{proof}
For \ref{i:einstein-poly:1} \& \ref{i:einstein-poly:2} see \cite{GJMS92} or~\cite[p.~122]{ChaEasOerYan08}; for \ref{i:einstein-poly:3} see \cite[Thm.~1.2]{Gov06}. 
All formulas above appear in these references up to a factor~$(-1)^{n/2}$, due to the sign convention in the definition of~$\Pol_g$.
\end{proof}
\begin{example}
If  $(M,g)$ is flat, then $\mathsf{P}_g=(-\Delta_g)^{n/2}$  is the positive \emph{poly-Laplacian}.
\end{example}
\begin{example}
If $(M,g)$ is the round sphere $\mathbb{S}^n$, then 
$\mathsf{P}_g=\prod_{j=1}^{n/2}\tbraket{-\Delta_g+\nu_j^{(n)}}$
with~$\nu_j^{(n)}$ as above (this formula already appears in \cite{Bra95}).
In particular,
$\mathsf{P}_g= \Delta_g^2-2\Delta_g$ in the case $n=4$, and $\mathsf{P}_g=-\Delta_g^3+10\Delta_g^2-24\Delta_g$ in the case $n=6$.
\end{example}

Conformally invariant operators with leading term a power of the Laplacian $\Delta_g$ have been a focus in mathematics and physics for decades.
For instance, Dirac \cite{Dirac36} constructs a conformally invariant wave operator on a four-dimensional surface in the five-dimensional projective plane in order to show that Maxwell equations are conformally invariant in a curved space-time (Lorentzian manifolds).
In the case of Riemannian manifolds the Yamabe operator
\begin{equation*}
  \Delta_g -  \frac{n-2}{4(n-1)} \mathsf{scal}_g \comma
\end{equation*}
encodes the behaviour of the Ricci curvature under conformal change and has proved of uttermost importance in the resolution of the Yamabe problem on compact Riemannian manifolds \cite{Yamabe60,Trudinger68,Aubin76,Schoen84}.
\cite{Paneitz} constructs a conformally invariant operator with leading term $\Delta^{2}_{g}$, and sixth-order analogues are constructed in \cite{Bra85,Wun86}.

\subsection{Construction of the co-polyharmonic operators}\label{construction-pol}
Now let us outline the construction of the  co-polyharmonic operators, introduced by C.R.~Graham, R.~Jenne, L.J.~Mason, and G.A.J.~Sparling \cite{GJMS92}.
They base their original construction on the \emph{ambient metric}, introduced by C.~Fefferman and C.R.~Graham \cite{FeffermanGraham85}, a Lorentzian metric on a suitable  manifold of dimension $n+2$.
In an alternative approach, proposed by C.R.~Graham and M.~Zworski \cite{GZ03}, the manifold~$(M,g)$ is regarded as the \emph{boundary at infinity} of 
an asymptotically hyperbolic Einstein manifold~$(N,h)$ of dimension~$n+1$.
Our presentation follows \cite{GZ03}, focussing on manifolds of even dimension $n$ and on operators with maximal degree $k=n/2$.

\subsubsection{The Poincar\'e metric associated to $(M,[g])$}
Consider a closed Riemannian manifold $(M,g)$ with conformal class $[g]$ and even dimension $n$.
Choose an $n+1$-dimensional Riemannian manifold $(N,h)$ with boundary such that $\partial N=M$, for instance, $N=[0,\infty)\times M$.
A Riemannian metric $h$ on $N$ is called \emph{conformally compact metric with conformal infinity $[g]$} if
\begin{equation}
{h}=\frac{\overline{h}}{x^2}, \qquad \overline{h}|_{T\partial{N}}\in [g]
\end{equation}
where $\overline{h}$ is a smooth metric on $N$, and $x:{N}\to\mathbb{R}_+$ is a smooth function such that
$\{x=0\}=\partial{N}$ and $dx|_{\partial{N}}\not=0$.
We say that the metric $h$ is \emph{asymptotically even} if it is given as
\begin{equation*}
{h}=\frac1{x^2}\braket{dx^2+\sum_{i,j=1}^n \overline h_{ij}(x,\xi)d\xi^id\xi^j}
\end{equation*}
where $x$ is as above, $(\xi^1,  \ldots,\xi^n)$ forms a coordinate system on $M$, and  $\overline h_{ij}$ for $1\le i,j\le n$ is an even function of $x$. 

\begin{definition} A  \emph{Poincar\'e metric associated to $[g]$} is a conformally compact metric~$h$ with conformal infinity $[g]$ which   is asymptotically even and satisfies
\begin{equation}
\mathsf{Ric}_g+ng=\mathcal{O}(x^{n-2}),\qquad {{\mathsf{tr}}}_g\big(\mathsf{Ric}_g+ng\big)=\mathcal{O}(x^{n+2}).
\end{equation}
\end{definition}

\begin{lemma}[{\cite[Theorem 2.3]{FeffermanGraham85}}] For every closed  $(M,[g])$ of even dimension $n$, there exists a  Poincar\'e metric $h$.
It is
uniquely determined up to addition of terms vanishing to order $n-2$ and up to a diffeomorphism
fixing $M$.
\end{lemma}

The prime example for this construction is provided by the Poincar\'e model of hyperbolic space: 
the $n$-dimensional round sphere $M=\mathbb{S}^n$ is the boundary of the  unit ball $N=B_1(0)\subset\mathbb{R}^{n+1}$ equipped with the hyperbolic metric $dh(r,\xi)=\frac{4}{(1-r^2)^2}\big[dr^2+dg(\xi)\big]$.

\subsubsection{The generalized Poisson operator on $({N},{h})$}
 Consider a closed  $(M,[g])$ of even dimension $n$ and let $h$ be the Poincar\'e metric associated to it on a suitable $N$ with $\partial N=M$.
Extending the traditional Landau notation, for a function $v$ on $N$ we say that  $v=\mathcal{O}(x^{\infty})$ if $v=\mathcal{O}(x^{a})$ as $x\to0$ for every~$a\in\N$.
 
\begin{lemma}[{\cite[Props.~4.2, 4.3]{GZ03}}] 
\begin{enumerate}[$(i)$, wide]
\item For every $f\in\mathcal{C}^\infty(M)$ there exists a solution to
\begin{equation}
\Delta_{h}u={\mathcal O}(x^\infty)
\end{equation}
of the form 
\begin{equation}
u=F+G\, x^n\log x\comma \qquad F,G\in \mathcal{C}^\infty(N)\comma \qquad F|_M=f \fstop
\end{equation}
Here $F$ is uniquely determined ${\mathrm{mod}}\, \mathcal{O}(x^n)$ and $G$ is uniquely determined ${\mathrm{mod}}\, \mathcal{O}(x^\infty)$.

\item Put $\sigma_n\coloneqq (-1)^{n/2}\, 2^n (n/2)!\, (n/2-1)!$ and
\begin{equation}
\mathsf{P}_g f\coloneqq -2\sigma_n\, G|_M.
\end{equation}
Then $\mathsf{P}_g$ is a differential operator on $M$ with principal part $(-\Delta_g)^{n/2}$.
It only depends on~$g$ and defines a conformally invariant operator which agrees with the operator constructed in \cite{GJMS92}.
\end{enumerate}
\end{lemma}

\noindent(No sign adjustment is required in comparison with~\cite{GZ03} since the convention there is that the Laplace--Beltrami operator is non-negative.)

\begin{remark}\label{scattering} Given $(M,[g])$ and $(N,h)$ as above, the co-polyharmonic operator $\mathsf{P}_g$ can alternatively be defined as residue at $s=n$ of the meromorphic family of scattering matrix operators $S(s)$, $s\in\mathbb C$, on $(N,h)$,
\begin{equation}\mathsf{P}_g=-2\sigma_n\, \text{Res}_{s=n}S(s)
\end{equation}
with $\sigma_n$ as above, 
\cite[Thm.~1]{GZ03}.
\end{remark}

\subsection{Branson's \texorpdfstring{$Q$}{Q}-curvature}
Co-polyharmonic operators are closely related to \emph{Branson's $Q$-curvature} in even dimension, another important notion in conformal geometry.

The notion of $Q$-curvature was introduced on arbitrary even-dimensional manifolds by T.~Branson \cite[p.~11]{Bra85}. Its construction and properties have since been studied by many authors.
In dimension~4, explicit computations for the $Q$-curvature are due to T.~Branson and B.~{\O}rsted \cite{BransonOrsted91}.
Its properties are very much akin to those of scalar curvature in 2 dimensions.
C.~Fefferman and K.~Hirachi \cite{FeffermanHirachi} presented an approach based on the ambient Lorentzian metric of \cite{FeffermanGraham85}.
The definition of $Q$-curvature may differ in the literature up to a sign or a factor~$2$.
Following \cite{GZ03}, with the notation of Remark~\ref{scattering}, we have
$Q_g=-2\sigma_n\, S(n) \mathbf{1}.$

The crucial property of $Q$-curvature is its behavior under conformal transformations.
\begin{proposition}[{\cite[Corollary 1.4]{Bra95}}]
	If $g'=e^{2\varphi}g$ then
	\begin{equation}\label{q-transform}
	e^{n\varphi}{Q}_{g'}={Q}_g+\mathsf{P}_g\varphi.\end{equation}
\end{proposition}
Our  sign convention for $Q_g$ comes from our sign convention for $\Pol_g$ together with the validity of  equation \eqref{q-transform}.
\begin{corollary}
  The \emph{total $Q$-curvature}
	$Q(M,g)\coloneqq \int_M Q_g d\mathsf{vol}_g$ is a conformal invariant. 
\end{corollary}

\begin{proof}
	By the previous Proposition,
	\begin{align*}{Q}(M, g')=\int_M {Q}_{g'}\,e^{n\varphi} d\mathsf{vol}_g
	&={Q}(M,g)+\int_M \mathsf{P}_g\varphi\, d\mathsf{vol}_g\\
	&={Q}(M,g)+\int_M \varphi\,\mathsf{P}_g \mathbf{1}\, d\mathsf{vol}_g={Q}(M,g)
	\end{align*}
	due to the self-adjointness of ${\mathsf{P}}_g$ and the fact that it annihilates constants.
\end{proof}

Again, explicit formulas are only known in low dimensions or for Einstein manifolds.

\begin{example}\label{ex: q-curvature} Important cases  are
\begin{enumerate}[$(i)$]
\item If $n=2$, then ${Q}_g=\frac12\mathsf{R}_g=\frac12\mathsf{scal}_g$ is half of the \emph{scalar curvature}, see e.g.~\cite[Eqn.~3.1, up to a factor~$-1=(-1)^{n/2}$]{ChaEasOerYan08}.

\item\label{i:ex: q-curvature:1} If $n=4$ then ${Q}_g=-\frac16\Delta_g \mathsf{scal}_g-\frac12|\mathsf{Ric}_g|^2+\frac16\mathsf{scal}_g^2$
with $\abs{\mathsf{Ric}_g}^2=\sum_{i,j}\mathsf{Ric}^{ij}\mathsf{Ric}_{ij}$, see \cite{BransonOrsted91}.

\item If $(M,g)$ is an Einstein manifold  with $\mathsf{Ric}_g=k\,g$ and even dimension, then \cite[Thm.~1.1, up to a factor~$(-1)^{n/2}$]{Gov06}
\begin{equation}
Q_g=(n-1)! \,\Big(\frac{k}{n-1}\Big)^{n/2}.
\end{equation}
In particular, for
the round sphere,~$Q_g=(n-1)!$. For instance, if~$n=4$, then~$Q=6$. 
\end{enumerate}
\end{example}

Recall that a Riemannian manifold is called \emph{conformally flat} if it is conformally equivalent to a flat manifold.

\begin{proposition} Let $\chi(M)$ denote the Euler characteristic of $(M,g)$.
\begin{enumerate}[$(i)$]

 \item In the case $n=2$, 
 \[
 {Q}(M,g)=2\pi\,\chi(M) \fstop
 \]

\item In the case $n=4$, 
\begin{equation*}
Q(M,g)=8\pi^2\chi(M)-\frac14\int_M |W|^2d\mathsf{vol}_g\comma
\end{equation*}
where~$W$ is the Weyl tensor, and $|W|^2=\sum_{a,b,c,d}W^{abcd}
W_{abcd}$. 
In particular, 
\begin{equation*}
{Q}(M,g)=8\pi^2\chi(M) \quad\Longleftrightarrow\quad (M,g)\text{ is conformally flat}.
\end{equation*}
\item For any even $n$, if $(M,g)$ is conformally flat, then with $c_n=\frac12(2n-1)! (4\pi)^{n/2}$,
\begin{equation*}
{Q}(M,g)=c_n \,\chi(M) \fstop
\end{equation*} 
\end{enumerate}
\end{proposition}

\begin{proof}
The two-dimensional claim follows from Example \ref{ex: q-curvature} and the Gauss--Bonnet Theorem. 
See \cite[p.~673]{BransonOrsted91} for the case of dimension four, and \cite[p.~3]{GZ03} for the case of conformally flat manifolds in even dimension.
Alternatively, see~\cite[pp.~122f., up to a factor~$(-1)^{n/2}$]{ChaEasOerYan08}.
\end{proof}

\begin{remark} Recall that for 2-dimensional oriented Riemannian manifolds, $\chi(M)=2-2 \mathfrak{g}$ where $\mathfrak{g}$ denotes the genus of $M$.
Furthermore, for the sphere in even dimension, $\chi(\mathbb{S}^n)=2$.
\end{remark}

\subsubsection{Some rigidity and equilibration results in $n=4$}

In dimension 4, the conformal invariant integral of the $Q$-curvature leads to remarkable rigidity and equilibration results, resembling famous analogous results in dimension~2. 
To formulate them, let us introduce another important conformal invariant, the
 Yamabe constant
\[
Y(M,g)\coloneqq \inf_{h\in[g]}\frac{\int_M\mathsf{scal}_{h}d\mathsf{vol}_{h}}{\sqrt{\vol_{h}(M)}}\fstop
\]

\begin{proposition}[{\cite{Gursky99},\cite{ChangYang95}}]
Assume $n=4$.
\begin{enumerate}[$(i)$]
\item
There exist closed hyperbolic manifolds with $Y(M,g)<0$ and $Q(M,g)> 16\pi^2$.
\item
 If $Y(M,g) \ge 0$, then $Q(M,g)\le 16\pi^2$ with equality if and only if $M={\mathbb{S}}^4$.

\item
 If $Y(M,g) \ge 0$ and $Q(M,g) \ge 0$, then $\mathsf{P}_g \ge 0$ and $\mathsf{P}_g u=0\Longleftrightarrow u$ is constant.
 
 \item If $Q(M,g)\le 16\pi^2$ and $\mathsf{P}_g \ge 0$ with $\mathsf{P}_g u=0\Longleftrightarrow u$ is constant, then there exists a conformal metric $g'$  with constant $Q$-curvature.
 \end{enumerate}
\end{proposition}

\begin{proposition}[{\cite[Theorem 4.1]{MalchiodiStruwe06}}]
For any $g_0=e^{2 \varphi_0}g$ on $M=\mathbb{S}^4$, the $Q$-curvature flow 
\[
\frac\partial{\partial t}g_t=-2(Q_{g_t}-\bar Q_{g_t})g_t
\]
(with $\bar Q_{g_t}\coloneqq  \langle Q_{g_t}\rangle_{g_t}$ the mean value of $Q_{g_t}$ on~$(\mathbb{S}^4, g_t)$) converges exponentially fast to a metric $g_\infty=e^{2\varphi_\infty}g$ of constant $Q$-curvature  6 in the sense that $\|\varphi_t-\varphi_\infty\|_{H^{4}}\le C\, e^{-\delta t}$ for some constants $C$ and $\delta>0$.
\end{proposition}

\section{Admissibility, Sobolev spaces,  and kernel estimates}
\subsection{Admissible manifolds}
\begin{definition} We say that a Riemannian manifold $(M,g)$ is \emph{admissible} if it is closed and of even dimension, and if the  co-polyharmonic operator $\mathsf{P}_g$ is positive definite on $\mathring L^2(M,\mathsf{vol}_g)$.\end{definition}

As an immediate consequence of Theorem \ref{Pol-basic} \ref{i:Pol-basic:5} we obtain

\begin{corollary}\label{t:admissible-invariant}
  Admissibility of a Riemannian manifold $(M,g)$ is a conformal invariance, or in other words, it is a property of the conformal class $(M,[g])$. 
\end{corollary}

More generally, admissibility of $(\M,\g)$ implies admissibility of any $(N,h)$ conformally equivalent to $(M,g)$ in the sense of Definition~\ref{d:ConformalEquiv} \ref{i:d:ConformalEquiv:2}.

\begin{example} Every closed 2-dimensional manifold  is admissible.
 \end{example}
 
Having at hand the explicit representation formula for the co-polyharmonic operators on Einstein manifolds from
Lemma \ref{einstein-poly}, we easily conclude

  \begin{proposition}\label{t:admissible-einstein}
  Every closed even-dimensional Einstein manifold with nonnegative Ricci curvature is admissible.
\end{proposition}

More generally, we obtain:

  \begin{proposition}\label{t:admissible-hyperbolic}
  A closed  Einstein manifold of even dimension $n$ and of Ricci curvature $-(n-1)\kappa$ is admissible if and only if $\lambda_1>\frac{n(n-2)}4 \kappa$.
\end{proposition}

\begin{proof} Since $M$ has constant Ricci curvature $-(n-1)\kappa$,  according to Proposition \ref{einstein-poly},
  $\mathsf{P}_g=\prod_{j=1}^{n/2}\left[-\Delta_g-\kappa\nu_j^{(n)}\right]$ with $\nu_j$ ranging between 0 and $\frac n2(\frac n2-1)$.
Thus $\Pol_g>0$ on~$\mathring L^2(M,\mathsf{vol}_g)$ if and only if $\lambda_1> \frac n2(\frac n2-1)\kappa$. 
\end{proof}

\begin{remark} %
\begin{enumerate}[$(a)$, wide]
\item
The number $\frac{n(n-2)}4$ is strictly smaller than 
$\frac{(n-1)^2}4$ which plays a prominent role as threshold for the spectral gap of hyperbolic manifolds (and which is also  
the spectral bound for the simply connected hyperbolic space). Many results in hyperbolic geometry deal with the question whether $\lambda_1$ is close to $\frac{(n-1)^2}4$.

\item The \emph{Elstrodt--Patterson--Sullivan Theorem}, \cite[Thm.\ (2.17)]{Sul87},  provides a lower bound for~$\lambda_1$ for a hyperbolic manifold $M={\mathbb H}^n/\Gamma$ in terms of the \emph{critical exponent} $\delta(\Gamma)$ of the Kleinian group $\Gamma$ acting on the simply connected hyperbolic space ${\mathbb H}^n$ of dimension $n$ and curvature~$-1$. More precisely,
\begin{equation}
\lambda_1>\frac{n(n-2)}4\quad\text{if (and only if) }\delta(\Gamma)<\frac{n}2\comma
\end{equation}
and, moreover, $\lambda_1=\frac{(n-1)^2}4$ if (and only if) even $\delta(\Gamma)\le\frac{n-1}2$.
Here $\delta(\Gamma)$ denotes the infimal value for which the Poincar\'e series for $\Gamma$ converges, that is,
\[
\delta(\Gamma)\coloneqq \inf\Big\{s\in\R: \ \sum_{\gamma\in\Gamma}\exp\Big(-s \,d(x,\gamma y)\Big)<\infty\Big\}\comma
\]
the latter being independent of the choice of $x,y\in M$.

\item Similar estimates for $\lambda_1$ exist in terms of the Hausdorff dimension $D$ of the \emph{limit set of~$\Gamma$} provided $\Gamma$ is \emph{geometrically finite without cusps}, see \cite[Thm.~(2.21)]{Sul87}. More precisely,
\begin{equation}
\lambda_1>\frac{n(n-2)}4\quad\text{if (and only if) }D<\frac{n}2\comma
\end{equation}
and, moreover, $\lambda_1=\frac{(n-1)^2}4$ if (and only if) even $D\le\frac{n-1}2$.
\end{enumerate}
\end{remark}

\begin{proposition} \label{non-admiss}
For every even dimension $n\ge4$, there exist  closed Einstein manifolds that are not admissible. They can be constructed, for instance, as
$M=M_1\times M_2$ where $M_1$ denotes any closed manifold of dimension $n-2$ and of constant curvature $-\frac1{n-3}$, and where $M_2$ denotes any closed hyperbolic Riemannian surface with $\lambda_1(M_2)\le2/3$. 
\end{proposition}
\begin{remark}
According to~\cite[Satz 1]{Bus77}, for every $\varepsilon>0$ there exist  closed hyperbolic Riemannian surfaces with genus 2 and $\lambda_1<\varepsilon$. 
\end{remark}

\begin{proof}[Proof of Proposition~\ref{non-admiss}]
By construction, $M$ is an Einstein manifold with constant Ricci curvature $-g$. Thus by Proposition~\ref{t:admissible-hyperbolic}, $M$ is admissible if and only if $\lambda_1(M)>\frac{n(n-2)}{4(n-1)}\ge \frac23$.
On the other hand, by construction $\lambda_1(M)\le\lambda_1(M_2)\le\frac23$.
\end{proof}

\subsection{Estimates for heat kernels and resolvent kernels}
Our main result in the section, Theorem \ref{est-k-log}, provides a sharp asymptotic estimate for 
the integral kernel  of $\Pol_g^{-1}$ on $\mathring L^2(M,\mathsf{vol}_g)$.
Deriving this 
requires precise estimates on the integral kernel of the operators ${(\alpha-\Delta)}^{-n/2}$ for $\alpha> -\lambda_1$.
These estimates in turn depend on sharp heat kernel estimates, the upper one of which is new.

For a Riemannian manifold~$(M,g)$, denote by~$\mathsf{sec}$ its sectional curvature, by~$\mathsf{inj}$ its injectivity radius, and by~$p_t$ its heat kernel, i.e.\ the integral kernel of the heat semigroup~$P_t\coloneqq \ee^{t\Delta}$.

\begin{proposition}\label{pt-estimates} Let $(M,g)$ be a closed $n$-dimensional manifold.
\begin{enumerate}[$(i)$]
\item\label{i:p:pt-estimates:1}
  Assume that $\mathsf{Ric}_g \ge -(n-1)a^2\,g$ and set $\lambda_*\coloneqq \frac{(n-1)^2}4a^2$ if $n\not=2$ and $\lambda_*=\frac{1}6a^2$ if $n=2$.
    Then for all $t>0$ and all $x,y\in M$,
\begin{equation}\label{lower-heat}
p_t(x,y) \ge \frac1{(4\pi t)^{n/2}}
\bigg(\frac{a\,d(x,y)}{\sinh(a\,d(x,y))}\bigg)^{\frac{n-1}2}
\ee^{-\frac{d^2(x,y)}{4t}}
\ee^{-\lambda_*t}.
\end{equation}
\item\label{i:p:pt-estimates:2} Let a ball $B=B_R(x)\subset M$ be given,  assume that $\mathsf{sec}\le b^2$ on $B$ and that $\mathsf{inj}_x \ge R$,
  and let $p^0_t$ denote the heat kernel  on $B$ with Dirichlet boundary conditions. Moreover, 
 \begin{itemize}
 \item
 in the case $n\not=2$, assume that $R\le\frac\pi b$, and set $\lambda^*\coloneqq \frac{n(n-1)}6b^2$, 
 \item in the case $n=2$, assume  that $R\le\frac\pi{2b}$, and set $\lambda^*\coloneqq \frac12\,b^2$.
 \end{itemize}
 
Then for all $t>0$ and all $y\in B$,
\begin{equation}\label{upper-heat}
p^0_t(x,y)\le \frac1{(4\pi t)^{n/2}}
\bigg(\frac{b\,d(x,y)}{\sin(b\,d(x,y))}\bigg)^{\frac{n-1}2}
\ee^{-\frac{d^2(x,y)}{4t}}
\ee^{+\lambda^*t}.
\end{equation}
\end{enumerate}
\end{proposition}
\begin{proof} \ref{i:p:pt-estimates:1} follows  from \cite[Cor.~4.2 and Rmk.~4.4(a)]{Stu92} (we work with the geometric heat semigroup $\ee^{t \Delta}$ rather than with the probabilistic semigroup $\ee^{t \frac{\Delta}{2}}$ as in \cite{Stu92}).

\ref{i:p:pt-estimates:2} Let $\bar M=\mathbb S^{b,n}$ denote the round sphere of dimension $n$ and radius $1/b$ (which has constant curvature $b^2$), fix a point $\bar x\in\bar M$, and let $\bar B$ denote the ball around $\bar x$ of radius~$R$ in~$\bar M$.
Denote by $\bar p_t^0$ the heat kernel  on $\bar B$ with Dirichlet boundary conditions. By rotational invariance,
\begin{equation*}
\bar p_t^0(\bar x,\bar y)=\bar p_t^0\tparen{\bar d(\bar x,\bar y)}
\end{equation*}
for some function $r\mapsto \bar p_t^0(r)$. According to the celebrated heat kernel comparison theorem of Debiard--Gaveau--Mazet \cite{DGM76}, 
\begin{equation}
p_t^0(x,y)\le \bar p_t^0\tparen{d(x, y)}
\end{equation}
for all $t>0$ and all $y\in B$.

We treat the case $n \ne 2$ first.
Following the strategy for deriving the lower bound~\eqref{lower-heat} in~\cite{Stu92}, define
\begin{equation*}
  \hat p_t^0(r)\coloneqq \frac1{(4\pi t)^{n/2}}
{\left(\frac{b\,r}{\sin(b\,r)}\right)}^{\frac{n-1}{2}}
\ee^{-\frac{r^2}{4t}}\ee^{\lambda^*t} = g_{t}(r)  {\left(\frac{br}{\sin(br)}\right)}^{\frac{n-1}{2}} \ee^{\lambda^{*}t}.
\end{equation*}
where $\lambda^*$ as defined above and $g_{t}(r) = {(4\pi t)}^{-n/2} \ee^{-r^{2}/4t}$ is the Gaussian kernel.
We show that the function $(t,\bar y)\mapsto H(t,\bar y)\coloneqq  \hat p_t^0(\bar d(\bar x,\bar y))$ is space-time super-harmonic on $(0,\infty)\times\bar B$.
Indeed, a direct computation yields:
\begin{align*}
  & \partial_{t} \log \hat{p}^{0} = \partial_{t} \log g + \lambda^{*}; \\
  & \partial_{r} \log \hat{p}^{0} = \partial_{r} \log g + \frac{n-1}{2} \left(\frac{1}{r} - b \frac{\cos br}{\sin br}\right); \\
& \partial_{rr}^{2} \log \hat{p}^{0} = \partial_{rr}^{2} \log g + \frac{n-1}{2} \left(\frac{b^{2}}{\sin^{2} br} - \frac{1}{r^{2}}\right).
\end{align*}
Now using that $H$ is a radial function and the chain rule we find that
\begin{equation}\label{H1}
  \frac{1}{H} (\partial_{t} - \bar{\Delta}) H = \partial_{t} \log \hat{p}^{0} - (n-1) b\, \frac{\cos br}{\sin br} \partial_{r} \log \hat{p}^{0} - \partial_{rr}^{2} \log \hat{p}^{0} - {\left(\partial_{r} \log \hat{p}^{0}\right)}^{2} \comma
\end{equation}
where the left-hand side is evaluated at $(t, \bar{y})$ and the right-hand side at $(t, \bar{d}(\bar{x}, \bar{y}))$.
We easily verify that $g$ satisfies:
\begin{equation*}
  \partial_{t} \log g - \partial_{rr}^{2} \log g - {(\partial_{r} \log g)}^{2} - \frac{n-1}{r} \partial_{r} \log g = 0 \fstop
\end{equation*}
We thus see that in \eqref{H1} all the appearances of $\log g$ cancel out, and we get:
\begin{equation*}
  \begin{split}
    \frac{1}{H} (\partial_{t} - \bar{\Delta}) H &=  \lambda^{*} - \frac{n-1}{2} \left(\frac{b^{2}}{\sin^{2} br} - \frac{1}{r^{2}}\right) \\
    &\quad - \frac{{(n-1)}^{2}}{2} b \frac{\cos br}{\sin br} \left(\frac{1}{r} - b \frac{\cos br}{\sin br}\right) 
    - \frac{{(n-1)}^{2}}{4} {\left(\frac{1}{r} - b \frac{\cos br}{\sin br}\right)}^{2} \\
                                                &=  \lambda^{*} - \frac{(n-1)(n-3)}{2} \frac{1}{r^{2}} + b^{2} \frac{\cos^{2} br}{\sin^{2} br} \frac{{(n-1)}^{2}}{4} - \frac{n-1}{2} b^{2} \frac{1}{\sin^{2} br} \\
                                                &=  \lambda^{*} + \frac{(n-1)(n-3)}{4} \left[ \frac{b^{2}}{\sin^{2} br} - \frac{1}{r^{2}}  \right] - b^{2} \frac{{(n-1)}^{2}}{4} \\
                                                & \ge  \lambda^{*} + \frac{(n-1)(n-3)}{4} \, \frac{b^{2}}{3} - b^{2} \frac{{(n-1)}^{2}}{4}                                          = 0 \fstop
  \end{split}
\end{equation*}
On the other hand, $\bar{p}$ is harmonic and by a comparison principle for solution of parabolic equations we thus have that $p^{0} \leq H$.
In order to properly justify the comparison principle, instead of working with~$p^{0}$ and~$H$ that have singular initial condition, we work instead with~$p^{0}_{R'}$ the solution to the heat equation with initial condition $\mathbf{1}_{B_{R'}(y)}$ and $H_{R'}$ which has the same expression as $H$ except that we choose
\begin{equation*}
  g_{R'}(t,r) = \int_{B^{\mathbb{R}^{n}}_{R'}(r)} {(4\pi t)}^{-\frac{n}{2}} \ee^{-{|y|}^{2}{4t}} \,dy \comma
\end{equation*}
instead of $g$.
The same computation yields that $H_{R'}$ is a super-solution to the heat equation.
Then we argue as in \cite{Stu92}.

Now, for the case $n = 2$, defining $\hat{p}^{0}$ and $H$ as above we still find that
\begin{equation*}
  \begin{split}
    \frac{1}{H}(\partial_{t} - \bar{\Delta}) H &=  \lambda^{*} - \frac{1}{4} \left[\frac{b^{2}}{\sin^{2} br} - \frac{1}{r^{2}} + b^{2} \right]
    \\
                                               & \ge  \lambda^{*} - b^2\left[\frac12-\frac{1}{\pi^{2}}\right]  \ge 0 \comma
  \end{split}
\end{equation*}
where we used that $r <\frac \pi {2b}$.
The rest of the proof is similar.
\end{proof}

Before stating our main estimates, let us introduce some notation and provide some auxiliary results.

\begin{lemma}\label{GreenEstimates}
\begin{enumerate}[$(i)$, wide]
\item\label{i:l:GreenEstimates:1} For every $\alpha>0$ and $s>0$, the \emph{resolvent operator}
$ {\mathsf{G}}_{s,\alpha}\coloneqq \left(\alpha-\Delta\right)^{-s}$ on $L^2=L^2(M,\mathsf{vol}_g)$ 
is an integral operator with kernel given by
\begin{align*} G_{s,\alpha}(x,y)\coloneqq \frac1{\Gamma(s)}\int_0^\infty \ee^{-\alpha t} \, t^{s-1}\, p_t(x,y)\,dt.
\end{align*}
Since 
$\langle P_tu\rangle_g=\langle u\rangle_g$ for all $u$,
the heat operator $P_t=\ee^{t\Delta}$ also acts on the \emph{grounded $L^2$-space} $\mathring L^2=\{u\in L^2(M,\mathsf{vol}_g): \, \langle u\rangle_g=0\}$, and so do the resolvent operators ${\mathsf{G}}_{s,\alpha}$.

\item\label{i:l:GreenEstimates:2} Restricted to $\mathring L^2$,  the \emph{resolvent operator}
\begin{equation*}
\mathring{\mathsf{G}}_{s,\alpha}={(\alpha-\Delta)}^{-s}\big\rvert_{\mathring L^2}
\end{equation*}
is a compact, symmetric operator for every $\alpha>-\lambda_1$ and $s>0$.
It admits a symmetric integral kernel \begin{align*} \mathring G_{s,\alpha}(x,y)\coloneqq \frac1{\Gamma(s)}\int_0^\infty \ee^{-\alpha t} \, t^{s-1}\, \mathring p_t(x,y)\,dt,
\end{align*}
defined in terms of the {grounded heat kernel}  $\mathring p_t(x,y)\coloneqq p_t(x,y)-\mathsf{vol}_g(M)^{-1}$. 

\item\label{i:l:GreenEstimates:3} 
By compactness of~$M$, the operator~$-\Delta$ has discrete spectrum~$(\lambda_j)_{j\in \mathbb{N}_0}$, counted with multiplicity, and the corresponding eigenfunctions~$(\chi_j)_{j\in \mathbb{N}_0}$ form an orthonormal basis for $L^2_g$.
In terms of these spectral data, the grounded resolvent kernel is the symmetric kernel
\begin{equation}\label{eq:representation-grounded-green-eigenfunctions}
\mathring G_{s,\alpha}(x,y)= \sum_{j=1}^\infty \frac{\chi_j(x)\, \chi_j(y)}{(\alpha+\lambda_j)^s}.
\end{equation}
\item\label{i:l:GreenEstimates:4} For every $s>n/2$ and $\alpha>0$ there exists $C$ such that for all $x,y\in M$
\begin{align}\label{eq:l:GreenEstimates:1}
\big|\mathring G_{s,\alpha}(x,y)\big|\leq C
\comma
\end{align}
and for every $s<n/2$ and $\alpha>0$ there exists $C$ such that for all $x,y\in M$
\begin{align}\label{eq:l:GreenEstimates:2}
 G_{s,\alpha}(x,y)\leq  \frac C{d(x,y)^{n-2s}}
\fstop
\end{align}
\end{enumerate}
\end{lemma}

\begin{proof}
  All of \ref{i:l:GreenEstimates:1}--\ref{i:l:GreenEstimates:3} but \eqref{eq:representation-grounded-green-eigenfunctions} are proven by Strichartz~{\cite[\S4]{Str83}}.
  Regarding \eqref{eq:representation-grounded-green-eigenfunctions}, by the Spectral Theorem, for all $u\in L^2$ and $\alpha>0$ (or $u\in \mathring L^2$ and $\alpha>-\lambda_1$),
\begin{align*}
  {(\alpha-\Delta)}^s u=\frac1{\Gamma(s)}\int_0^\infty \ee^{-\alpha t}\, t^{s-1} \,P_tu\,dt=\sum_{j=0}^\infty \frac{\scalar{u}{\chi_j}_{L^2}}{(\alpha+\lambda_j)^s}\,\chi_j
\end{align*}
with convergence of integral and sum in $L^2$ (or in $\mathring L^2$, resp.).
Thus \eqref{eq:representation-grounded-green-eigenfunctions} readily follows.

\ref{i:l:GreenEstimates:4} Estimate~\eqref{eq:l:GreenEstimates:1} is a consequence of~\cite[Cor.~6.2]{LzDSKopStu20}. %
In order to show~\eqref{eq:l:GreenEstimates:2} fix~$\varepsilon\ll 1$,~$x,y\in\M$ with~$0<r\coloneqq  d_\g(x,y)<\varepsilon$, and set
\begin{equation*}
G_{s,\alpha}(x,y)=\underbrace{\frac1{\Gamma(s)}\int_0^\varepsilon \ee^{-\alpha t} t^{s-1} p_t(x,y)\diff t}_{I_1}+ \underbrace{\frac1{\Gamma(s)}\int_\varepsilon^\infty \ee^{-\alpha t} t^{s-1} p_t(x,y)\diff t}_{I_2} \fstop
\end{equation*}

As a consequence of the upper heat kernel estimate~\cite[Eqn.~(2.4)]{LzDSKopStu20}, there exists a constant~$C=C(g,s,\alpha,\varepsilon)>0$ independent of~$x$,~$y$, and such that~$I_2\leq C$ and
\begin{align*}
I_1 \leq C\, \int_0^\varepsilon e^{-\frac{r^2}{t} } t^{s-n/2-1} \diff t \fstop
\end{align*}
Combining these estimates together,
\begin{align*}
G_{s,\alpha}(x,y) \leq C \paren{1+ \int_0^\varepsilon e^{-\frac{r^2}{t} } t^{s-n/2-1} \diff t} \comma
\end{align*}
and the assertion now follows from the known asymptotic expansion of the Exponential Integral function
\begin{equation*}
E_{s-n/2+1}\tparen{\tfrac{r^2}{\varepsilon}}\coloneqq  \int_0^\varepsilon e^{-\frac{r^2}{t} } t^{s-n/2-1} \diff t \asymp \Gamma(n/2-s)\, r^{2s-n} \quad \text{as}\quad r\to 0\fstop \qedhere
\end{equation*}
\end{proof}

\begin{remark}
 For all~$s>0$, the operators ${\mathsf{G}}_{s,\alpha}$ and $\mathring{\mathsf{G}}_{s,\alpha}$ are powers of~${\mathsf{G}}_{\alpha}\coloneqq {\mathsf{G}}_{1,\alpha}$ and $\mathring{\mathsf{G}}_{\alpha}\coloneqq \mathring{\mathsf{G}}_{1,\alpha}$, that is,
\begin{equation*}
  {\mathsf{G}}_{s,\alpha}=({\mathsf{G}}_{\alpha})^s,\qquad \mathring{\mathsf{G}}_{s,\alpha}=(\mathring{\mathsf{G}}_{\alpha})^s.
\end{equation*}
\end{remark}

\begin{example} Let $(M,g)$ be the 2-dimensional round sphere $\mathbb S^2$.
  Then according to \cite[Thm.~6.12]{LzDSKopStu20}, 
\begin{equation*}
\mathring G_{1,0}(x,y)=-\frac1{4\pi}\paren{1+2\log\sin\frac{d(x,y)}2}\fstop
\end{equation*}
\end{example}

\begin{proposition}\label{prop: greenslog} Let $(M,g)$ be a compact $n$-dimensional manifold and $\alpha>-\lambda_1$. 
  Then for all $x$ and $y \in M$:
\begin{align*}
  & \abs{G_{n/2,\alpha}(x,y)-{a_n}\log\frac1{d(x,y)}} \le C_0; \\
  & \abs{\mathring{G}_{n/2,\alpha}(x,y)-{a_n}\log\frac1{d(x,y)}} \le C_0;
\end{align*}
for some $C_0=C_0(g,\alpha)>0$ and 
\begin{equation}\label{an}
{a_n}\coloneqq \frac{2}{\Gamma(n/2)\, (4\pi)^{n/2}} \fstop
\end{equation}
\end{proposition}
\begin{proof}
  For convenience we split the proof.

  \paragraph{Lower estimate for the ungrounded kernel}
  Take $\lambda_{*}$ as in Proposition \ref{pt-estimates} \ref{i:p:pt-estimates:1} and $\alpha > \lambda_{*}$.
  For  the non-grounded resolvent kernel, the lower heat kernel estimate \eqref{lower-heat} yields, with $x$ and $y \in M$, and $r=d(x,y)$,
\begin{align*}
 G_{n/2,\alpha}(x,y)&=\frac1{\Gamma(\frac n2)}\int_0^\infty \ee^{-\alpha t} p_t(x,y)\, t^{n/2-1}dt
 \\
& \ge \frac1{\Gamma(\frac n2)\, (4\pi)^{n/2}}\,\bigg(\frac{ar}{\sinh(ar)}\bigg)^{\frac{n-1}2} \int_0^\infty \ee^{-(\alpha+\lambda_*) t}\,
\ee^{-\frac{r^2}{4t}}\,\frac{dt}t.
\end{align*}
By~\cite{LzDSKopStu20}, Eqn.~(6.8), for every $\beta>0$:
\begin{equation*}
\int_0^\infty \ee^{-\beta t}\,
\ee^{-\frac{r^2}{4t}}\,\frac{dt}t= 4\pi\cdot  G_{1,\beta}^{\mathbb{R}^2}(r) \ge 2\,
\log \frac{1}{r}-C_\beta \fstop
\end{equation*}
Combining the two previous estimates yields
\begin{equation*}
   G_{n/2,\alpha}(x,y) - a_{n} \log\frac1{d(x,y)}> - C_{\alpha + \lambda^{*}}\comma \qquad  x, y \in M \fstop
\end{equation*}

\paragraph{Upper estimate for the ungrounded kernel with Dirichlet boundary conditions}
Consider the case $\alpha>\lambda^*$, with $\lambda^{*}$ as in Lemma \ref{pt-estimates} \ref{i:p:pt-estimates:2}.
We estimate the contribution of~$p_t^0$ as before, with $x$ and $y \in M$, and $r=d(x,y)$:
\begin{align*}
 G^0_{n/2,\alpha}(x,y)&\coloneqq \frac1{\Gamma(\frac n2)}\int_0^\infty \ee^{-\alpha t}\, p^0_t(x,y)\, t^{n/2-1}dt\\
&\le \frac1{\Gamma(\frac n2)\, (4\pi)^{n/2}}\,\bigg(\frac{ar}{\sin(br)}\bigg)^{\frac{n-1}2} \int_0^\infty \ee^{(-\alpha+\lambda^*) t}\,
\ee^{-\frac{r^2}{4t}}\,\frac{dt}{t}\comma
\end{align*}
and we can use the fact that, by~\cite{LzDSKopStu20}, ibid.:
\begin{equation*}
  \int_0^\infty \ee^{-\beta t}\,
\ee^{-\frac{r^2}{4t}}\,\frac{dt}t= 4\pi\cdot  G_{1,\beta}^{\mathbb{R}^2}(r) \le 2\,
\log \frac{1}{r}+C_\beta\comma \qquad \beta > 0 \fstop
\end{equation*}
The two estimates yield
\begin{align*}
  G^0_{n/2,\alpha}(x,y) - a_{n} \log\frac1{d(x,y)}< C_{\alpha - \lambda_{*}}\comma \qquad x,y \in M \fstop
\end{align*}

\paragraph{Upper estimate for the ungrounded kernel}
We now estimate the remainder $G_{n/2,\alpha} - G_{n/2,\alpha}^{0}$.
Choose $0<\beta<\alpha$.
For every $n \ge 2$ and suitable $C,C' > 0$,
\begin{align*}
0&\le G_{n/2,\alpha}(x,y)-  G^0_{n/2,\alpha}(x,y)\coloneqq \frac1{\Gamma(\frac n2)}\int_0^\infty \ee^{-\alpha t} \Big(p_t(x,y)-p^0_t(x,y)\Big)\,t^{n/2-1}dt
\\
&\le C\int_0^\infty \ee^{-\beta t} \Big(p_t(x,y)-p^0_t(x,y)\Big)\,dt=C\,\Big(G_{1,\beta}- G^0_{1,\beta}\Big)(x,y)
\\
&\le C\sup_{z\in \partial B} G_{1,\beta}(x,z)\le C' \fstop
\end{align*}
Above, the second to last inequality follows from the maximum principle for local solutions to $(-\Delta_g+\beta)u=0$, and the last inequality from the elliptic Harnack inequality for positive local solutions to $(-\Delta_g+\beta)u=0$.

\paragraph{Bounds for the grounded kernel}
The lower and upper bounds for the \emph{grounded} resolvent kernel $\mathring G_\alpha$ for $\alpha>\lambda_*$ then follow from the previous bounds and the fact that
$\mathring p_t(x,y)=p_t(x,y)-\frac1{\mathsf{vol}(M)}$ and
\begin{equation*}
\frac1{\Gamma(\frac n2)}\int_0^\infty \ee^{-\alpha t} t^{n/2-1}dt=\alpha^{-n/2} \fstop
\end{equation*}

\paragraph{Bounds for all $\alpha > - \lambda_{1}$}
In order to show the desired estimates for $G_\alpha$ in the whole range of $\alpha>-\lambda_1$, we use a perturbation argument based on the \emph{resolvent identity}
\begin{align}\label{eq: ab}
\mathring{\mathsf{G}}_{\alpha}=\mathring{\mathsf{G}}_{\beta} +(\beta-\alpha)\,\mathring{\mathsf{G}}_{\beta} \mathring{\mathsf{G}}_{\alpha},
\end{align}
valid for all $\beta>-\lambda_1$ and employed below for $\beta>\lambda_*$. 
By iteration, it follows that
\begin{align*}
\mathring{\mathsf{G}}_{\alpha}=\mathring{\mathsf{G}}_{\beta}\, \paren{\sum_{\ell=0}^\infty \tparen{(\beta-\alpha)\mathring{\mathsf{G}}_{\beta}}^\ell}\fstop
\end{align*}
The series is absolutely converging in $\mathrm{Lin}(\mathring{L}^2,\mathring{L}^2)$, since
\begin{equation*}
\tnorm{(\beta-\alpha)\mathring{\mathsf{G}}_{\beta}}_{\mathring{L}^2,\mathring{L}^2}\leq(\beta-\alpha)/(\beta+\lambda_1)<1.
\end{equation*}
Let $\mathsf{T}=(\beta-\alpha)\sum_{\ell=0}^\infty \tparen{(\beta-\alpha)\mathring{\mathsf{G}}_\beta}^\ell$.
Then
\begin{align*}
\mathring{\mathsf{G}}_{\alpha}^{n/2}=\mathring{\mathsf{G}}_\beta^{n/2}({\mathsf{Id}}+\mathring{\mathsf{G}}_\beta {\mathsf{T}})^{n/2}
=\mathring{\mathsf{G}}_{\beta/2}^{n/2}\left(1 +\sum_{k=1}^{n/2} \binom{n/2}{k}\mathring{\mathsf{G}}_\beta^k {\mathsf{T}}^k\right)
=\mathring{\mathsf{G}}_\beta^{n/2}+\mathring{\mathsf{G}}_\beta^{n/2+1}\tilde{\mathsf{T}},
\end{align*}
where $\tilde {\mathsf{T}}=\sum_{k=0}^{n/2-1} \binom{n/2}{k+1}\mathring{\mathsf{G}}_\beta^k{\mathsf{T}}^{k+1}$. Consequently, since all operators involved commute with each other,
\begin{align*}
\mathring{\mathsf{G}}_{\alpha}^{n/2}-\mathring{\mathsf{G}}_\beta^{n/2}=\underbrace{\mathring{\mathsf{G}}_\beta^{n/4+1/2}}_{\mathring{L}^2\to \mathring{L}^\infty} \underbrace{\vphantom{\mathring{\mathsf{G}}_\beta^{n/4+1/2}} \tilde {\mathsf{T}}}_{\mathring{L}^2\to \mathring{L}^2} \underbrace{\mathring{\mathsf{G}}_\beta^{n/4+1/2}}_{\mathring{L}^1\to \mathring{L}^2}.
\end{align*} 
Moreover, $\mathring{\mathsf{G}}_\beta^{n/4+1/2}$ is a bounded linear operator both from $\mathring L^1$ to $\mathring L^2$ and from~$\mathring L^2$ to~$\mathring L^\infty$. 
Indeed, for $u\in \mathring L^1$,
\begin{align*}
&\norm{\mathring{\mathsf G}^{n/4+1/2}_{\beta} u}_{L^2}^2=
\\
&\quad=\int\braket{\int \mathring G_{(n+2)/4,\beta}(x,y)\, u(y)\, d\mathsf{vol}(y)  \int \mathring G_{(n+2)/4,\beta}(x,z)\, u(z)\, d\mathsf{vol}(z)}d\mathsf{vol}(x)
\\
&\quad=\iint \mathring G_{(n+2)/2,\beta}(y,z)\, u(y)\, u(z)\, d\mathsf{vol}(y)d\mathsf{vol}(z)
\\
&\quad \leq \sup_{y,z} \mathring G_{(n+2)/2,\beta}(y,z)\cdot \norm{u}_{L^1}^2\comma
\end{align*}
and for $u\in\mathring L^2,$
\begin{align*}\big\|\mathring{\mathsf{G}}^{n/4+1/2}_{\beta} u\big\|_{L^\infty}^2=&\ \sup_x \bigg(\int \mathring G_{(n+2)/4,\beta}(x,y)u(y)d\mathsf{vol}(y)\bigg)^2
\\
\leq&\ \sup_x \int \mathring G_{(n+2)/4,\beta}^2(x,y)d\mathsf{vol}(y)\cdot \int u^2(y)\,d\mathsf{vol}(y)
\\
=&\ \sup_x \mathring G_{(n+2)/2,\beta}(x,x)
\cdot \|u\|_{L^2}^2\,.
\end{align*}
Finiteness of both expressions is granted for $\beta>0$ by Lemma \ref{GreenEstimates}. Thus summarizing we obtain
\begin{equation*}
\norm{\mathring{\mathsf{G}}_\alpha^{n/2} - \mathring{\mathsf{G}}^{n/2}_{\beta}}_{\mathring L^1,\mathring L^\infty}<\infty \fstop
\end{equation*}
Furthermore, consider
$\tparen{\mathring{\mathsf{G}}_{\alpha}^{n/2}-\mathring{\mathsf{G}}_\beta^{n/2}}\circ \pi_g \colon L^1\to L^\infty$, where as usual $\pi_g \colon u\mapsto u-\langle u\rangle_g $.
Since $\pi_g$ is the identity on $\mathring L^1$, by virtue of~\cite[Thm.~2.2.5]{DunPet40}, the operator
$\tparen{\mathring{\mathsf{G}}_{\alpha}^{n/2}-\mathring{\mathsf{G}}_\beta^{n/2}}\colon \mathring L^1\to L^\infty$ admits a bounded integral kernel. 
Therefore, $\mathring{\mathsf{G}}_{\alpha}^{n/2}$ admits an integral kernel with the same logarithmic divergence as $\mathring{\mathsf{G}}_\beta^{n/2}$.
\end{proof}

For completeness, we provide an estimate for the co-polyharmonic heat kernel which, in the case~$n=2$, reduces to the standard Gaussian estimate.
\begin{remark}[{\cite[Theorem 1.1]{Elst97}}]
Assume that the compact manifold $(M,g)$ of even dimension $n$ is a Lie group. 
Then the co-polyharmonic  heat semigroup $e^{-t\,\mathsf{P}_g}$ has an integral kernel  (co-polyharmonic  heat kernel'), the modulus of which
can be estimated by
\begin{equation*}
\abs{p_t(x,y)}\le \frac{C_1}{t\wedge 1} \,\exp\braket{ -\left(\frac{d(x,y)^n}{C_2\,t}\right)^\frac1{n-1}} \fstop
\end{equation*}
\end{remark}

\subsection{Sobolev spaces and pairings}

For  $s \in \mathbb{R}_+$, we define the usual \emph{Sobolev spaces}  ${\mathcal H}^{s}_g\coloneqq  {(1-\Delta_g)}^{-\frac{s}{2}} L^2_g$ with norm
$\|u\|_{{\mathcal H}^s_g}\coloneqq \| (1-\Delta_g)^{\frac s2}u\|_{L^2_g}$, and ${\mathcal H}^{-s}_g$ as the completion of $L^2_g$ w.r.t.~the norm
$\|u\|_{{\mathcal H}^s_g}\coloneqq \| (1-\Delta_g)^{-\frac s2}u\|_{L^2_g}$
such that formally ${\mathcal H}^{-s}_g\coloneqq  {(1-\Delta_g)}^{\frac{s}{2}} L^2_g$.

For our purpose, however, it is more convenient to use slightly different Hilbert spaces defined in terms of the \emph{normalized co-polyharmonic operator}
\begin{equation}
\mathsf{p}_g\coloneqq a_n\,\mathsf{P}_g
\end{equation}
with $a_n$ as defined in  \eqref{an}, i.e.~$
{a_n}\coloneqq \frac{2}{\Gamma(n/2)\, (4\pi)^{n/2}}.$

We put
$$H^{s}_g\coloneqq  {(1+\mathsf p_g)}^{-\frac{s}{n}} L^2_g, \qquad
\|u\|_{H^s_g}\coloneqq \| (1+\mathsf p_g)^{\frac s{n}}u\|_{L^2_g}$$
and define $H^{-s}_g$ as the completion of $L^2_g$ w.r.t.~the norm
$\|u\|_{H^{-s}_g}\coloneqq \| (1+\mathsf p_g)^{-\frac s{n}}u\|_{L^2_g}$.
Moreover,
we define the  \emph{grounded Sobolev spaces}
$$\mathring H^{s}_g\coloneqq  {\mathsf p}_g^{-\frac{s}{n}} \mathring L^2_g, \qquad
\|u\|_{\mathring H^s_g}\coloneqq \| {\mathsf p}_g^{\frac s{n}}u\|_{\mathring L^2_g}$$
and define $\mathring H^{-s}_g$ as the completion of $\mathring L^2_g$ w.r.t.~the norm
$\|u\|_{\mathring H^{-s}_g}\coloneqq \| {\mathsf p}_g^{-\frac s{n}}u\|_{\mathring L^2_g}$.

The scalar product
$\langle u|v\rangle_g:= \int_M uv\,d\vol_g$
on  $L^2_g$
 satisfies
 $|\langle u|v\rangle_g|\le \|u\|_{H^{-s}_g}\cdot \|v\|_{H^s_g}$
for $u\in L^2_g$ and $v\in H^s_g$, and thus continuously extends to a 
bilinear form
\[
\langle \emparg | \emparg \rangle_g: \ H^{-s}_g \times H^s_g\to\R
\]
for every $s\ge0$.
For $u\in H^{-s}_g$, consistently with our previous definition for $u\in L^2_g$, we put
\[
\langle u\rangle_g\coloneqq \frac1{\vol_g(M)} \langle u |1\rangle_g \ , \qquad
\pi_g(u)\coloneqq u-\langle u\rangle_g\ .
\]

\begin{lemma}
\label{l:PropertiesCoPolyH}
For every  admissible manifold~$(M,g)$:
\begin{enumerate}[$(i)$]
\item\label{i:l:PropertiesCoPolyH:1}
  The co-polyharmonic operator $\mathsf{P}_g$ is a compact perturbation of the poly-Laplacian:  for every~$\alpha>-\lambda_1$ there exists $C_\alpha=C(\alpha,g)>0$ such that the operator 
~$\mathsf{S}_\alpha\coloneqq \mathsf{P}_g-{(\alpha-\Delta_g)}^{n/2}$ satisfies
 \begin{align}\label{eq: res}
    \scalar{\mathsf{S}_\alpha  u}{u}_{L^2_g}\le C_{\alpha} \cdot \big\|(\alpha-\Delta)^{\frac{n-1}4}u\big\|_{L^2_g}^2, \qquad \forall  u\in{\mathcal H}^{n/2}.
 \end{align}
In particular,~$\scalar{\mathsf{S}_1 u}{u}_{L^2_g}\le C_1 \norm{u}_{{\mathcal H}^{(n-1)/2}}^2$.

\item\label{i:l:PropertiesCoPolyH:3}
  For every~$s\in\R$, the spaces $\mathcal H^s_g$ and $H^s_g$ coincide as sets and the respective norms are bi-Lipschitz equivalent to each other.
%

\label{i:l:PropertiesCoPolyH:5}

\item\label{i:l:PropertiesCoPolyH:6} The operator $\mathsf{p}_g$ with domain $H_g^n$ has discrete spectrum~$\mathrm{spec}(\mathsf{p}_g)= \set{\nu_j}_{j\in \mathbb{N}_0}$, indexed with multiplicities, satisfying~$\nu_j \geq 0$ for all~$j$, and~$\nu_0=0$ with multiplicity~$1$.
The corresponding family of eigenfunctions~$\seq{\psi_j}_{j\in\mathbb{N}_0}$ forms an orthonormal basis of~$L^2_g$.

\item\label{i:l:PropertiesCoPolyH:3b}
  For every~$s\in\R$,
  $$u\in \mathring{H}^s_g \quad\Longleftrightarrow\quad
u=\sum_{j\ge 1}\alpha_j \psi_j \text{ with } \sum_{j\ge 1} \nu_j^{\frac{2s}n}\alpha_j^2<\infty$$
and
$$u\in {H}^s_g \quad\Longleftrightarrow\quad u=\sum_{j\ge 0}\alpha_j \psi_j \text{ with } \sum_{j\ge 0} (1+\nu_j)^{\frac{2s}n}\alpha_j^2<\infty.$$
Hence, in particular,
 $\mathring{H}^s_g = \{u\in H^s_g: \langle u\rangle_g=0\}$.

\item\label{i:l:PropertiesCoPolyH:4aa} For every $s>0$ there exists $C=C(s)$ such that for all $r\in\R$,
$$\mathring{H}^{r+s}_g \subset \mathring{H}^r_g, \qquad \|\emparg\|_{\mathring{H}^r_g}\le C\,  \|\emparg\|_{\mathring{H}^{r+s}_g}.$$

\item\label{i:l:PropertiesCoPolyH:4} For every~$r\in\R$, the bounded operator $\mathsf{p}_g\colon \mathring{H}_g^{n+r}\to \mathring{H}_g^r$ has bounded inverse $\mathsf{k}_g\colon \mathring{H}_g^r\to\mathring{H}^{n+r}_g$.
For $r=0$, the operator $\mathsf{k}_g\colon \mathring{L}^2_g\to \mathring{L}^2_g$
admits a unique non-relabeled extension $\mathsf{k}_g\colon L^2_g\to \mathring{L}^2_g$, vanishing on constants and satisfying~$\mathsf{k}_g\mathsf{p}_g=\pi_g$ on~$L^2_g$.
This extension is an integral operator on~$L^2_g$ with symmetric kernel
\begin{align}\label{eq:l:PropertiesCoPolyH:1}
k_g(x,y)\coloneqq  \sum_{j =1}^\infty \frac{\psi_j(x)\, \psi_j(y)}{\nu_j} \comma \qquad x,y\in M\comma
\end{align}
where the convergence of the series is understood in~$L^2_g\otimes L^2_g$. 

\item\label{i:l:PropertiesCoPolyH:8} For $\ell\in\N$, define the operators $\mathsf{k}_{g,\ell}\colon {L}^2_g\to \mathring{L}^2_g$ by
\begin{equation*}
\mathsf{k}_{g,\ell}\, u\coloneqq  \sum_{j =1}^{\ell} \frac{\scalar{\psi_j}{u}_{L^2}}{\nu_j} \, \psi_j\fstop
\end{equation*}
Then for every $u\in L^2_g$, as $\ell\to\infty$,
\begin{equation} \mathsf{k}_{g,\ell}\, u \longrightarrow \mathsf{k}_{g}u\quad\text{in }L^\infty \fstop
\end{equation}

\item\label{l:co-poly-weyl}
The spectrum of $\mathsf p_g$ satisfies the Weyl asymptotic.
With $N(\nu)$ the number of eigenvalues lower than $\nu$, we get~$N(\nu) = c \, \nu + O(\nu^{1-1/n})$ as~$\nu \to \infty$.
In particular, we find
\begin{equation}\label{e:weyl-bound}
\nu_{j} = c j + O(j^{1-1/n})\comma \qquad j \to \infty\fstop
\end{equation}

\item\label{i:l:PropertiesCoPolyH:7} Given any~$r\in\R$,  the embedding $\mathsf{Id}:  \mathring{H}^{r+s}_g \hookrightarrow \mathring{H}^r_g$ is trace-class if and only if $s>n$, and it is Hilbert--Schmidt if and only if $s>\frac n2$.

\end{enumerate}
\end{lemma}

\begin{proof}
\ref{i:l:PropertiesCoPolyH:1}
For every $\alpha>-\lambda_1$, the operator $\mathsf{S}_\alpha$ is a linear differential operator of order $\le n-1$ with smooth (hence bounded) coefficients on $M$, and~\eqref{eq: res} readily follows.

\ref{i:l:PropertiesCoPolyH:3} 
It suffices to show the statement for~$s=1/2$. The claim for any other $s>0$ then follows by spectral calculus, and for $s<0$ by duality.

As a consequence of Theorem~\ref{Pol-basic}~\ref{i:Pol-basic:4} and admissibility, the (strictly) positive operator~$(\mathsf{p}_g, \mathring{\mathcal H}^n_g)$ has positive self-adjoint square root~$(\sqrt{\mathsf{p}_g}, \mathring{\mathcal H}^{n/2}_g)$, and the latter defines a  Hilbert norm on~$\mathring{\mathcal H}^{n/2}_g$.
Thus, the linear operator $\iota\coloneqq (-\Delta_g)^{-n/4}\sqrt{\mathsf{p}_g}\colon \mathring{\mathcal H}^{n/2}_g\to \mathring{\mathcal H}^{n/2}_g$ is well-defined, positive, and injective.
Moreover, $\iota$ is an isometry
\begin{equation*}
\iota \colon \paren{\mathring{\mathcal H}^{n/2}_g, \tnorm{\sqrt{\mathsf{p}_g}\emparg}_{\mathring{\mathcal H}_g}}\longrightarrow \paren{\mathring{\mathcal H}^{n/2}_g, \norm{\emparg}_{\mathring{\mathcal H}^{n/2}_g}}\comma
\end{equation*}
and in fact unitary, since~$\ker \iota=\set{0}$ by strict positivity of both~$\sqrt{\mathsf{p}_g}$ and~$(-\Delta_g)^{-n/4}$ on the appropriate spaces of grounded functions.
As a consequence,~$\iota\colon \mathring{\mathcal H}^{n/2}_g \to \mathring{\mathcal H}^{n/2}_g$ is surjective, and thus bijective.
It suffices to show it is also $\mathring{\mathcal H}^{n/2}_g$-bounded, in which case it has a $\mathring{\mathcal H}^{n/2}_g$-bounded inverse~$\iota^{-1}$ by the Bounded Inverse Theorem.
The former fact follows if we show that~$\iota\iota^*$ is $\mathring{\mathcal H}^{n/2}_g$-bounded. We have
\begin{align*}
\iota\iota^*=(-\Delta_g)^{-n/4} \mathsf{P}_g (-\Delta_g)^{-n/4} = \mathsf{Id}_{\mathring{\mathcal H}^{n/2}_g} + (-\Delta_g)^{-n/4} \mathsf{S}_0 (-\Delta_g)^{-n/4}\fstop
\end{align*}
By squaring the operators in~\eqref{eq: res} with~$\alpha=0$, the latter is a $\mathring{\mathcal H}^{n/2}_g$-bounded perturbation of the identity on~$\mathring{\mathcal H}^{n/2}_g$, and the assertion follows.


\ref{i:l:PropertiesCoPolyH:6}
Since~$\mathring{H}^n_g$ embeds compactly into~$\mathring{H}_g^0$ by the Rellich--Kondrashov Theorem, the operator~$\mathsf{k}_g\colon \mathring{H}^0_g\to \mathring{H}^0_g$ is compact, being the composition of the bounded operator~$\mathsf{k}_g\colon \mathring{H}^0_g\to \mathring{H}^n_g$ with the compact Sobolev embedding.
The spectral properties follow from the (strict) positivity of~$(p_g,\mathring{H}^n)$ on~$\mathring{H}^0_g$ and the $\mathring{H}^0_g$-compactness of~$\mathsf{k}_g$.
The assertion on eigenfunctions holds by the Spectral Theorem for unbounded self-adjoint operators.

\ref{i:l:PropertiesCoPolyH:3b} Direct calculation and the fact that $\inf_{j\ge 1}\frac{\nu_j}{1+\nu_j}>0$.

\ref{i:l:PropertiesCoPolyH:4aa} 
Let us first observe that 
$\|(1-\Delta_g)^{1/2}u\|_{L^2_g}\ge \frac1C\, \|u\|_{L^2_g}
$
for all $u\in L^2_g$, and thus
$$\|(1-\Delta_g)^{s/2}u\|_{L^2_g}\ge \frac1{C^s}\, \|u\|_{L^2_g}
$$
for all $s>0$. By positivity of $\mathsf p_g$ on $\mathring L^2_g$ and norm equivalence of $\mathcal H^n_g$ and $H^n_g$ it follows 
$$\|\mathsf p_g u\|_{\mathring L^2_g}\ge \frac1{C'}\,\|(1+\mathsf p_g) u\|_{\mathring L^2_g}\ge \frac1{C''}\|(1-\Delta_g)^{n/2}u\|_{\mathring L^2_g}\ge \frac1{C'''}\, \|u\|_{\mathring L^2_g}
$$
for $u\in\mathring L^2_g$. 
This lower estimate for the self-adjoint operator $\mathsf p_g$ implies an analogous estimate for any of its positive powers.
Thus
$$ \|u\|_{\mathring{H}^{s}_g}=\|\mathsf p_g^{s/n} u\|_{\mathring L^2_g}\ge
\frac1{C_s}\, \|u\|_{\mathring L^2_g}
$$
 for any $s>0$.
This proves the claim for $r=0$. The claim for general $r$ follows readily.

\ref{i:l:PropertiesCoPolyH:4} 
It suffices to show the statement for~$r=0$.
We show that~$\sqrt{\mathsf{p}_g}\colon \mathring{H}^{n/2}_g\to \mathring{H}_g^0$ is invertible with bounded inverse~$\sqrt{\mathsf{k}_g}$. 
As a consequence of the bijectivity of~$\iota$ in~\ref{i:l:PropertiesCoPolyH:3}, and since~$(-\Delta_g)^{n/2}\colon \mathring{H}^{n/2}_g\to \mathring{H}_g^0$ is surjective, the operator~$\sqrt{\mathsf{p}_g} = (-\Delta_g)^{n/2}\iota \colon \mathring{H}^{n/2}_g\to \mathring{H}_g^0$ is as well surjective, and thus bijective.
Its inverse~$\sqrt{\mathsf{k}}_g \coloneqq  \iota^{-1} (-\Delta_g)^{-n/2}\colon \mathring{H}_g^0\to \mathring{H}^{n/2}_g$ is a bounded operator, since so are~$\iota^{-1}\colon \mathring{H}^{n/2}_g\to \mathring{H}^{n/2}_g$, by~\ref{i:l:PropertiesCoPolyH:3}, and~$(-\Delta_g)^{-n/2}\colon \mathring{H}_g^0\to \mathring{H}^{n/2}_g$.

\ref{i:l:PropertiesCoPolyH:8} By the norm equivalence stated in~\ref{i:l:PropertiesCoPolyH:3},
\begin{align*}
\norm{\mathsf{k}_gu-\mathsf{k}_{g,\ell}\, u}_{\mathcal H^n_g}^2\simeq \norm{\mathsf p_g(\mathsf{k}_g u-\mathsf{k}_{g,\ell}\, u)}_{L^2_g}^2=\sum_{j=\ell+1}^\infty
\scalar{\psi_j}{u}_{L^2_g}^2\longrightarrow0
\end{align*} as $\ell\to\infty$ for every $u\in L^2_g$. Hence, by Sobolev embedding, $\mathsf{k}_{g,\ell}\, u \to \mathsf{k}_{g}u$ in $L^\infty$.

  \ref{l:co-poly-weyl} is H{\"{o}}rmander's Weyl law for positive pseudo-differential operators.
Indeed, choosing~$dx=d\vol_g$ and integrating~\cite[Eqn.~(1.1)]{Hoe68}, the assertion follows from~\cite[Thm.~1.1]{Hoe68}.

\ref{i:l:PropertiesCoPolyH:7} For any $s>0$, the embedding $\mathsf{Id}:  \mathring{H}^{r+s}_g \hookrightarrow \mathring{H}^r_g$ is trace-class (or Hilbert--Schmidt, resp.) if and only if the operator
$\mathsf p_g^{-s/n}: \mathring L^2_g\to \mathring L^2_g$ is so. By definition, the latter is trace-class (or Hilbert--Schmidt, resp.) if and only if
$$\sum_j \nu_j^{-s/n}<\infty\qquad \Big(\text{or }\sum_j \nu_j^{-2s/n}<\infty \text{, resp.}\Big)$$
which 
in turn --- according to  \ref{l:co-poly-weyl} --- holds true if and only if $s>n$ (or $s>n/2$, resp.).

\end{proof}

\begin{remark}
\begin{enumerate}[$(a)$, wide]
\item Elliptic regularity theory implies that off the diagonal of $M\times M$, the function  $(x,y)\mapsto k_g(x,y)$ is $\mathcal{C}^\infty$.

\item The symmetry of the integral kernel $k_g$ implies that
\begin{equation}\label{k-ground}
  \int_M k_g(x,y)\,d\mathsf{vol}_g(y)=0\comma \qquad x\in M \fstop
\end{equation}
Indeed, ${\mathsf{k}}_g f\in \mathring L^2_g$ implies
$\int\braket{\int k_g(x,y)d\mathsf{vol}_g(x)}f(y)\,d\mathsf{vol}_g(y)=0$ for all $f\in \mathring L^2_g$
which in turn implies that $\int k_g(x,y)d\mathsf{vol}_g(x)$ is constant in $y$.
By symmetry, this constant must vanish.
\end{enumerate}
\end{remark}

Of particular importance in the sequel will be the spaces $H^{s}_g$ for $s=\frac n2$ and $s=-\frac n2$.

\begin{definition} Given any admissible manifold $(M,g)$, we denote the scalar products for the Hilbert spaces  $H_g^{n/2}$ and $H_g^{-n/2}$
 by
\begin{align}\label{e:p-form}
\mathfrak{p}_g(u,v) &\coloneqq \int \sqrt{\mathsf{p}_g} u\, \sqrt{\mathsf{p}_g} v \dd \vol_g\ , \qquad & u,\, v \in H_g^{n/2} \comma\\
\mathcal{K}_g(u,v) &\coloneqq \int \sqrt{\mathsf{k}_g} u\, \sqrt{\mathsf{k}_g} v \dd \vol_g\ , \qquad & u,\, v \in \mathring H_g^{-n/2} \fstop
\end{align}
\end{definition}
Restricted to the space $\mathring L^2_g$, the bilinear form 
 $\mathcal{K}_{g}$ is given by
 \begin{align}\label{cov-rel}
   \mathcal{K}_\g(u,v)&= \scalar{\mathsf{k}_gu}{v}_{L^2_g}
=\iint u(x)\, k_\g(x,y)\, v(y)\, d\vol_\g(x) \, d\vol_\g(y) \,.
\end{align}
Observe that the right-hand side here is also well defined for ungrounded $u,v\in L^2_g$ which allows us to consider 
 $\mathcal{K}_{g}$  also  as a bilinear form on $L^2_g$ with
 $\mathcal{K}_\g(u+C,v+C')=\mathcal{K}_\g(u,v)$ for $u,v\in \mathring L^2(\M,\vol_\g)$ and $C,C'\in\R$. 
   Moreover, we always implicitly extend the operator $\mathsf{k}_g$ to~$L^2_g$ by setting $\mathsf{k}_g c = 0$ for all constant $c$.
  It is the \emph{pseudo-inverse} of $\mathsf{p}_g$ on~$L^2_g$ in the sense that:
  \begin{equation*}
    \mathsf{k}_g \mathsf{p}_g = \mathsf{p}_g \mathsf{k}_g = \pi_{g} \fstop
  \end{equation*}
  We call~$\mathsf{k}_g$ the \emph{co-polyharmonic Green operator}.
  It has the integral kernel $k_g$ given in~\eqref{eq:l:PropertiesCoPolyH:1}, and we call $k_g$ the \emph{co-polyharmonic Green kernel}.


\bigskip

In the following lemma we make use of some arguments in complex-interpolation theory.
We refer the reader to~\cite[\S1.2.1]{Tri78} for the necessary standard definitions of interpolation couple, and to~\cite[\S1.9.2-3]{Tri78} for results on complex interpolation.
We denote by~$[A_0,A_1]_\theta$,~$\theta\in (0,1)$, the standard complex interpolation of Banach spaces~$A_0,A_1$.

For Banach spaces~$A_0,A_1,A$ of functions on~$M$ we write
\[
\cdot \colon A_0 \times A_1 \longrightarrow A
\]
to indicate that the pointwise product on~$A_0\times A_1$ is a continuous bilinear map with range contained in~$A$, i.e.\ for every~$f\in A_0$ and~$g\in A_1$ we have~$fg\in A$ and~$\norm{fg}_A\lesssim \norm{f}_{A_0}\norm{g}_{A_1}$.

\begin{lemma}\label{Hn/2}
\begin{enumerate}[$(i)$]
\item\label{i:l:Multiplication:1} For every~$s\in [0,\tfrac{n}{2}]$, for every~$\varepsilon>0$,
\[
\cdot\colon \mathcal{H}^s_g \times \mathcal{H}^{n/2+\varepsilon}_g \longrightarrow \mathcal{H}^s_g \ ;
\]

\item\label{H+n/2} For every~$r,s,t\in\R$ with~$s,t\geq r\geq 0$ and~$s+t\geq r+\tfrac{n}{2}$,
\begin{equation}\label{eq:l:Multiplication:0.1}
\cdot\colon \mathcal{H}^s_g \times \mathcal{H}^t_g \longrightarrow \mathcal{H}^r_g \ ;
\end{equation}

\item\label{i:l:Multiplication:3} For every~$t\geq \tfrac{n}{2}$, for every~$s\in\R$,
\begin{equation}\label{eq:l:Multiplication:0.2}
\cdot\colon \mathcal{H}^s_g \times \mathcal{H}^t_g \longrightarrow \mathcal{H}^s_g \ ;
\end{equation}

\item\label{i:l:Multiplication:4} For every~$s\in \R$, for every~$\varphi\in\mathcal{C}^\infty(M)$,
\begin{equation}\label{weighted-pairing}
\langle e^{n\varphi} u | v\rangle_g = \langle u | e^{n \varphi} v\rangle_g=\langle u|  v\rangle_{e^{2\varphi}g} \ , \qquad u\in \mathcal{H}^{-s}_g, v\in\mathcal{H}^s_g \ .
\end{equation}

\item\label{H+-n/2} All the above assertions hold with~$H^s_g$ in place of~$\mathcal{H}^s_g$.
\end{enumerate}
\end{lemma}

\begin{proof}
For~\ref{i:l:Multiplication:1} and~\ref{H+n/2} we adapt to our setting the proof for Euclidean spaces in~\cite[Lem.~5.2, Thm.~5.1]{BehHol21}.

\ref{i:l:Multiplication:1} Since~$\tfrac{n}{2}+\varepsilon>\tfrac{n}{2}$, by~\cite[Thm.~24]{CouRusTar01} the space~$\mathcal{H}^{n/2+\varepsilon}_g$ is an algebra and
\begin{equation}\label{eq:l:Multiplication:1}
\cdot\colon \mathcal{H}^{n/2+\varepsilon}_g \times \mathcal{H}^{n/2+\varepsilon}_g \longrightarrow \mathcal{H}^{n/2+\varepsilon}_g .
\end{equation}

Combing~\cite[Thm.~5(iii), Thm.~2(iii), Thm.~4(i)]{Tri85} we have~$\mathcal{H}^{n/2+\varepsilon}_g\hookrightarrow L^\infty_g$.
Thus, since $\cdot \colon L^\infty_g\times L^2_g\to L^2_g$, we further have
\begin{equation}\label{eq:l:Multiplication:2}
\cdot \colon \mathcal{H}^{n/2+\varepsilon}_g \times \mathcal{H}^0_g \longrightarrow \mathcal{H}^0_g.
\end{equation}

By complex interpolation of bilinear forms (see~\cite[\S1.19.5]{Tri78}), the pointwise product in~\eqref{eq:l:Multiplication:1} and~\eqref{eq:l:Multiplication:2} interpolates to
\begin{equation}
\cdot \colon \mathcal{H}^{n/2+\varepsilon}_g \times [\mathcal{H}^0_g,\mathcal{H}^{n/2+\varepsilon}_g]_\theta \longrightarrow  [\mathcal{H}^0_g,\mathcal{H}^{n/2+\varepsilon}_g]_\theta, \qquad \theta\in (0,1) .
\end{equation}
Choosing~$\theta$ so that~$s= \theta(\tfrac{n}{2}+\varepsilon)$ we have~$[\mathcal{H}^0_g,\mathcal{H}^{n/2+\varepsilon}_g]_\theta= \mathcal{H}^s_g$ by~\cite[Cor.~4.6]{Str83}, and the conclusion follows.

\ref{H+n/2} If~$r>\tfrac{n}{2}$ the space~$\mathcal{H}^r_g$ is an algebra and therefore, since~$s,t\geq r$,
\[
\cdot\colon \mathcal{H}^s_g \times \mathcal{H}^t_g \hookrightarrow \mathcal{H}^r_g \times \mathcal{H}^r_g \longrightarrow \mathcal{H}^r_g \ ,
\]
which is the assertion.
If otherwise~$r\in [0,\tfrac{n}{2}]$, let~$\varepsilon\coloneqq s+t-r-\tfrac{n}{2}>0$. By~\ref{i:l:Multiplication:1},
\begin{subequations}
\begin{gather}
\label{eq:l:Multiplication:3a}
\mathcal{H}^r_g \times \mathcal{H}^{n/2+\varepsilon}_g \longrightarrow \mathcal{H}^r_g \ ,
\\
\label{eq:l:Multiplication:3b}
\mathcal{H}^{n/2+\varepsilon}_g \times \mathcal{H}^r_g \longrightarrow \mathcal{H}^r_g \ .
\end{gather}
\end{subequations}
Now, since~$r\leq s$ we have~$s\leq s+t -r$ and so~$s\leq \tfrac{n}{2}+\varepsilon$.
Thus, there exists~$\theta\in [0,1]$ with~$(1-\theta)r+\theta(\tfrac{n}{2}+\varepsilon)=s$ and~$(1-\theta)(\tfrac{n}{2}+\varepsilon)+\theta r= t$.
Again by~\cite[Cor.~4.6]{Str83} we have~$[\mathcal{H}^r_g,\mathcal{H}^{n/2+\varepsilon}_g]_\theta= \mathcal{H}^s_g$ and~$[\mathcal{H}^{n/2+\varepsilon}_g,\mathcal{H}^r_g]_\theta= \mathcal{H}^t_g$, and~\eqref{eq:l:Multiplication:0.1} follows by complex interpolation of the pointwise product in~\eqref{eq:l:Multiplication:3a},~\eqref{eq:l:Multiplication:3b}.

\ref{i:l:Multiplication:3} If~$s\geq 0$, the assertion is~\ref{H+n/2} with~$r\coloneqq s$. If~$s<0$ we argue as follows.
For every~$f\in \mathcal{H}^t_g$ and~$v\in\mathcal{H}^{-s}_g$, we have~$fv\in\mathcal{H}^{-s}_g$ by~\ref{H+n/2}.
Thus, for every~$u\in L^2_g$, for some constant~$C=C_{s,t}>0$,
\[
\scalar{fu}{v}_{L^2_g}= \scalar{u}{fv}_{L^2_g} \leq \norm{u}_{\mathcal{H}^s_g} \norm{fv}_{\mathcal{H}^{-s}_g} \leq C_{s,t}\norm{u}_{\mathcal{H}^s_g} \norm{v}_{\mathcal{H}^{-s}_g} \ .
\]
As a consequence,~$fu$ defines a continuous linear functional on~$\mathcal{H}^{-s}_g$ and we have
\[
\norm{fu}_{\mathcal{H}^{s}_g}\leq C_{s,t}\norm{f}_{\mathcal{H}^t_g} \norm{u}_{\mathcal{H}^s_g} \ , \qquad u\in L^2_g \ .
\]
By density of~$L^2_g$ in~$\mathcal{H}^s_g$, the above inequality extends to~$u\in \mathcal{H}^s_g$, which proves the assertion.

\ref{i:l:Multiplication:4} follows easily from~\ref{i:l:Multiplication:3} by approximation of $u\in \mathcal{H}^{-s}_g$ by $u_j\in L^2_g$.

\ref{H+-n/2} is an immediate consequence of Lemma~\ref{l:PropertiesCoPolyH}~\ref{i:l:PropertiesCoPolyH:3}.
\end{proof}


\subsection{Estimates for  co-polyharmonic Green kernels}

\begin{theorem}\label{est-k-log} For every admissible manifold $(M,g)$, the co-polyharmonic Green kernel~$k_g$ satisfies
\begin{equation}\label{eq:d:Admissible:2}
\abs{k_g(x,y)-\log\frac1{d_g(x,y)}}\le C_0
\end{equation}
for some $C_0=C_0(g)$.
\end{theorem}

\begin{proof}
By the second resolvent identity for the operators~$\mathsf{k}_g$ and $\frac1{a_n}\mathring{\mathsf{G}}_{n/2}\colon \mathring{H}_g^0\to \mathring{H}_g^0$,
\begin{equation}\label{eq:t:Admissibility:1}
\mathsf{k}_g - \frac1{a_n}\mathring{\mathsf{G}}_{n/2} =\mathsf{k}_g\, \mathsf{S}_0\, \mathring{\mathsf{G}}_{n/2}=\mathsf{k}_g\, \mathsf{S}_0 \, \mathring{\mathsf G}_{\frac{n-1}{4}}\, \mathring{\mathsf{G}}_{\frac{n+1}{4}}
\end{equation}
with~$\mathsf{S}_0=\mathsf P_g-(-\Delta_g)^{n/2}$ as in Lemma~\ref{l:PropertiesCoPolyH}~\ref{i:l:PropertiesCoPolyH:1}.
Similarly to the proof of Proposition~\ref{prop: greenslog}, the operators~$\mathring{\mathsf{G}}_{\frac{n+1}{4}}\colon \mathring{L}^1\to \mathring{H}^0$ and~$\mathring{\mathsf{G}}_{\frac{n-1}{4}}\colon\mathring{H}^0\to \mathring{H}^{\frac{n-1}{2}}$ are bounded.
By Theorem~\ref{Pol-basic}~\ref{i:Pol-basic:3},~$\mathsf{S}_0$ is a differential operator of order at most~$n-1$ with smooth (hence bounded) coefficients.
As a consequence,~$\mathsf{S}_0\colon \mathring{H}^{\frac{n-1}{2}}\to\mathring{H}^{-\frac{n-1}{2}}$ is a bounded operator.
Furthermore, choosing~$r=-\frac{n-1}{2}$ in Lemma~\ref{l:PropertiesCoPolyH}~\ref{i:l:PropertiesCoPolyH:3}, the operator~$\mathsf{k}_g\colon \mathring{H}^{-\frac{n-1}{2}}\to \mathring{H}^{\frac{n+1}{2}}$ is bounded.

Combining the previous assertions with~\eqref{eq:t:Admissibility:1} shows that~$\mathsf{k}_g-\frac1{a_n}\mathring{\mathsf{G}}_{n/2}\colon \mathring{L}^1\to \mathring{H}^{\frac{n+1}{2}}$ is bounded, thus~$\mathsf{k}_g-\frac1{a_n}\mathring{\mathsf{G}}_{n/2}\colon \mathring{L}^1\to L^\infty$ is bounded as well, by continuity of the Sobolev--Morrey embedding.
Finally, by~\cite[Thm.~2.2.5]{DunPet40}, the latter operator admits a bounded integral kernel, and the conclusion follows from Proposition~\ref{prop: greenslog}.
\end{proof}

The previous theorem has also been 
derived with 
different (and partly rather sketchy) arguments
in \cite[Lemma 2.1]{Ndiaye}. 


\begin{proposition}\label{conf-k}\label{t:Covariant} 
  Assume that $(M,g)$ is admissible and that $g'\coloneqq e^{2\varphi}g$ for some $\varphi\in\mathcal{C}^\infty(M)$.
  Then the co-polyharmonic Green operator ${\mathsf{k}}_{g'}$ 
is given by 
\begin{equation}\label{k-trafo2}
{\mathsf{k}}_{g'}u=\big(\pi_{g'}\circ{\mathsf{k}}_{g}\big)\big(e^{n\varphi} u\big) \comma \qquad u\in L^2\comma
\end{equation}
and  
  the co-polyharmonic Green kernel $k_{g'}$ 
 by 
\begin{align}\label{trafo-k}
k_{g'}(x,y)=& \ k_g(x,y)-\frac12\bar\phi(x) -\frac12\bar\phi(y)
\end{align}
with $\bar\phi\in\C^\infty(M)$  defined by 
\begin{align*}\bar\phi&\coloneqq  \frac2{\mathsf{vol}_{g'}(M)}\int k_g(\emparg,z)\,d\mathsf{vol}_{g'}(z)-\frac1{\mathsf{vol}_{g'}(M)^2}\iint k_g(z,w)\,d\mathsf{vol}_{g'}(z)\,d\mathsf{vol}_{g'}(w)\fstop
 \end{align*}
 \end{proposition}

\begin{proof} Let $k_{g'}$ be the integral kernel defined by the right-hand side of~\eqref{trafo-k}. Obviously, $k_{g'}$ is symmetric.
Furthermore, by \eqref{eq:d:Admissible:2}
\begin{align}
\nonumber
\abs{k_{g'}(x,y)-k_{g}(x,y)}\leq&\ \frac{3}{\mathsf{vol}_{g'}(M)} \, \sup_w \int \big|k_g(z,w)\big|d\mathsf{vol}_{g'}(z)
\\
\label{eq:t:Covariant:4}
\leq&\ 3\, C\, \mathsf{vol}_{g'}(M) + \frac{3}{\mathsf{vol}_{g'}(M)} \, \sup_w \int \abs{\log\frac{1}{d_g(z,w)}} d\mathsf{vol}_{g'}(z) <\infty \comma
\end{align}
and, since  
$e^{\inf \varphi}\, d_g \leq d_{g'} \leq e^{\sup \varphi}\, d_g$,
\begin{align}\label{eq:t:Covariant:5}
\abs{\log\frac{1}{d_g(x,y)}-\log\frac{1}{d_{g'}(x,y)}} \leq C_{\varphi,g} \fstop
\end{align}
Thus the kernel~$k_{g'}$ satisfies~\eqref{eq:d:Admissible:2} with~$k_{g'}$ in place of~$k_g$ and~$g'$ in place of~$g$ for some constant~$C_0(g')$.

Moreover, straightforward calculation yields the identity \eqref{k-trafo2} for the integral operator~$\mathsf{k}_{g'}$ associated with the kernel~$k_{g'}$.
It remains to prove that the operator~$\mathsf{k}_{g'}$ is the inverse of $\mathsf{p}_{g'}$.
To see this, recall that we have $\mathsf{p}_{g'} = \ee^{-n \varphi} \mathsf{p}_{g}$.
  Thus for all $u \in \mathring H^{n/2}_{g'}$,
  \begin{equation*}
    \begin{split}
      {\mathsf{k}}_{g'} \, \mathsf{p}_{g'} u &= {\mathsf{k}}_{g}\mathsf{p}_{g}u - \big\langle{\mathsf{k}}_{g} \mathsf{p}_{g}u \big\rangle_{g'} 
                                         = u - \langle u\rangle_{g} - \big\langle u - \langle u\rangle_{g}\big\rangle_{g'} 
                                          = u.
    \end{split}
  \end{equation*}
  Consequently we have~${\mathsf{k}}_{g'} \, \mathsf{p}_{g'} u = u = \mathsf{p}_{g'}\, {\mathsf{k}}_{g'}u$ and the claim follows by uniqueness of the inverse.
\end{proof}

\begin{remark}\label{phi-bar} The transformation formula \eqref{trafo-k} for the co-polyharmonic Green kernels can be re-phrased as follows. Given $\varphi\in\mathcal{C}^\infty(M)$, let $\varphi_0\coloneqq \varphi-c$ with $c$ chosen such that $\int e^{n\varphi_0}d\vol=1$. Then
\begin{equation}
k_{e^{2\varphi}g}(x,y)=k_g(x,y)-{\mathsf{k}}_g\big(e^{n\varphi_0}\big)(x)-{\mathsf{k}}_g\big(e^{n\varphi_0}\big)(y)+\big\langle{\mathsf{k}}_g\big(e^{n\varphi_0}\big), e^{n\varphi_0}\big\rangle_{L^2_g} \fstop
\end{equation}
\end{remark}

 \begin{proposition}\label{prop:hn} Given any admissible manifold $(M,g)$ and  $g'=e^{2\varphi}g$ with $\varphi\in{\mathcal C}^\infty(\M)$, then
\begin{enumerate}[$(i)$]
\item\label{hn/2} $u\in \mathring H^{n/2}_g$ implies $u\in \mathring H^{n/2}_{g'}$ and $\|u\|_{\mathring H^{n/2}_{g'}}=\|u\|_{\mathring H^{n/2}_{g}}$ or, in other words,
\begin{equation}\mathfrak p_{g'}(u,u)=\mathfrak p_{g}(u,u).\end{equation}

\item\label{h-n/2} $u\in \mathring H^{-n/2}_g$  implies 
 $e^{-n\varphi} u\in \mathring H^{-n/2}_{g'}$ and $\|e^{-n\varphi} u\|_{\mathring H^{n/2}_{g'}}=\|u\|_{\mathring H^{n/2}_{g}}$
or, in other words,
\begin{equation}\label{conf-trafo3}\mathcal K_{g'}(e^{-n\varphi}u,e^{-n\varphi}u)=\mathcal K_{g}(u,u).\end{equation}
\end{enumerate}
\end{proposition}

\begin{proof}
\ref{hn/2} Immediate consequence of $\mathsf P_{g'}u=e^{-n\varphi}\mathsf P_g$ and $\vol_{g'}=e^{n\varphi}\vol_g$.

\ref{h-n/2}  The norm identity follows from \eqref{k-trafo2} and \eqref{Hn/2}\ref{H+-n/2}:
\begin{align*}
\mathcal K_{g'}(e^{-n\varphi}u,e^{-n\varphi}u)&=\langle e^{-n\varphi}u|\mathsf k_{g'}(^{-n\varphi}u)\rangle_{g'}=
\langle e^{-n\varphi}u|\pi_{g'}(\mathsf k_{g}u)\rangle_{g'}\\
&=
\langle u|\pi_{g'}(\mathsf k_{g}u)\rangle_{g}=
\langle u|\mathsf k_{g}u\rangle_{g}-\langle u|1\rangle_{g}\cdot \langle \mathsf k_gu\rangle_{g'}=\mathcal K_{g}(u,u).
\end{align*}
Moreover, $\langle u\rangle_g=0$ if and only if $\langle e^{-n\varphi}u\rangle_{g'}=0$.
\end{proof}

\section{The co-polyharmonic  Gaussian field}\label{s:copolyharmonic-field}
In what follows, we consider an admissible manifold $(\M,\g)$ of even dimension~$n$.
We make use of the normalized co-polyharmonic operator~${\mathsf{p}}_\g\coloneqq  {a_n}{\mathsf{P}}_\g$ and its inverse
 $\mathsf k_g$
with  symmetric integral kernel
$k_g$ which 
has precise logarithmic divergence
\begin{equation}\label{log-div}
\Big|k_g(x,y)-\log\frac1{d_g(x,y)}\Big|\le C, \qquad \forall x,\, y \in M.
\end{equation}
Our goal is to define and analyze a random field $h$ on $M$ with law formally   characterized as
\begin{equation}\label{heur-gauss}
  d\boldsymbol{\nu}_g(h)= \frac1{Z_g}\exp\Big(-\frac{1}2\mathfrak{p}_g(h,h)\Big)\,dh
\end{equation}
where $dh$ stands for the (non-existing) uniform distribution on fields and $Z_g$ denotes some normalization constant.

\subsection{Existence and uniqueness, equivalent characterizations}

\begin{definition} A  co-polyharmonic Gaussian field on $(M,g)$ is a centered Gaussian random variable $h$ on $\mathring H^{-s}_g$ for some $s>0$ with covariance
\begin{equation}
\label{eq:covariance-chp}
  \mathbf{E} \Big[\langle h|u\rangle_g\cdot \langle h|v\rangle_g\Big]=\mathcal K_g(u,v)
   \quad\qquad \forall u,v\in \mathring H^{s}_g \fstop
\end{equation}
Here  $\langle\emparg|\emparg\rangle_g$ on the left hand side denotes the pairing $\langle\emparg|\emparg\rangle_{\mathring H^{-s}_g,\mathring H^{s}_g}$, and the right-hand side can be rewritten as 
\begin{equation}
\mathcal K_g(u,v)=\iint k_g(x,y)\, u(x)\, v(y)\, d\vol_g(x)\, d\vol_g(y)=\scalar{u}{v}_{\mathring H^{-n/2}_g}.
\end{equation}
\end{definition}

%

\begin{theorem}
  For every admissible manifold~$(\M,\g)$ there exists a  co-polyharmonic Gaussian field, unique in distribution.
  
  More precisely, for any $s>0$ there exists an $\mathring H^{-s}_g$-valued  co-polyharmonic Gaussian field which is unique in distribution. It is supported on 
  $$\bigcap_{t>0} \mathring H^{-t}_g.$$
  For all $s,t>0$, an $\mathring H^{-s}_g$-valued  co-polyharmonic Gaussian field and an
   $\mathring H^{-t}_g$-valued  co-polyharmonic Gaussian field coincide in distribution.
  
  \end{theorem}
  We denote the law of \emph{the }  co-polyharmonic Gaussian field by $\mathsf{CGF}_{\M,\g}$ or simply --- since throughout this paper  mostly $M$ is  fixed --- by $\mathsf{CGF}_{\g}$.
Equivalently, $\mathsf{CGF}_\g$ can be characterized as the unique centered Gaussian probability measure $\boldsymbol{\nu}_g$ on $\mathring H^{-s}_g$ that satisfies
\begin{align}\label{eq:CharacteristicF}
  \int
  e^{i\scalar{h}{u}_g}\,d\boldsymbol{\nu}_g(h)=
  \exp\tbraket{-\tfrac{1}{2}\mathcal{K}_g(u,u)}
  \quad\qquad\forall u\in  \mathring H^{s}_g\fstop
\end{align}

\begin{proof}\label{abstract-wiener-space}
Let us consider the \emph{abstract Wiener space} $(\iota, H, B)$ in the sense of Gross, following the notation and presentation in \cite{Bog98}, with the Banach space $B\coloneqq \mathring H^{-\varepsilon}_g$ (`Wiener space'), the Hilbert space $H\coloneqq \mathring H^{n/2}_g$ (`Cameron-Martin space'), and the embedding
$\iota: H\hookrightarrow B$ which is `measurable in the sense of Gross', cf.~\cite[Example 3.9.7 ]{Bog98}, since 
$$\|u\|_B=\| Tu\|_H$$
with the operator $T=\mathsf p_g^{-1-\varepsilon/n}$ being Hilbert--Schmidt according to Lemma~\ref{l:PropertiesCoPolyH}~\ref{i:l:PropertiesCoPolyH:7}.
The existence of the requested Gaussian measure on $B$ is then a key result of the theory of abstract Wiener spaces 
\cite[Theorem 3.9.5 ]{Bog98}.
 \end{proof}

\begin{remark} Based on the Bochner--Minlos Theorem, the co-polyharmonic Gaussian field can alternatively be defined as a random field  on the space $\mathfrak{D}'$ of distributions on $M$.
Here ~$\mathfrak{D}\coloneqq \mathcal C^\infty(M)$ denotes the space of test functions, endowed with its usual Fr\'echet topology, which is is a nuclear space, see, for instance, the comments preceding~\cite[Ch.~II, Thm.~10, p.~55]{Gro66}.
$\mathfrak{D}'$, the topological dual of~$\mathfrak{D}$, is endowed with the Borel $\sigma$-algebra induced by the weak* topology.
Set~$\chi(u)\coloneqq  \exp\tbraket{-\tfrac{1}{2} \mathcal{K}_g(u,u)}$.
It satisfies~$\chi(0)=1$.
Moreover, since~$M$ is admissible,~$\mathcal{K}_g$ is a semi-definite inner product on~$\mathfrak{D}$, thus, by, e.g.,~\cite[Prop.~2.4]{LodSheSunWat16}, $\chi$ is totally positive definite.
By Lemma~\ref{l:PropertiesCoPolyH}\ref{i:l:PropertiesCoPolyH:3}, $u\mapsto\sqrt{\mathcal{K}_g(u,u)}$ is continuous with respect to the~$\mathcal H^{-n/2}_g$-norm on~$\mathfrak{D}$.
Since~$\mathfrak{D}$ embeds continuously into~$\mathcal H^{-s}_g$ for every~$s\in \R$, the functional~$\chi$ is continuous on~$\mathfrak{D}$.
The claim follows by the Bochner--Minlos Theorem~\cite[\S{IV.4.3}, Thm.~4.3, p.~410]{VakTarCho87}.
\end{remark}

In the sequel, we will freely switch between the general representation of the co-polyharmonic Gaussian field as a measurable map 
$h: \Omega\to \mathring H^{-s}_g$  with law $h_*  \mathbf P= \mathsf{CGF}_{\g)}$, defined on some probability space $(\Omega, \mathfrak F, \mathbf P)$, and the standard representation of the co-polyharmonic Gaussian where
$\Omega=\mathring H^{-s}_g$ and $h=\mathrm{Id}$.

%
%
Our definition implies that a co-polyharmonic Gaussian field $h$ is \emph{grounded}, in the sense that $\langle h|c\rangle_g = 0$ for all constant $c$.
%

 \begin{remark}  An $\mathring H^{-\varepsilon}_g$-valued centered Gaussian random field~$h$ is a co-polyharmonic Gaussian field on~$(\M,\g)$  if and only if $\xi\coloneqq \sqrt{\mathsf p_\g}\,h$ is a \emph{grounded white noise on~$(\M,\g)$}, i.e., a $\mathring H^{-n/2-\varepsilon}_g$-valued centered Gaussian random field with covariance
\begin{equation}\label{cov-white}
  \mathbf{E} \big[\langle{\xi}|{u}\rangle_g \langle{\xi}|{v}\rangle_g\big]=\langle{u}|{v}\rangle_{\mathring L^2_g}   \quad\qquad \forall u,v\in \mathring H^{n/2+\varepsilon}_g \fstop
\end{equation}
Vice versa, given any grounded white noise $\xi$ on $(\M,\g)$, then $h\coloneqq \sqrt{{\mathsf{k}}_g}\,\xi$ is a co-polyharmonic Gaussian field on~$(\M,\g)$.
\end{remark}


\begin{remark}[cf.~ Proposition \ref{t:approx-cph}]\label{rem-cgf}
 Let $(M,g)$ be an admissible manifold and $h$ an $\mathring H^{-\varepsilon}_g$-valued  co-polyharmonic Gaussian field on it, defined on some probability space $(\Omega, \mathfrak F, \mathbf P)$.
%
%
Then in analogy to the definition of the It\={o} integral and in view of \eqref{eq:covariance-chp}, the mapping $\langle h|\emparg\rangle_g \colon \mathring H^{n/2}_g \to L^{2}(\mathbf P)$ extends to a linear isometry 
\begin{equation}\boldsymbol{\langle\!\!\langle} h|\emparg\boldsymbol{\rangle\!\!\rangle}_g \colon\mathring H^{-n/2}_g \to L^{2}(\mathbf P)\end{equation}
in the spirit of It\={o}'s $L^2$-isometry.
\end{remark}

%
The heuristic characterization \eqref{heur-gauss} of the measure $\mathsf{CGF}_\g$ manifests itself in  various important properties.
  As every Gaussian measure, $\mathsf{CGF}_\g$ satisfies a \emph{large deviation principle} whose rate function is given by the Cameron--Martin norm \cite[Chap.~II, Prop.~1.5 and Thm.~1.6]{Azencott}.
  In our case, this yields 
  \begin{proposition}
    For every co-polyharmonic field $h$, and for every Borel set $A\subset \mathring H^{-\varepsilon}_g$:
  \begin{equation*}
    \begin{split}
      -\inf_{u\in A^0}\mathfrak{p}_\g(u) &\le \liminf_{\beta\to 0}2\beta^2 \, {\mathbf P}[ \beta h\in A] \\
                                               & \le \limsup_{\beta\to 0}2\beta^2 \, {\mathbf P}[ \beta h\in A]\le-\inf_{u\in \bar A} \mathfrak{p}_{g}(u).
    \end{split}
\end{equation*}
Here $A^0$ and $\bar A$ respectively denote the interior and the closure of $A$ in the topology of~$\mathring H^{-\epsilon}_g$ for given $\epsilon>0$, the functional~$\mathfrak{p}_g$ is defined in~\eqref{e:p-form} and we set $\mathfrak{p}_g(u) = \infty$ if $u \not\in H^{n/2}_g$.
%
\end{proposition}

Next we recall the celebrated \emph{change of variable formula of Girsanov type}, also known as Cameron--Martin theorem, see for instance \cite[Theorem 14.1]{Janson}.
\begin{proposition}\label{girsanov}
  If $\varphi\in \mathring H^{n/2}_g$ and $h\sim \mathsf{CGF}_\g$, 
then $h+\varphi$ is distributed according to 
\begin{equation*}
  \exp\left(
  \boldsymbol{\langle\!\!\langle} h|\mathsf p_g\varphi\boldsymbol{\rangle\!\!\rangle}_g
 -\frac{1}{2} \mathfrak{p}_g(\varphi,\varphi) \right) \, d\mathsf{CGF}_\g(h).
\end{equation*}
For $\varphi\in \mathring H^{s}_g$ with $s>n$, the $L^2(\mathbf P)$-random variable $\boldsymbol{\langle\!\!\langle} h|\mathsf p_g\varphi\boldsymbol{\rangle\!\!\rangle}_g$ can be equivalently replaced by the usual pairing $\langle h|\mathsf p_g\varphi\rangle_g$.
\end{proposition}

\begin{remark} \label{positive-k}
Many of our subsequent results rely on the seminal work of J.-P.~Kahane~\cite{Kah85} on Gaussian multiplicative chaos.
His results apply to Gaussian random fields~$\tilde h$ on a metric space $(M,d)$ with covariance kernel $\tilde k$ with a logarithmic divergence: $|\tilde k(x,y)+\log{d(x,y)}|\le C$.
In addition to non-negative definiteness, he assumes that~$\tilde k$ is non-negative. Of course, this is not satisfied by our kernel~$k_g$.
However, as we are going to explain now, it  imposes no serious obstacle to applying his results in our setting.

Given the kernel $k_g$ as defined above, observe that it is smooth outside the diagonal and positive in the neighborhood of the diagonal.  Define a new kernel by 
\[
\bar k(x,y)\coloneqq k_g(x,y)+C\ge0
\]
with  $C\coloneqq -\min_{x,y\in M}k_g(x,y)<\infty$.
By construction $\bar k$ is non-negative.
Furthermore, it is also non-negative definite since it is the covariance kernel for the Gaussian field 
\[
\bar h\coloneqq  h+\sqrt C\,\xi
\]
where $h$ denotes the co-polyharmonic Gaussian field associated with $k_g$, and $\xi$ denotes a standard Gaussian variable independent of $h$.
\end{remark}

\subsection{Approximations} 

  As anticipated, our goal is to construct the random measure $d\mu(x) = e^{h(x)} d\vol_g(x)$.
  Due to the non-smooth nature of $h$ this requires approximating $h$ by smooth fields (and properly renormalizing).
  Co-polyharmonic Gaussian Fields may be approximated in various ways, the random measure obtained being essentially independent on the choice of the approximation \cite{Sha16}.
  Here, we present a number of different approximations: through their expansion in terms of eigenfunctions of the normalized co-polyharmonic operator~$\mathsf p_\g$; by convolution with (smooth or non-smooth) functions; by a discretization procedure.

Let us first discuss the eigenfunctions approximation.
As before, we denote by~$(\psi_j)_{j\in\N_0}$ the complete $L^2$-orthonormal system consisting of eigenfunctions of~$\mathsf p_\g$, each with corresponding eigenvalue~$\nu_j$.
In addition, we consider a sequence $(\xi_j)_{j\in \N_1}$  of independent and identically distributed standard Gaussian variables.
For each~$\ell \in \mathbb{N}_{0}$, we define the random test function
\begin{align}\label{eq:approx-cph-eigenfunctions}
h_\ell(x)\coloneqq  \sum_{j=1}^\ell \frac1{\sqrt{\nu_j}}{\psi_j(x)\, \xi_j}, \qquad x \in M \fstop
\end{align}
The covariance of the random field $h_{\ell}$ is given by:
\begin{equation}\label{eq:approx-kernels-eigenfunctions}
k_{\ell}(x,y)\coloneqq  {\mathbf E}\Big[h_\ell(x)\, h_\ell(y)\Big]=\sum_{j=1}^\ell \frac1{\nu_j}{\psi_j(x)\, \psi_j(y)}\comma \qquad x, y \in M \fstop
\end{equation}
Our next result establishes that the random field $h_\ell$ converges to the random field $h$.

\begin{proposition}\label{t:approx-cph}
Let~$(\M,\g)$ be admissible and~$\seq{h_\ell}_{\ell\in \N}$ defined as above.
Then
\begin{enumerate}[$(i)$]

\item\label{i:approx-cph:wiener-space}
for all $\varepsilon > 0$, the field~$h_{\ell}$, regarded as a random element of $\mathring{H}_g^{-\varepsilon}$, converges as~$\ell \to \infty$ to a co-polyharmonic Gaussian field $h$ in $L^{2}(\mathbf{P})$ and $\mathbf{P}$-a.s.
In particular, $h \in \mathring\Hil_g^{-\varepsilon}$ $\mathbf{P}$-a.s. Moreover, $h\not\in L^2_g$ $\mathbf{P}$-a.s.

\item\label{i:approx-cph:ito-isometry}
for every $u\in \mathring H_g^{-n/2}$,
the sequence~$\seq{\langle{u}|{h_\ell}\rangle_{g}}_{\ell\in \N}$ is a centered, $L^2$-bounded martingale on~$(\Omega, \mathfrak{F},\mathbf P)$ converging to 
$\boldsymbol{\langle\!\!\langle} h|u\boldsymbol{\rangle\!\!\rangle}_g$ 
 $\mathbf{P}$-a.s.\ and in $L^2(\mathbf{P})$ as $\ell\to\infty$, cf.\ Remark~\ref{rem-cgf}. 
 In this sense,
 $\mathbf{P}$-a.s.\begin{align*}
\boldsymbol{\langle\!\!\langle} h|u\boldsymbol{\rangle\!\!\rangle}_g 
= \sum_{j=1}^\infty \frac1{\sqrt{\nu_j}}\langle{u}|{\psi_j}\rangle_g\, \xi_j\fstop
\end{align*}
\end{enumerate}
\end{proposition}

\begin{proof}
  The proof follows from the abstract construction of \cite{Gross} and Definition \ref{abstract-wiener-space}.
  For completeness, we outline a simple proof in our setting.

\ref{i:approx-cph:wiener-space}
Let $\ell$ and $p \in \mathbb{N}$, and $\varepsilon>0$.
According to Lemma~\ref{l:PropertiesCoPolyH}~\ref{i:l:PropertiesCoPolyH:5}, we have that, $\mathbf{P}$-almost surely
\begin{equation*}
\norm{\sum_{j=\ell+1}^{p} \frac{\psi_{j} \xi_{j}}{\sqrt{\nu_{j}}} }_{\mathring{H}^{-\varepsilon}}^2 = \sum_{j=\ell+1}^{p} \frac{\xi_{j}^2}{\nu_{j}^{1+2\varepsilon/n}} \simeq \sum_{j=\ell+1}^{p} \frac{\xi_{j}^2}{j^{1+2\varepsilon/n}} \fstop
\end{equation*}
The sum on the right-hand side is a generalized chi-square random variable with variance $\sum_{j=\ell+1}^{p} j^{-2\varepsilon/n - 1}$.
It converges as~$p\to\infty$ if and only if $\varepsilon > 0$.
This shows that the series
\begin{equation}\label{eq:p:approx-cph:1}
  h \coloneqq  \sum_{j=1}^{\infty} \frac{\psi_{j}\, \xi_{j}}{\sqrt{\nu_{j}}} \comma
\end{equation}
exists $\mathbf{P}$-almost surely in $\mathring H_g^{-\varepsilon}$ but $h\not\in L^2_g$.
The proof of the convergence in $L^{2}(\mathbf{P})$ is carried out in the same way.
Since~$h$ is an $L^{2}(\mathbf{P})$-limit of Gaussian fields, it is itself Gaussian.
For~$u,v \in \mathring H^{n/2}_g$ its covariance is given by
\begin{equation}\label{eq:p:approx-cph:2}
  \mathbf{E}\tbraket{\langle{h}|{u}\rangle_g \langle{h}|{v}\rangle_g} =  \sum_{j=1}^{\infty} \frac1{\nu_{j}} {\langle{\psi_{j}}|{u}\rangle_{L_g^2} \langle{\psi_{j}}|{v}\rangle_{L_g^2}}= \mathcal{K}(u,v) \fstop
\end{equation}

\ref{i:approx-cph:ito-isometry}
  For all $u \in \mathring{H}^{-n/2}_g$, the sequence~$\seq{\langle{u}|{h_{\ell}}\rangle_g}_{\ell\in\mathbb{N}}$ is a martingale as a sum of independent and identically distributed random variables.
  Moreover, by orthogonality, 
\begin{equation*}
\sup_{\ell\in\N}\mathbf{E} \tbraket{\scalar{u}{h_{\ell}}_g^{2}} 
= \sum_{j=1}^{\infty} \frac1{\nu_j} {\langle{u}|{\psi_j}\rangle_{g}}^{2}= \mathcal{K}(u,u) = \norm{u}^{2}_{\mathring{H}^{-n/2}_g} < \infty \fstop
\end{equation*}
Thus, the martingale is $L^{2}(\mathbf{P})$-bounded and convergence follows from Doob's Martingale Convergence Theorem.
The limit is the requested 
$\boldsymbol{\langle\!\!\langle} h|u\boldsymbol{\rangle\!\!\rangle}_g 
\in L^2(\mathbf P)$.
\end{proof}

The previous result allows us to construct a co-polyharmonic Gaussian field 
on every probability space  that supports a sequence of independent and identically distributed standard normal variables.
It is also important to know that an approximation $h_\ell\to h$ 
as in the previous proposition 
holds for  every co-polyharmonic Gaussian field, independently of the construction of the latter.

\begin{remark}\label{finite-noise-appr}
Given any co-polyharmonic Gaussian field $h$, and the sequence of eigenfunctions $(\psi_j)_{j\in\N_0}$ as above, define a  sequence $(\xi_j)_{j\in\N}$ of independent and identically distributed standard normal variables by setting~$\xi_j\coloneqq \scalar{h}{\psi_j}_g$ for all~$j \in \mathbb{N}$, and 
  a sequence of  Gaussian random fields $(h_\ell)_{\ell\in\N}$ by
  \begin{equation}\label{eq:PartialSum2}
    h_\ell:\Omega\longrightarrow \mathring L^2_g\comma \qquad h^\omega_\ell(x)\coloneqq \sum_{j=1}^\ell \frac{\psi_j(x)}{\sqrt{ \nu_j}}\, \scalar{h^\omega}{\psi_j}_g \,.
  \end{equation}
  Then for every~$u\in \mathring H^{n/2}_g$,   as $\ell\to\infty$,
\begin{equation*}
\scalar{h_\ell}{u}_g\longmapsto \scalar{h}{u}_g\qquad\text{
$\mathbf P$-a.s.\ and in $L^2(\mathbf P)$}\fstop
\end{equation*}
\end{remark}

Now, let us consider more general approximations. The previous eigenfunction approximation will appear as  a particular case.

\begin{proposition}\label{convolution-approx} 
Let $h$ be a co-polyharmonic Gaussian field on $(M,g)$, and for each~$\ell\in\N$ let~$q_\ell\in  L^2(M^2,\vol_g \otimes \vol_g)$ be such that~$\mathsf{q}_\ell u\to u$ in~$L^2_g$ for all $u\in \mathring L_g^2$, where
  \[
  \mathsf{q}_\ell u(x)\coloneqq \scalar{q_\ell(x,\emparg)}{u}_{L^2_g} \fstop
  \]
  \begin{enumerate}[$(i)$]
  \item\label{i:t:Convolution-Approx:0} 
   Then for every  $\ell\in\N$, the field of functions~$h_\ell$ on~$M$ defined by
   \begin{equation}\label{eq:convolution-approx:1}
  h_\ell(y)=(\mathsf{q}_\ell^* h) (y)\coloneqq 
  \boldsymbol{\langle\!\!\langle} h|q_\ell(\emparg,y)\boldsymbol{\rangle\!\!\rangle}_g \end{equation}
   is a centered Gaussian field 
 with covariance function
 \begin{equation}\label{eq:convolution-approx:2}
k_\ell(x,y)= \tparen{(\mathsf{q}_\ell\otimes \mathsf{q}_\ell) K}(x,y)\coloneqq \iint K(x',y')\,q_\ell (x,x')\, q_\ell (y,y')\, d\vol_{g}(y')\,d\vol_{g}(x') \fstop
\end{equation}

  \item\label{i:t:ConvolutionApprox:1} 
  As $\ell\to\infty$, for  every~$u\in  \mathring L^2_g$,
\begin{equation}\label{alleswirdgut}
\scalar{h_\ell}{u}_{L^2_g}\longmapsto 
 \boldsymbol{\langle\!\!\langle} h|u\boldsymbol{\rangle\!\!\rangle}_g 
\qquad\text{
$\mathbf P$-a.s.\ and in $L^2(\mathbf P)$}\fstop
\end{equation}
\end{enumerate}
\end{proposition}

\begin{proof} \ref{i:t:Convolution-Approx:0} is obvious. To see \ref{i:t:ConvolutionApprox:1}, observe that $\scalar{h_\ell}{u}_{L^2_g}= \scalar{h}{\mathsf{q}_\ell  u}_g$ and  thus by \eqref{eq:covariance-chp} 
for every $u\in  \mathring L^2_g$, 
  \begin{equation*}
    \begin{split}
      {\mathbf E} \braket{\tabs{
      \boldsymbol{\langle\!\!\langle} h|u\boldsymbol{\rangle\!\!\rangle}_g 
    - \scalar{{h}_\ell}{u}_{L_g^2}}^2} = \
      {\mathbf E} \braket{\tabs{
      \boldsymbol{\langle\!\!\langle} h| u-\mathsf{q}_\ell  u \boldsymbol{\rangle\!\!\rangle}_g 
      }^2} \
                                                         = & \  \norm{u-\mathsf{q}_\ell u}_{\mathring H_g^{-n/2}}^2 \\
                                                                            \leq &\ C \, \norm{u -\mathsf{q}_{\ell} u}_{\mathring L_g^{2}}^2 \xrightarrow{\ \ell\to\infty\ } 0 
    \end{split}
\end{equation*}
where the last inequality follows from \ref{l:PropertiesCoPolyH}\ref{i:l:PropertiesCoPolyH:4aa}.
\end{proof}

\begin{example}\label{convolution-approx2} 
  \begin{enumerate}[$(i)$, wide]
  \item\label{i:t:Convolution-Approx:1} \emph{Probability kernels.}
Let~$\{ q_{\ell}(x,\emparg)\vol_g : \ell \in \mathbb{N}, x\in \M \}$ be a family of probability measures on~$\M$ with~$q_\ell\in L^\infty(\vol_g \otimes \vol_g)$ non-negative, and such that~$q_{\ell}(x,\emparg) \vol_{g}$ converges weakly to $\delta_{x}$ as $\ell \to \infty$ for each~$x\in\M$.
Then $\mathsf{q}_\ell u\to u$ in $L^2$ as $\ell\to\infty$ for all $u\in L^2$.
 
 Particular cases of \ref{i:t:Convolution-Approx:1} are \ref{i:t:Convolution-Approx:2} and \ref{i:t:Convolution-Approx:3} below.
 
 \item\label{i:t:Convolution-Approx:2}  \emph{Discretization.}
 Let $(\mathfrak P_\ell)_{\ell\in\N}$ be a family of Borel partitions of $\M$ with
 \[
 \sup\{\mathrm{diam}_g(A): A\in \mathfrak P_\ell\}\to0 \quad \text{as} \quad \ell\to\infty\fstop
 \]
 For $\ell\in\N$ put
\[
q_\ell(x,y)\coloneqq \sum_{A\in \mathfrak P_\ell} \frac{1}{\vol_g(A)} {\mathbf{1}}_A(x) \mathbf{1}_A(y) \fstop
\]
In other words, for given $x\in\M$ we have that $q_\ell(x,\emparg)=\frac1{\vol_g(A)}  {\mathbf{1}}_A$ with the unique ${A\in \mathfrak P_\ell}$ which contains $x$. Letting~$h_\ell$ be defined as in Proposition~\ref{convolution-approx} then yields
\[
h_\ell(x)=\frac1{\vol_g(A)} \scalar{h}{{\mathbf{1}}_A}_g\qquad  A\in \mathfrak P_\ell\comma x\in A \fstop
\]
This is a centered Gaussian random field $(h_\ell(x))_{x\in M}$ with covariance function
\begin{equation*}
k_{\ell}(x,y) = \frac{1}{\vol_{g}(A_{\ell}^{x})\, \vol_{g}(A_{\ell}^{y})} \int_{A_{\ell}^{x}} \int_{A_{\ell}^{y}} k(x', y')\, d\vol_{g}(x')\, d\vol_{g}(y')\comma
\end{equation*}
where $A_{\ell}^{x}$ is the unique element of $\mathfrak{P}_{\ell}$ containing $x$.

\item\label{i:t:Convolution-Approx:3} \emph{Heat kernel approximation.}
Let $q_\ell(x,y)\coloneqq p_{1/\ell}(x,y)$ be defined in terms of the heat kernel on $M$.
Then $\mathsf{q}_\ell u\to u$ in $L^2$  and thus in particular \eqref{alleswirdgut} holds for all $u\in L^2$. Even more,
\eqref{alleswirdgut} holds for all $u\in \mathring H^{-n/2}$.

\item\label{i:t:Convolution-Approx:4} \emph{Eigenfunctions approximation.} In terms of the eigenfunctions for the co-poly\-harmonic operator~$\Pol_\g$ we define
\[
q_\ell(x,y)\coloneqq \sum_{j=0}^\ell \psi_j(x)\, \psi_j(y)\fstop
\]
In other words,
$\mathsf{q}_\ell: L^2\to L^2$ is the projection onto the linear span of the first $1+\ell$ eigenfunctions.
Then 
$\mathsf{q}_\ell u\to u$  in $L^2$ as $\ell\to\infty$ for all $u\in L^2$.
\end{enumerate}
\end{example}

\begin{proof} \ref{i:t:Convolution-Approx:1}
Since $\mathcal C_b(M)$ is dense in $L^2(X)$ and since by Jensen's inequality
$\norm{\mathsf{q}_\ell u-\mathsf{q}_\ell v}_{L^2}\leq \norm{u- v}_{L^2}$,
it suffices to prove that $\mathsf{q}_\ell u\to u$ in $L^2$ as $\ell\to\infty$ for $u\in\mathcal C_b(M)$.
To see the latter, observe that $\mathsf{q}_\ell u(x)\to u(x)$ for each~$x$ by weak convergence of~$q_\ell(x,\emparg)\vol_g$ to~$\delta_x$, and that~$\norm{\mathsf{q}_\ell u}_{L^\infty}\le \norm{u}_{L^\infty}<\infty$.
\ref{i:t:Convolution-Approx:2} is straightforward.
\ref{i:t:Convolution-Approx:3} 
If $q_\ell=p_{1/\ell}$ and $u\in \mathring H^{-n/2}$ we have with $v\coloneqq \mathring{\mathsf{G}}^{n/4}u\in L^2$,
\begin{align*}
\norm{u-\mathsf{q}_\ell u}_{\mathring H^{-n/2}}\lesssim\norm{v-\mathsf{q}_\ell v}_{L^2}\xrightarrow{\ \ell\to\infty\ } 0
\end{align*}
where the last inequality follows from \ref{l:PropertiesCoPolyH}\ref{i:l:PropertiesCoPolyH:4aa}.

\ref{i:t:Convolution-Approx:4} Readily follows from the fact that~$(\psi_j)_{j\in\N_0}$ is a complete $L^2$-orthonormal system, Lemma~\ref{l:PropertiesCoPolyH}~\ref{i:l:PropertiesCoPolyH:6}.
\end{proof}

\subsection{Conformal quasi-invariance}

\begin{theorem}\label{t:ConformalChangeField}
Consider an admissible  Riemannian manifold $(\M,\g)$ and $\g'=e^{2\varphi}\g$ with~$\varphi\in\C^\infty(\M)$.
If $h$ is distributed according to   ${\mathsf{CGF}}_\g$ then
\begin{equation}\label{conf-field-formel}
h'\coloneqq h -\langle h\rangle_{g'}\end{equation}
is distributed according to ${\mathsf{CGF}}_{\g'}$. 
\end{theorem}

\begin{proof} 
By construction, $h$ is a centered Gaussian random field on $\mathring H^{-\varepsilon}_{g}$ for some/any $\varepsilon>0$. Let us choose $\varepsilon=n/2$.
According to Proposition \ref{prop:hn}, the random field $h'$ as defined above then is a centered Gaussian random field on $\mathring H^{-n/2}_{g'}$. Moreover, for all $u,v\in \mathring H_{g'}^{-n/2}$ (which 
implies $\ee^{n\varphi}u,\ee^{n\varphi}v \in \mathring H_{g}^{-n/2}$),
\begin{align*}
\mathbf E\Big[\scalar{h'}{u}_{g'}\cdot \scalar{h'}{v}_{g'}\Big]&\stackrel{\eqref{weighted-pairing}}=
\mathbf E\Big[\scalar{h'}{e^{n\varphi}u}_{g}\cdot \scalar{h'}{e^{n\varphi}v}_{g}\Big]\\
&
\stackrel{\eqref{conf-field-formel}}=\mathbf E\Big[\scalar{h}{e^{n\varphi}u}_{g}\cdot \scalar{h}{e^{n\varphi}v}_{g}\Big]\\
&\stackrel{\eqref{eq:covariance-chp}}=
\mathcal{K}_\g\big( \ee^{n\varphi}u,\ee^{n\varphi}v \big)
\stackrel{\eqref{conf-trafo3}}= \mathcal{K}_{g'}(u,v).
\end{align*}
Thus $h'$ shares the defining properties of the co-polyharmonic field on $(M,g')$. 
%
\end{proof}

The conformal quasi-invariance of the ${\mathsf{CGF}}_{M,g}$ indeed holds true in a more general form.
\begin{theorem}
Assume that $(\M,\g)$ and $(\M',\g')$ are conformally equivalent with diffeomorphism~$\Phi$ and conformal weight $e^{2\varphi}$ such that
$\Phi^*\g'=\ee^{2\varphi}\g$. Furthermore assume that  $h$ is distributed according to ${\mathsf{CGF}}_{M,g}$ and $h'$ is  distributed according to ${\mathsf{CGF}}_{\M',\g'}$. Then
\begin{equation}\label{h-conf-trna}
h'\stackrel{(\mathrm{d})}=\big(h -\langle h\rangle_{g'}\big)\circ \Phi^{-1}\fstop
\end{equation}
\end{theorem}

\begin{proof} Assume that $h$ is distributed according to ${\mathsf{CGF}}_{M,g}$. Then with $g^*\coloneqq e^{2\varphi}g$ by the previous Theorem \ref{t:ConformalChangeField} the field 
$$h^*:=h-\langle h\rangle_{g^*}$$
is distributed according to ${\mathsf{CGF}}_{M,g^*}$. Thus for the proof of the claim we may assume without restriction that $\varphi=0$ and $g^*=g, h^*= h$. In other words, assume that $(\M,\g)$ and $(\M',\g')$ are \emph{isometric} with diffeomorphism~$\Phi:M\to M'$  satisfying
$\Phi^* {g'}=g$ (`pull back of the metrics'). 
Then $\Phi_* \vol_{M,g}=\vol_{M',g'}$ (`push forward of the measurers') and, by the uniqueness of the co-polyharmonic operator as stated in Theorem \ref{Pol-basic}\ref{i:Pol-basic:6},
$$\mathsf P_{M',g'} u=  \mathsf P_{M,g} \big(u\circ \Phi\big)\circ\Phi^{-1}.$$
This invariance of the co-polyharmonic operators carries over to the kernels of their normalized inverse
$$k_{M',g'}\big(x',y'\big)=k_{M,g}\Big(\Phi^{-1}(x'),\Phi^{-1}(y')\Big)\qquad\forall x',y'\in M'$$
as well as to the associated bilinear forms
$$\mathcal  K_{M',g'}\big(u',v'\big)=\mathcal K_{M,g}\Big(u'\circ\Phi,v'\circ\Phi\Big)\qquad\forall u,v\in H^{-n/2}_{M',g'}.$$
Consequently, if $h$ is distributed according to ${\mathsf{CGF}}_{M,g}$ then $h'\coloneqq h\circ \Phi^{-1}$ is distributed according to ${\mathsf{CGF}}_{M',g'}$.
\end{proof}


To get rid of the additive correction term in \eqref{conf-field-formel}, one can consider the
`random variable' $h + a$, called  \emph{ungrounded co-polyharmonic Gaussian field},  where $h$ is distributed according to ${\mathsf{CGF}}^{M,g}$ and where $a$ is a constant 
distributed according to the Lebesgue measure on the line
(the latter not being a probability measure).

More formally, given any admissible manifold $(M,g)$, the distribution of the corresponding co-polyharmonic Gaussian field is a probability measure ${\boldsymbol\nu}_g$ on the grounded Sobolev space $\mathring H^{-\varepsilon}_g$. To override the influence of additive constants, we consider the (non-finite) measure $\widehat{\boldsymbol\nu}_g$  on   the (ungrounded) Sobolev space $H^{-\varepsilon}_g$ defined as the image measure of~${\boldsymbol\nu}_g\otimes\mathfrak{L}^1$ under the map
\begin{equation*}
(h,a)\mapsto h+a\fstop
\end{equation*}

\begin{definition} The measure  $\widehat{\boldsymbol\nu}_g$ is called  \emph{law of the ungrounded co-polyharmonic Gaussian field} and denoted by $\widehat{\mathsf{CGF}}_g$. 

We write~$\hbar\sim \widehat{\mathsf{CGF}}_g$ to indicate that a measurable map $\hbar: \Omega\to H^{-\varepsilon}_g$, defined on some measure space $(\Omega, {\mathfrak F}, \sf m)$, is distributed according to~$\widehat{\mathsf{CGF}}_g$, i.e., 
$\hbar_*{\sf m}=\widehat{\mathsf{CGF}}_g$.
\end{definition}

The conformal quasi-invariance of the probability measures ${\mathsf{CGF}}_{g}$ leads to a 
conformal invariance of the measures $\widehat{\mathsf{CGF}}_{g}$.

\begin{proposition} \label{CGF-factorized}
If $\hbar\sim\widehat{\mathsf{CGF}}_{g}$ and $\hbar'\sim\widehat{\mathsf{CGF}}_{g'}$ with $g'=e^{2\varphi}g$ then
\begin{equation*}
\hbar'\stackrel{(\mathrm{d})}=\hbar \fstop
\end{equation*}
\end{proposition}

\begin{proof} 
Let $h$ be a (grounded) co-polyharmonic field on $(M,g)$ and $h'\coloneqq h-\langle h\rangle_{g'}$. Then for all $F\in L^1(H^{-\varepsilon}_{g'}, \widehat{\boldsymbol\nu}_{g'})$, by  translation invariance of 1-dimensional Lebesgue measure and Theorem~\ref{t:ConformalChangeField},
\begin{align*}\int F\, d\widehat{\boldsymbol\nu}_{g'}=
 \int_\R\mathbf E'[F(h'+a)] \, da
                                                                        &=\int_\R \mathbf E[F(h-\av{h}_{g'}+a)]\,da
\\
                                                                          &=\int_\R\mathbf E[F(h+a)]\,da=\int F\, d\widehat{\boldsymbol\nu}_{g}\fstop \qedhere
\end{align*}
\end{proof}
\begin{corollary}\label{extend-field}
On each class $(M,[g])$ of conformally equivalent admissible 
Riemannian manifolds, $\widehat{\mathsf{CGF}}_{\M,\g}$ defines a conformally invariant  random field. 
\end{corollary}

The change of variable formula of Girsanov type of Proposition~\ref{girsanov} extends from shifts $\varphi\in\mathring H^{n/2}_g$ to shifts  $\varphi\in H^{n/2}_g$.

\begin{corollary}\label{girsanov2}
  If $\varphi\in H^{n/2}_g$ and $\hbar\sim \widehat{\mathsf{CGF}}_\g$, 
then $\hbar+\varphi$ is distributed according to 
\begin{equation*}
  \exp\left(
  \boldsymbol{\langle\!\!\langle} \hbar|\mathsf p_g\varphi\boldsymbol{\rangle\!\!\rangle}_g
 -\frac{1}{2} \mathfrak{p}_g(\varphi,\varphi) \right) \, d\widehat{\mathsf{CGF}}_\g(\hbar) \ .
\end{equation*}
\end{corollary}

\section{The Liouville Quantum Gravity measure}

Fix an admissible manifold $(M,g)$ and a co-polyharmonic Gaussian field
$h:\Omega\to \mathring H^{-\varepsilon}_g$,
Our naive goal is to study the `random geometry' on $M$
obtained by the random conformal transformation, 
\begin{equation*}
 \ee^{2h}g\comma
\end{equation*} 
and in particular to study the associated  `random volume measure' given as
\begin{equation}\label{liou-n1}
\ee^{n h(x)}d\vol_g(x) \fstop
\end{equation} 
It easily can be seen that --- due to the singular nature of the noise $h$ --- all approximating 
sequences of this measure diverge as long as no additional renormalization is built in. 

A more tractable goal is to study (for suitable $\gamma\in\R$) the random measure $\mu^{\gamma h}$ formally given as
\begin{equation}\label{liou-n2}
d\mu^{\gamma h}(x)=
e^{\gamma h(x)-\frac{\gamma^2}2{\mathbf E}[h(x)^2]}d \vol_\g(x).
\end{equation} 
Since $h$ is not a function but only a distribution, both \eqref{liou-n1} and  \eqref{liou-n2} are ill-defined. However,
replacing $h$ by its finite-dimensional noise approximation $h_{\ell}$ as constructed in Proposition~\ref{finite-noise-appr}, leads to a sequence $(\mu^{\gamma h})_\ell$ of random measures on $M$ which, as $\ell\to\infty$, almost surely, converges to a random measure $\mu^{\gamma h}$ on $M$, the \emph{plain Liouville Quantum Gravity measure} on the $n$-dimensional manifold $M$.
Let $\mathcal{M}_b(M)$ denote the set of finite positive Borel measures on $M$.
We equip it with the Borel $\sigma$-algebra associated with its usual \emph{weak topology}.

\subsection{Gaussian multiplicative chaos}
In the following Theorem, we construct the Gaussian multiplicative chaos~$\mu^{\gamma h}$ associated to a co-polyharmonic Gaussian field~$h$ on~$(\M,\g)$.
In view of Theorem \ref{t:approx-cph}, we can look at the co-polyharmonic Gaussian field $h$ as a random element of $H^{-\varepsilon}_g$ for any $\varepsilon > 0$.

\begin{theorem} \label{t:existence-liouville}
  Let an admissible manifold $(M,g)$ and a real number $\gamma$ with $|\gamma|<\sqrt{2n}$ be given as well as a $\mathring H^{-\varepsilon}_g$ co-polyhmarmonic Gaussian field $h$, defined on some probability space $(\Omega, \mathfrak F, \mathbf P)$.
  Then there exists a measurable map
  \begin{equation}\label{map-mu}
    \mu_g^{\gamma\bullet} \colon \mathring H^{-\varepsilon}_g \to \mathcal{M}_{b}(M), \qquad f \mapsto \mu_g^{\gamma f}\comma
  \end{equation}
  with the following properties:

  \begin{enumerate}[$(i)$]
    \item\label{i:liouville-Cameron--Martin-shift} for $\mathbf{P}$-a.e.~$h$ and every $\varphi \in \mathring H^{n/2}_g$,
      \begin{equation}\label{e:liouville-cm-shift}
        \mu_g^{\gamma(h+ \varphi)} = e^{\gamma\varphi}\, \mu_g^{\gamma h} \fstop
      \end{equation}

    \item\label{i:liouville-campbell} for all Borel measurable $f \colon \mathring H^{-\varepsilon}_g \times M \to [0,\infty]$, we have that
      \begin{equation}\label{e:liouville-campbell}
        \mathbf{E} \int f(h, x) d\mu_g^{\gamma h}(x) = \mathbf{E} \int f\tparen{h + \gamma k_g(x,\emparg), x} \, d\vol_g(x) \fstop
      \end{equation}

    \item\label{i:liouville-moments} for all $p \in \tparen{-\infty, \frac{2n}{\gamma^{2}}}$,
      \begin{equation*}
        \mathbf{E}\tbraket{ {\mu_g^{\gamma h}(M)}^{p} } < \infty \fstop
      \end{equation*}
  \end{enumerate}
\end{theorem}

\begin{remark} 
  \eqref{e:liouville-campbell} implies that $\mathbf{E}[\mu_g^{\gamma h}] = \vol_g$.
\end{remark}

 \begin{definition}
   The random measure $\mu_g^{\gamma h}$ is called the \emph{plain Liouville Quantum Gravity measure on~$(M,g)$}.
\end{definition}

\begin{proof}
The result follows from general results regarding the theory of Gaussian multiplicative chaos by Kahane \cite{Kah85} and Shamov \cite{Sha16}.
Shamov \cite{Sha16} gives an axiomatic definition of Gaussian multiplicative chaos and shows that the limit measure is in fact independent of the choice of approximating sequence.
In the language of \cite{Sha16}, our result follows from the existence of a \emph{sub-critical Gaussian multiplicative chaos over the Gaussian field $h$, identified with the mapping $\boldsymbol{\langle\!\!\langle} h|\emparg\boldsymbol{\rangle\!\!\rangle}_g \colon \mathring{H}_g^{-n/2} \to L^{2}(\boldsymbol{\nu}_g)$, and the operator $\mathsf{k}_g \colon \mathring{H}_g^{-n/2} \to \mathring{H}_g^{n/2} \subset L_g^{0}$}.
Properties~\ref{i:liouville-Cameron--Martin-shift} and \ref{i:liouville-campbell} being respectively \cite[Dfn.~11 (3)]{Sha16} and \cite[Thm.~4]{Sha16}.
The moments estimates \ref{i:liouville-moments} can be found in \cite[Thm.~4]{Kah85} for $p > 0$ and~\cite[Thm.~2.12]{RhoVar14} for~$p < 0$.
  
The existence of the Gaussian multiplicative chaos $\mu^{\gamma h}_g$ --- or, more precisely, the existence of the random variables $\int_M u\,d\mu^{\gamma h}_g$ as a limit of uniformly integrable martingales --- follows from an argument of \cite[Thm.\ 4 Variant 1]{Kah85}. This argument is stated in the slightly more restrictive setting of positive kernels. This restriction,  however, does not harm in our case. Indeed, passing from $h$ to $\hat h\coloneqq h+C\xi$ with some standard normal variable~$\xi$ independent of $h$ will change $k_g$ into $\hat k_g+C^2$ which is eventually (for sufficiently large $C$) a positive kernel.
  The corresponding random measures are then related to each other according to
\[
\hat\mu^{\gamma h}_g=\exp\tparen{\gamma C\xi-\tfrac{\gamma^2}{2} C^2}\, \mu^{\gamma h}_g \fstop \qedhere
\]
\end{proof}

\begin{remark}
Regarding uniform integrability and the existence of (plain) Liouville Quantum Gravity measure, the work~\cite{Ber17} provides an alternative approach based on the study of thick points of the underlying Gaussian fields.
\end{remark}

\subsection{Approximations}
Let us recall the content of \cite[Thm.\ 25]{Sha16} specified to our setting.

\begin{lemma}\label{sha-app}
Let~$q_{\ell}\in L^2(M^2,\vol_g\otimes\vol_g)$ be a family of kernels as in Proposition~\ref{convolution-approx} and let $(h_{\ell}(y))_{y\in M}$ be Gaussian fields as in~\eqref{eq:convolution-approx:1} with covariance kernel $k_{g,\ell}$ as in~\eqref{eq:convolution-approx:2}.
Further set
\begin{equation}\label{eq:approx-liouville}
  d\mu^{\gamma h_\ell}_{g}(x)\coloneqq 
  \exp\Big(\gamma h_\ell(x)-\frac{\gamma^2}2k_{g,\ell}(x,x)\Big)\, d \vol_{g}(x)\fstop
\end{equation}
Assume that
\begin{enumerate}[$(i)$]
  \item\label{i:AssumptionShamov:1} The family $(\mu^{h_{\ell}}(M))_{\ell\in\N}$ is uniformly integrable;
  \item\label{i:AssumptionShamov:2} For all $u \in \mathring H^{n/2}_g$, $\mathsf{q}_\ell u \to u $ in $L^{0}(M,\vol_g)$;
    \item\label{i:AssumptionShamov:3} $k_{g,\ell} \to k_g$ in $L^{0}(M^2,\vol_g \otimes \vol_g)$.\end{enumerate}
Then $\mu_g^{\gamma h_{\ell}} \to \mu_g^{\gamma h}$ weakly as Borel measures on $M$ in ${\mathbf P}$-probability as~$\ell\to\infty$. Even more, for every $u\in L^1(M,\vol_g)$,
\begin{equation}\label{prec-conv}
\int_M u\,d\mu_g^{\gamma h_{\ell}} \to  \int_M u\,d\mu_g^{\gamma h}\quad\text{in }L^1({\mathbf P})\quad \text{as}\quad \ell\to\infty\fstop
\end{equation}
\end{lemma}

In practice, the criteria~\ref{i:AssumptionShamov:2} and~\ref{i:AssumptionShamov:3} of the previous lemma are easy to verify. The remaining challenge is the verification of~\ref{i:AssumptionShamov:1}. 

\begin{lemma}\label{est-app} Assume that for every $\vartheta>1$ there exists $C\geq 0$, $\ell_\vartheta\in\N$, and a non-decreasing sequence $(c_\ell)_{\ell\in\N}$ such that for all $\ell\ge \ell_\vartheta$ and all $x,y\in M$,
\begin{equation}\label{cru-est}
k_{g,\ell}(x,y)\le \vartheta\log\left(\frac1{d(x,y)}\wedge c_\ell\right)+C\fstop
\end{equation}
Then for every $\gamma\in (0,\sqrt{2n})$, the family $(\mu_g^{\gamma h_{\ell}}(M))_{\ell\in\N}$ as in~\eqref{eq:approx-liouville} is uniformly integrable.
\end{lemma}

\begin{proof} This follows from Kahane's comparison lemma  \cite{Kah85}, cf.~\cite[Thm.s~27,~28]{Sha16}.
\end{proof}

In the rest of this section, let~$q_\ell$ be a family of probability kernels as in Example~\ref{convolution-approx2}~\ref{i:t:Convolution-Approx:1}.
The lemmas above allow us to obtain the following crucial approximation results.

\begin{theorem}\label{conv-appr}
Let the kernels $q_\ell$ be given in terms of a compactly supported, non-increasing function $\eta:\R_+\to\R_+$ as
\[
q_\ell(x,y)\coloneqq \frac1{N_\ell(x)}\eta\big(\ell\, d(x,y)\big)\comma \qquad N_\ell(x)\coloneqq \int_M \eta\big(\ell\, d(x,y)\big)\,d\vol_g(y)\fstop
\]
   Then with $\mu_g^{\gamma h_{\ell}}$ defined as in \eqref{eq:approx-liouville},
 \[
 \mu_g^{\gamma h_{\ell}} \to \mu_g^{\gamma h}\quad\text{as} \quad \ell\to\infty
 \]
  in the sense made precise in \eqref{prec-conv}.
\end{theorem}
\begin{remark} The assertion of the previous theorem holds as well for
\[
q_\ell(x,y)\coloneqq \frac1{N_\ell^*}\eta\big(\ell\,d(x,y)\big)
\]
with the `Euclidean normalization' $N^*_\ell\coloneqq \ell^{-n}\int_{\R^n} \eta(|y|)\,d\Leb^n(y)$ in the place of the `Riemannian normalization' $N_\ell(x)$.
\end{remark}

\begin{proof}
Assume that~$\eta$ is supported in $[0,R]$.
The verification of the criteria~\ref{i:AssumptionShamov:2} and~\ref{i:AssumptionShamov:3} in Lemma \ref{sha-app} is straightforward:
\ref{i:AssumptionShamov:2} was proven in Example~\ref{convolution-approx2}~\ref{i:t:Convolution-Approx:1}.
\ref{i:AssumptionShamov:3} follows from the fact that $k_g(x,\emparg)$ is bounded and continuous outside of any $\varepsilon$-neighborhood of~$x$, that~$q_\ell(x,\emparg)$ is supported in an $R/\ell$-neighborhood of $y$, and that $q_\ell(x,\emparg)\to \delta_x$ as $\ell\to\infty$.
Thus for every~$x,y $ with $d(x,y)\ge 2\varepsilon$ and every~$\ell\geq R/\varepsilon$,
\begin{align*}
k_{g,\ell}(x,y)=\int_{B_\varepsilon(y)}\int_{B_\varepsilon(x)} k_g(x',y')\, q_\ell(x',x)\, q_\ell(y',y)d\vol_g(x')\,d\vol_g(y') \xrightarrow{\ \ell\to\infty \ } k_g(x,y)\fstop
\end{align*}

Our verification of the criterion~\ref{i:AssumptionShamov:1} in Lemma \ref{sha-app} is based on Lemma~\ref{est-app}, the verification of which will in turn be based  on the following auxiliary results.

\begin{claim}\label{claim1} For all $x,y\in M$ with $d(x,y)\ge 3R/\ell$,
\[
\iint \log\frac1{d(x',y')}\,q_\ell(x,x')\,q_\ell(y,y')\,d\vol_g(x')\,d\vol_g(y')\le \log\frac{1}{d(x,y)}+\log 3\fstop
\]
\end{claim}
\begin{proof}
Combining the assumption $d(x,y)\ge 3R/\ell$ and the facts that $d(x,x')\le R/\ell$ for all~$x'$ in the support of $q_\ell(x,\emparg)$ and  $d(y,y')\le R/\ell$ for all $y'$ in the support of $q_\ell(y,\emparg)$, yields 
\[
d(x',y')\geq d(x,y)-d(x,x')-d(y,y')\geq d(x,y)- 2R/\ell\geq \frac{1}{3} d(x,y) \fstop
\]
Thus, the claim readily follows.
\end{proof}

\begin{claim}\label{claim2} For every $\vartheta>1$ there exist $\ell_\vartheta\in\N$ such that 
\[
\int \log\frac1{d(x,z)}\, q_\ell(y,z)\,d\vol_g(z)\leq \frac{\vartheta}{N_\ell^*} \int_{\R^n} \log\frac1{|x'-z|}\,\eta\big(\ell\, |y'-z|\big)\,d\Leb^n(z)+\vartheta\log\vartheta
\]
for all $\ell\geq \ell_\vartheta$, all~$x,y\in M$ with $d(x,y)<4R /\ell$, and all $x',y'\in\R^n$ with $d(x,y)=|x'-y'|$.
\end{claim}
\begin{proof}
Denote by~$\mathrm{inj}_g(M)>0$ the injectivity radius of~$(M,g)$.
For every~$y\in M$ set~$y'\coloneqq 0\in\R^n$ and use the exponential map~$\exp_y: \R^n\to M$ to identify $\varepsilon$-neighborhoods of $y\in M$ with $\varepsilon$-neighborhoods of $y'\in\R^n$ for all~$\varepsilon\in (0,\mathrm{inj}_g(M))$.
Since~$M$ is compact and smooth, for every~$\vartheta>1$ there exists~$\varepsilon_\vartheta\in (0,\mathrm{inj}_g(M))$ so small that~$\exp_y$ deforms both distances and volume elements in $\varepsilon$-neighborhoods of~$y$ by a factor less than~$\vartheta$, for every~$y\in M$ and every~$\varepsilon\in (0,\varepsilon_\vartheta)$.
Choose~$\ell_\vartheta$ so that~$4R/\ell_\vartheta<\varepsilon_\vartheta$.
Thus,
\begin{align*}
\int \log\frac1{d(x,z)} q_\ell(y,z)\,d\vol_g(z)&\le \frac{\vartheta}{N_\ell^*} \int_{\R^n} \log\frac{\vartheta}{|x'-z|}\eta\big(\ell\, |y'-z|\big)\,d\Leb^n(z)
\\
&= \frac{\vartheta}{N_\ell^*} \int_{\R^n} \log\frac{1}{|x'-z|}\eta\big(\ell\, |y'-z|\big)\,d\Leb^n(z)+\vartheta\log\vartheta\fstop
&&\qedhere
\end{align*}
\end{proof}

\begin{claim}\label{claim3} For all $x,y\in\R^n$,
\[
\int_{\R^n} \log\frac1{|x-z|}\,\eta\big(\ell\, |y-z|\big)\,d\Leb^n(z)\le  \int_{\R^n} \log\frac1{|z|}\,\eta\big(\ell\, |z|\big)\,d\Leb^n(z)\fstop
\]
\end{claim}

\begin{proof} Without restriction $y=0$ and $\ell=1$. For $r\ge0$, consider
\[
\phi(r)\coloneqq  \int_{\R^n} \log\frac1{|rx-z|}\,\eta(|z|)\,d\Leb^n(z) \fstop
\]
Then
\begin{align*}\phi'(r)&=
\int_{\R^n}\frac{\langle rx-z,x\rangle}{|rx-z|^2}\,\eta\big( |z|\big)\,d\Leb^n(z)=
\int_{\R^n}\frac{\langle z,x\rangle}{|z|^2}\,\eta\big( |z-rx|\big)\,d\Leb^n(z)\\
&=
\int_{\{z: \langle z,x\rangle\ge0\}}\frac{\langle z,x\rangle}{|z|^2}\,\Big(\eta\big( |z-rx|\big)-\eta\big( |z+rx|\big)
\Big)
\,d\Leb^n(z)\le0
\end{align*}
since $t\mapsto \eta(t)$ is non-increasing.
\end{proof}

\begin{claim}\label{claim4} There exists $C^*\geq 0$ such that, for all $\ell\in \N$,
$$ \frac{1}{N_\ell^*} \int_{\R^n} \log\frac1{|z|}\,\eta\big(\ell\, |z|\big)\,d\Leb^n(z)\le \log\ell+C^*\fstop$$
\end{claim}
\begin{proof} Straightforward with $C^*\coloneqq \frac{1}{N_1^*} \int_{\R^n} \log\frac1{|z|}\,\eta\big( |z|\big)\,d\Leb^n(z)$.
\end{proof}

Now let us conclude the \emph{proof of  Theorem} \ref{conv-appr}. 
Fix~$\vartheta>1$, and choose~$\ell_\vartheta$ as in  the proof of Claim~\ref{claim2}.
It remains to verify the estimate~\eqref{cru-est}.
For~$x,y$ with~$d(x,y)\ge 3R/\ell$, this is derived in Claim~\ref{claim1}.
For~$x,y$ with~$d(x,y)< 3R/\ell$ (hence~$d(x,y)<\varepsilon_\vartheta$), the Claims~\ref{claim2},~\ref{claim3},~\ref{claim4} yield
\begin{align*}
\int \log\frac1{d(x',y')}\, q_\ell(y,y')\,d\vol_g(y')\le \vartheta\big(\log\ell +C^*\big)
+\vartheta\log\vartheta
\end{align*}
for every $x'$ in the support of $q_\ell(x,\emparg)$, and thus 
  \begin{align*}
\iint \log\frac1{d(x',y')}\,q_\ell(x,x')\,q_\ell(y,y')\,d\vol_g(x')\,d\vol_g(y')\le \vartheta\big(\log\ell +C^*\big)
+\vartheta\log\vartheta\fstop
\end{align*}
This proves the estimate \eqref{cru-est} with $c_\ell\coloneqq \ell$ and~$C\coloneqq C^*\vartheta^2\log\vartheta$, and the proof of the theorem is herewith complete.
 \end{proof}
The previous results in particular applies to the kernel
$q_\ell(x,y)\coloneqq \frac{1}{\vol_g(B_{1/\ell}(x))} {\mathbf{1}}_{B_{1/\ell}(x)}(y)$. Similar arguments apply to discretization kernels.

\begin{theorem}
 Let $(\mathfrak P_\ell)_{\ell\in\N}$ be a family of partitions of $\M$ with 
 $d_\ell\coloneqq\sup\{\mathrm{diam}(A): A\in \mathfrak P_\ell\}\to0$ as $\ell\to\infty$, see Example \ref{convolution-approx2}~\ref{i:t:Convolution-Approx:2}, and $\inf\{\vol_g(A)/d_\ell^n: A\in \mathfrak P_\ell, \ell\in\N\}>0$.
Let
\[
q_\ell\coloneqq\sum_{A\in \mathfrak P_\ell} \frac{1}{\vol_g(A)}\, {\bf 1}_A\otimes {\bf 1}_A\fstop
\]
Then with $\mu_g^{\gamma h_{\ell}}$ defined as in \eqref{eq:approx-liouville},
\[
\mu_g^{\gamma h_{\ell}} \to \mu_g^{\gamma h}\quad\text{as }\ell\to\infty
\]
  in the sense made precise in~\eqref{prec-conv}.
 \end{theorem}

\begin{proof} 
Again the argumentation will be based on  Lemma \ref{sha-app}. The verification of  the criteria~\ref{i:AssumptionShamov:2} and~\ref{i:AssumptionShamov:3} there is again straightforward. Criterion~\ref{i:AssumptionShamov:1} will be verified as before by means of Lemma \ref{est-app}.
To verify \eqref{cru-est}, assume without restriction that $(\mathfrak P_\ell)_{\ell\in\N}$ is given with 
\[
\mathrm{diam}(A)\le d_\ell,\qquad \vol_g(A)\ge v_\ell\ge V d_\ell^n
\]
for all $A\in \mathfrak P_\ell,\, \ell\in\N$ and for some constant~$V>0$ independent of~$\ell$.
Then for $x,y\in M$ with $d(x,y)>3d_\ell$, we obtain as in Claim~\ref{claim1} that
\[
\iint \log\frac1{d(x',y')}\,q_\ell(x,x')\,q_\ell(y,y')\,d\vol_g(x')\,d\vol_g(y')\le \log\frac1{d(x,y)}+\log3\fstop
\]
Furthermore, for every $\vartheta>1$, every sufficiently large~$\ell$, and for all~$x,y\in M$ with $d(x,y)\le 3d_\ell$, we have that
\begin{align*}
\iint \log&\frac1{d(x',y')}\,q_\ell(x,x')\,q_\ell(y,y')\,d\vol_g(x')\,d\vol_g(y')\\
&\le\sup_{x'\in A_x}
\frac1{\vol_g(A_y)}\int_{A_y} \log\frac1{d(x',y')}\,d\vol_g(y')\\
&\le\sup_{x'\in \R^n}\sup_{\stackrel{A\subset \R^n}{\Leb^n(A)\le v_\ell}}
\frac{\vartheta}{\Leb^n(A)}\int_{A} \log\frac1{|x'-y'|}\,d\Leb^n(y')
\end{align*}
by comparison of Riemannian and Euclidean distances and volumes. 
Since~$\abs{x'-y'}$ is translation invariant, we may dispense with the supremum over~$x'$ and assume instead that~$x'=0\in\R^n$.
Furthermore,
\begin{align*}
\sup_{\stackrel{A\subset \R^n}{\Leb^n(A)\le v_\ell}}
\frac{1}{\Leb^n(A)}\int_{A} \log\frac1{|x'-y'|}\,d\Leb^n(y')=&\ \sup_{v\leq v_\ell} \sup_{\substack{A\subset \R^n\\ \Leb^n(A)=v}}\frac{1}{v} \int_{\R^n} \mathbf{1}_A(y') \log\frac{1}{\abs{y'}} d\Leb^n(y')
\\
\leq& \sup_{v\leq v_\ell} \frac{1}{v} \int_{B_r(0)} \log\frac{1}{\abs{y'}} d\Leb^n(y')
\end{align*}
by Hardy--Littlewood inequality and spherical symmetry of~$-\log\abs{y'}$, where~$r=r(v)$ is so that~$\Leb^n(B_r(0))=v$.
Furthermore, since~$\ell\mapsto v_\ell$ is monotone decreasing to~$0$, we may choose~$\ell$ additionally so large that~$v_\ell\leq 1$. For all such~$\ell$, since~$r(v)\leq r(v_\ell)\leq 1$ and $-\log\abs{y'}\geq 1$ on~$B_r(0)$, the function~$v\mapsto \frac{1}{v}\int_{B_{r(v)}(0)} \log \frac{1}{\abs{y'}}d \Leb^n(y')$ is increasing for~$v\in [0,v_\ell)$.
We have therefore that, for every $\vartheta>1$, every sufficiently large $\ell$, every~$x,y\in M$ with~$d(x,y)\le 3d_\ell$,
\begin{align*}
\iint \log\frac1{d(x',y')}\,q_\ell(x,x')\,q_\ell(y,y')\,d\vol_g(x')\,d\vol_g(y') \leq \frac{1}{v_\ell}\int_{B_r(0)} \log\frac{1}{\abs{y'}} d\Leb^n(y')
\end{align*}
with $r>0$ such that $\Leb^n(B_r(0))=v_\ell$.
For such~$r$ we may compute
\begin{align*}
\frac{1}{\Leb^n(B_r(0))}\int_{B_r(0)} \log\frac1{|y'|}\,d\Leb^n(y')
&= \frac{n}{r^n}\int_0^r \log \frac1s \, s^{n-1}ds=  \frac{1}{n\,r^n}\int_0^{r^n} \log \frac1t \, dt\\
&=\frac{1}{n\,r^n}r^n(1-\log r^n)=\frac1n+\log\frac1r \fstop
\end{align*}
That is, for $d(x,y)\le 3d_\ell$ with sufficiently large $\ell$,
\begin{align*}
\iint \log&\frac1{d(x',y')}\,q_\ell(x,x')\,q_\ell(y,y')\,d\vol_g(x')\,d\vol_g(y')
\le\vartheta\,\left(\frac1n+\log\frac1r\right)
\le C+\vartheta\log\frac1{3 d_\ell}
\end{align*}
since
$r=\left({v_\ell}/{c_n}\right)^{1/n}\ge \left({V}/{c_n}\right)^{1/n}\, d_\ell$ with~$c_n=\Leb^n(B_1(0))$. 
Thus, summarizing, for all $x,y\in M$ and all sufficiently large $\ell$,
\begin{align*}
\iint \log&\frac1{d(x',y')}\, q_\ell(x,x')\, q_\ell(y,y')\,d\vol_g(x')\,d\vol_g(y')
\le\C+\vartheta\log\frac1{d(x,y)\vee 3 d_\ell}\fstop \qedhere
\end{align*}
\end{proof}

The previous theorems do \emph{not} apply to the kernels
\begin{equation*}
q_{\ell}\coloneqq \sum_{j=1}^{\ell} \psi_{j}\otimes \psi_{j}
\end{equation*}
for the eigenspace projections.
These kernels are not nonnegative and not supported on small balls, even for large~$\ell$.
Nevertheless, $\mu^{\gamma h}$ can also be obtained via eigenfunctions approximation according to our next result.

\begin{theorem}\label{t:EigenfunctionApproxGMC}
  Consider the eigenfunctions approximation $\seq{h_{\ell}}_{\ell \in \mathbb{N}}$ given in \eqref{eq:approx-cph-eigenfunctions} with covariance kernel $\seq{k_{g,\ell}}_{\ell \in \mathbb{N}}$ as in \eqref{eq:approx-kernels-eigenfunctions}.
  Let $\seq{\mu_g^{\gamma h_{\ell}}}_{\ell \in \mathbb{N}}$ be as in \eqref{eq:approx-liouville}.
  Then for all Borel~$B \subset M$
  \begin{equation*}
    \mathbf{E}\tbraket{ \mu_g^{\gamma h}(B) \,\big|\, \xi_{1}, \dots, \xi_{\ell} } = \mu_g^{\gamma h_{\ell}}(B) \fstop
  \end{equation*}
  In particular, $\seq{\mu_g^{\gamma h_{\ell}}(B)}_{\ell\in\N}$ is a uniformly integrable martingale {and thus
  $\mathbf P$-a.s. and in $L^1(\mathbf P)$
  \begin{equation}
  \mu_g^{\gamma h_{\ell}}(B)\to \mu_g^{\gamma h}(B)\qquad\text{as }\ell\to\infty\fstop
  \end{equation}
  }
\end{theorem}

\begin{proof}%
  Fix $\varepsilon > 0$ and consider the function $F \colon \mathcal{C}^{\infty}(M) \times \mathring{H}^{-\varepsilon}_{g}$
\begin{equation*}
    F(\Phi, \Psi) \coloneqq \int_{B} e^{\gamma \Phi(x)} d\mu_g^{\gamma \Psi}(x), \qquad \Phi\in\C^\infty(M), \Psi \in \mathring H_g^{-\varepsilon} \fstop
  \end{equation*}
Fix a Borel set $B \subset M$.
  Since $h_{\ell}$ is almost surely smooth, in view of~\eqref{e:liouville-cm-shift} we find that
  \begin{equation*}
    \mu_g^{\gamma h}(B) = F(h_{\ell}, h - h_{\ell}).
  \end{equation*}
  Again by \eqref{e:liouville-cm-shift}, $\mu_{g}^{\gamma(h-h_{\ell})}(dx) = e^{-\gamma h_{\ell}(x)} \mu^{\gamma h}_{g}(dx)$, and we see that, for $\Phi \in \mathcal{C}^{\infty}$ deterministic
  \begin{equation*}
    G(\Phi) \coloneqq \mathbf{E} F(\Phi, h - h_{\ell}) = \mathbf{E} \int_{B} e^{\gamma \Phi(x)} e^{-\gamma h_{\ell}(x)} d\mu_g^{\gamma h}(x) \fstop
  \end{equation*}
  Observing that $h_\ell(x)=\pi_{\ell,x}(h)$ and 
  \[
  k_{g,\ell}(x,y)=\pi_{\ell,x}\tparen{k_g(x,\emparg)}(y)\, , \qquad \pi_{\ell,x}(u)\coloneqq \sum_{j=1}^\ell \langle u|\psi_j\rangle_g \psi_j(x)\, , \quad u\in H^{-\varepsilon}_g\, ,
  \]
  putting
  $f(h, x) \coloneqq e^{\gamma \Phi(x)} e^{-\gamma \pi_{\ell,x}(h)}$, and
applying \eqref{e:liouville-campbell},  we find that
  \begin{align*}
G(\Phi) &= \mathbf{E}  \int_{B} f(h,x) d\mu^{\gamma h}_g(x)= \mathbf{E}  \int_{B} f(h+\gamma k_g(x,\emparg),x) d\vol_g(x)\\
&= \mathbf{E}  \int_{B} e^{\gamma \Phi(x)} e^{-\gamma h_{\ell}(x)  - \gamma^{2} k_{g,\ell}(x,x)} d\vol_g(x).
\end{align*}
    Using that $h_{\ell}(x)$ is Gaussian, we conclude that
  \begin{equation*}
    G(\Phi) = \int_{B} e^{\gamma \Phi(x)} e^{-\frac{\gamma^{2}}{2} k_{g,\ell}(x,x)} d\vol_g(x) \fstop
  \end{equation*}
  Since $h_{\ell}$ and $h - h_{\ell}$ are independent and $h_{\ell}$ is measurable with respect to $\xi_{1}, \dots, \xi_{\ell}$, we have that
  \begin{equation*}
    \mathbf{E}\tbraket{ \mu_g^{\gamma h}(B) \,\big|\, \xi_{1}, \dots, \xi_{\ell} } = \mathbf{E}\tbraket{ F(h_{\ell}, h - h_{\ell}) \,\big|\, \xi_{1}, \dots, \xi_{\ell} } = G(h_{\ell}) = \mu_g^{\gamma h_{\ell}}(B) \fstop \qedhere
  \end{equation*}
\end{proof}

\subsection{Conformal quasi-invariance}
  \begin{theorem}\label{q-in-meas}\label{t:invariance-lqg}
Assume that the Riemannian manifold $(\M,\g)$ is admissible and that $\g'=e^{2\varphi}\g$ with $\varphi\in\C^\infty(\M)$. 
For $\gamma\in (-\sqrt{2n},\sqrt{2n})$, let $\mu_g^{\gamma h}$ and $\mu_{g'}^{\gamma h'}$ denote the plain Liouville Quantum Gravity measures on~$(M,g)$ and $(M,g')$, resp., with $h\sim{\sf CGF}_{\M,\g}$ and $h'\sim {\sf CGF}_{\M,\g'}$.
Set $v'\coloneqq \vol_{g'}(M)$, and define a centered Gaussian random variable $\xi$ and 
a function $\bar\varphi\in\C^\infty(\M)$  by
\begin{equation}\label{def-bar-phi}
\xi\coloneqq \langle h\rangle_{g'}
\comma \qquad
\bar\varphi\coloneqq \frac2{v'}\,{\sf k}_g(e^{n\varphi})-\frac1{{v'}^2}\,\mathcal{K}_g(e^{n\varphi},e^{n\varphi})\,.
\end{equation}
Then
\begin{equation}
\mu_{g'}^{\gamma h'}
\ \stackrel{{\rm (d)}}=\
e^{-\gamma\xi+ \frac{\gamma^2}{2} \bar\varphi+n\varphi}\, \mu_g^{\gamma h} \fstop
\end{equation}
\end{theorem}

  Our formulation of the \emph{plain} Liouville quantum gravity measure is slightly different from the one usually considered in dimension $2$, 
  see Section \ref{liouville-measure} for more details.

\begin{proof} 
Let $h \sim \mathsf{CGF}_{\M,\g}$.
For $\ell\in\N$, let $h_{\ell}$ be the Gaussian random field defined by~\eqref{eq:approx-cph-eigenfunctions}, and define the random fields
\begin{align}\label{eq:t:ConformalField:0.5}
  h_{\ell}^{\prime}\coloneqq  h_{\ell}-\langle h_{\ell} \rangle_{g'} \,, \qquad h^{\prime} \coloneqq   h
-\av{h}_{g'}
'\,.
\end{align}

\paragraph{Random fields}
The convergence $h_{\ell}\to h$ for $\ell\to\infty$  as stated in Prop.~\ref{t:approx-cph} implies an analogous convergence $h_{\ell}^{\prime}\to h^{\prime}$.
More precisely, for every $u\in \mathring H^{n/2}_g=\mathring H^{n/2}_{g'}$, 
\begin{align*}
  \lim_{\ell\to\infty} \scalar{h_{\ell}}{u}_{g'}=\lim_{\ell\to\infty}  \scalar{h_{\ell}}{\ee^{n\varphi}u}_{g} = \scalar{h}{\ee^{n\varphi}u}_g = \scalar{ h}{u}_{g'},
\end{align*}
as well as $\lim_{\ell\to\infty} \langle h_{\ell} \rangle_{g'}=\av{h}_{g'}$\,, the convergences being 
$\mathbf P$-a.s.~and in $L^2(\mathbf P)$,
and thus
\begin{align}\label{eq:t:ConformalField:1}
\lim_{\ell\to\infty} \scalar{h^{\prime}_{\ell}}{u}_{g'}=\scalar{h^{\prime}}{u}_{g'},  \qquad \text{$\mathbf P$-a.s.~and in $L^2(\mathbf P)$} \fstop
\end{align}

Let us set~$k'_{g,\ell}(x,y)\coloneqq  {\mathbf E}\big[h^{\prime}_{\ell}(x)\, h^{\prime}_{\ell}(y)\big]$ and let us denote by~${\sf k}'_{g,\ell}$ the corresponding integral operator on~$L^2_{g'}$, namely
\begin{equation*}
({\sf k}_{g,\ell}' u)(x)\coloneqq  \int k'_{g,\ell}(x,y) u(y)d \vol_{\g'}(y)\comma \qquad u\in L^2_{\g'}\fstop
\end{equation*}
Then
\begin{align*}
  \lim_\ell \iint u(x)\, k'_{g,\ell}(x,y)\, v(y)\, d\vol_{\g'}^{\otimes 2}(x,y) &=  \lim_\ell {\mathbf E}\braket{\scalar{h^{\prime}_{\ell}}{u}_{g'}\, \scalar{h^{\prime}_{\ell} }{v}_{g'} }
\\
&= {\mathbf E}\braket{\scalar{h^{\prime}}{u}_{g'}\, \scalar{h^{\prime}}{v}_{g'}} \\
&=\iint u(x)\, k_{g'}(x,y)\, v(y) \,d\vol_{\g'}^{\otimes 2}(x,y) \,,
\end{align*}
where the first equality holds by definition of~$k'_{g,\ell}$, the second equality holds by~\eqref{eq:t:ConformalField:1}, and the third equality holds since~$h^{\prime}\sim  \mathsf{CGF}_{\M,\g'}$ by Theorem~\ref{t:ConformalChangeField}.
In particular, we have the following convergences
\begin{align}
  \label{eq:Shamov2} & \lim_\ell \sqrt{{\sf k}'_{g,\ell}}\, u = \sqrt{{\sf k}_{g'}} u\,, \qquad \forall u\in L^2_{\g'} \comma
\\
\label{eq:Shamov3} & \lim_\ell k'_{g,\ell} = k_{g'}, \qquad \text{a.e.~on }M \times M.
\end{align}

\smallskip

\paragraph{Random measures}
Now, let us set, for all $\ell \in \mathbb{N}$:
\begin{align*}
\mu_g^{\gamma h_{\ell}} 
\coloneqq  \ee^{\gamma h_{\ell}-\frac{\gamma^2}{2}{\mathbf E}[h_{\ell}^2]} \, \vol_\g \comma
\qquad \text{resp.}\qquad 
\mu^{\gamma h'_\ell}_{g'}
\coloneqq  \ee^{\gamma h^{\prime}_{\ell}-\frac{\gamma^2}{2}{\mathbf E}[(h^{\prime}_{\ell})^2]} \, \vol_{\g'}\,.
\end{align*}
On the one hand, by Theorem~\ref{t:existence-liouville} we have that
\begin{align}\label{eq:t:ConformalField:2}
\lim_\ell \int u \,d\mu_g^{\gamma h_{\ell}} = \int u \,d \mu_g^{\gamma h} \comma \qquad u\in\C(\M)\comma
\end{align}
in~$L^1(\mathbf P)$.
On the other hand,
similarly to the proof of Theorem~\ref{t:existence-liouville}, the martingale~$\{\mu_{g'}^{\gamma h^{\prime}_{\ell}}(\M) : \ell \in \N\}$ is uniformly integrable.
Together with~\eqref{eq:Shamov2} and \eqref{eq:Shamov3}, this verifies the assumptions in~\cite[Thm.~25]{Sha16}, hence
 \begin{align}\label{eq:t:ConformalField:3}
   \lim_\ell \int u \,d\mu_{g'}^{\gamma h^{\prime}_{\ell}} = \int u \,d\mu_{g'}^{\gamma h'} \comma \qquad u\in\C_b(\M)\comma
 \end{align}
  in~$L^1(\mathbf P)$.

\smallskip

\paragraph{Radon--Nikodym derivative}
Similarly to Theorem~\ref{t:Covariant}, we can compute~$k'_{g,\ell}$ explicitly.
For short write $\mathsf{m}_{g'} = \vol_{g'} / v'$.
Then we have
\begin{align*}
k'_{g,\ell}(x,y)\coloneqq  &\ {\mathbf E}\Big[h^{\prime}_{\ell}(x)\, h^{\prime}_{\ell}(y)\Big]=\ {\mathbf E}\Big[\big(h_{\ell}(x) -\langle h_{\ell}\rangle_{\g'}\big)\, \big(h_{\ell}(y) - \langle h_{\ell}\rangle_{\g'}\big)\Big]
\\
=&\ k_{g,\ell}(x,y) + \iint k_{g,\ell}(w,z)\, d\mathsf{m}_{\g'}^{\otimes 2}(w,z)
\\
&\ - \int k_{g,\ell}(x,z)\, d\mathsf{m}_{\g'}(z) - \int k_{g,\ell}(y,w)\, d\mathsf{m}_{\g'}(w)
\\
=&\ k_{g,\ell}(x,y)-\frac12\bar\varphi_\ell(x)-\frac12\bar\varphi_\ell(y),
\end{align*}
where we have set
\begin{align*}
  \bar\varphi_\ell(\emparg) & \coloneqq 2 \int k_{g,\ell}(\emparg,z) \, d\mathsf{m}_{\g'}(z)
  - \iint k_{g,\ell}(w,z)\, d\mathsf{m}_{\g'}^{\otimes 2}(w,z)
\\
                            &=\frac2{v'}\, {\sf k}_{g,\ell}(e^{n\varphi})-\frac1{{v'}^2}\,\mathcal{K}_{g,\ell}(e^{n\varphi},e^{n\varphi}).
\end{align*}
Thus in particular,
\begin{align*}
k'_{g,\ell}(x,x)-k_{g,\ell}(x,x)=&\ \bar\varphi_\ell(x).
\end{align*}
Furthermore, set $\xi_\ell\coloneqq \langle h_{\ell} \rangle_{g'}=\scalar{h_{\ell}}{e^{n\varphi}}_g$. 
Then almost surely:
\begin{equation*}
  \begin{split}
    \log\frac{d \mu_g^{\gamma h_{\ell}}}{d\mu_{g'}^{\gamma h^{\prime}_{\ell}}}(x) &=
    \gamma\, h_{\ell}(x)-\frac{\gamma^2}2{\mathbf E}\braket{{h_{\ell}(x)}^2}
    -\gamma\, h^{\prime}_{\ell}(x)+\frac{\gamma^2}2{\mathbf E}\braket{{h^{\prime}_{\ell}(x)}^2}-n\varphi(x)\\
    &=    \gamma\,\langle h_{\ell} \rangle_{g'} + \frac{\gamma^2}2{\mathbf E}\braket{h^{\prime}_{\ell}(x)^2-h_{\ell}(x)^2}-n\varphi(x)\\
    &=\gamma\,\xi_\ell+\frac{\gamma^2}2\big(k'_\ell(x,x)-k_\ell(x,x)\big)-n\varphi(x)\\
    &=\gamma\,\xi_\ell-\frac{\gamma^2}2 \bar\varphi_\ell(x)-n\varphi(x),
  \end{split}
\end{equation*}
and thus for every $u\in\C_b(M)$,
\begin{align}\label{mu-ell-repr}
\int_M u(x)\, d{\mu_g^{\gamma h_{\ell}^{\prime}}(x)}=\int_M \ee^{-\gamma\,\xi_\ell+\frac{\gamma^2}2 \bar\varphi_\ell(x)+n\varphi(x)}\,
u(x)\, d{\mu_{g'}^{\gamma h_{\ell}}(x)} \fstop
\end{align}
\smallskip

\paragraph{Convergence}
As $\ell\to\infty$, by~\eqref{eq:t:ConformalField:1} applied with $u=e^{n\bar\varphi}$, we have that~$\xi_\ell\to \xi$, $\mathbf{P}$-a.s.
Moreover,~$\bar\varphi_\ell\to \bar\varphi$ in $L^\infty(M,\vol_g)$ 
according to
Lemma \ref{l:PropertiesCoPolyH} \ref{i:l:PropertiesCoPolyH:8}.
Together with the representation formula~\eqref{mu-ell-repr} and the convergence obtained in~\eqref{eq:t:ConformalField:2} and~\eqref{eq:t:ConformalField:3}, this implies that, in $L^1(\mathbf{P})$,
\begin{align*}
\int_M u(x)\, d{\mu_{g'}^{\gamma h'}(x)}=&
\lim_{\ell\to\infty}\int_M u(x)\, d\mu_{g'}^{\gamma h_{\ell}'}(x)
\\
=& \lim_{\ell\to\infty}\int_M \ee^{-\gamma\,\xi_\ell+\frac{\gamma^2}2 \bar\varphi_\ell(x)+n\varphi(x)}\,
u(x)\, d\mu_g^{\gamma h_{\ell}}(x)
\\
=&\lim_{\ell\to\infty}\int_M \ee^{-\gamma\,\xi+\frac{\gamma^2}2  \bar\varphi(x)+n\varphi(x)}\,
u(x)\, d{\mu_g^{\gamma h_{\ell}}(x)}
\\
=&\int_M \ee^{-\gamma\,\xi+\frac{\gamma^2}2  \bar\varphi(x)+n\varphi(x)}\,
u(x)\, d{\mu_g^{\gamma h}(x)} \fstop
\end{align*}
This proves the claim.
\end{proof}

 \begin{corollary}\label{extend-conf-meas}
Assume that $(\M,\g)$ and $(\M',\g')$ are admissible and conformally equivalent with diffeomorphism $\Phi$ and conformal weight $e^{2\varphi}$.
Let $h$ and $h_{\prime}$ denote the  co-polyharmonic random fields, and
 $\mu_g^{\gamma h}$ and $\mu_{g'}^{\gamma h'}$  the corresponding plain Liouville Quantum Gravity measures on $(\M,\g)$ and $(\M',\g')$, resp.
Then
\begin{equation}
\mu_{g'}^{\gamma h'}
\ \stackrel{{\rm (d)}}=\
\Phi_*\Big(e^{-\gamma\xi+\frac{\gamma^2}2\bar\varphi+n\varphi}\, \mu_g^{\gamma h}\Big)
\end{equation}
with $\xi$ and  $\bar\varphi$  as above.
\end{corollary}
 
As for the co-polyharmonic Gaussian field, the conformal quasi-invariance simplifies whenever we consider plain Liouville quantum gravity measures constructed from ungrounded fields. 
To this end, we use the identification $H^{-\varepsilon}_g \simeq \mathring H^{-\varepsilon}_g \oplus \R$ with bijections $\hbar\mapsto (\pi_g(\hbar), \av{\hbar}_g)$, $(h,a)\mapsto h+a$,
and extend the definition of the map 
$\mathring H^{-\varepsilon}_g \to \mathcal{M}_{b}(M)$, $h \mapsto \mu_g^{\gamma h}$
from  \eqref{map-mu}
to a map $
H^{-\varepsilon}_g \to \mathcal{M}_{b}(M)$
by putting 
\begin{equation}
\mu_g^{\gamma (h+a)}\coloneqq e^{\gamma a}\, \mu_g^{\gamma  h}\fstop
\end{equation}

\begin{corollary}\label{t:invariance-lqg-lebesgue}
  Assume that $\hbar \sim \widehat{\mathsf{CGF}}_{g}$ and $\hbar' \sim \widehat{\mathsf{CGF}}_{g'}$, then
  \begin{equation*}
    {\mu}^{\gamma \hbar'}_{g'} \stackrel{(d)}= \ee^{n\varphi + \frac{\gamma^{2}}{2} \bar{\varphi}} \mu_g^{\gamma \hbar}.
  \end{equation*}
Even more, for all measurable $F \colon H^{-\varepsilon}_g \times\mathcal{M}_{b}(M) \to \mathbb{R}_{+}$:
  \begin{equation*}
    \int F\Big(\hbar',\,  \mu_{g'}^{\gamma \hbar'}\Big) d\widehat{\mathsf{CGF}}_{g'}(\hbar') =
    \int F\Big(\hbar, \, \ee^{n \varphi + \frac{\gamma^{2}}{2} \bar{\varphi}} \mu_g^{\gamma \hbar}\Big)\, d\widehat{\mathsf{CGF}}_{g}(\hbar) \fstop
  \end{equation*}
\end{corollary}
\begin{proof}
  Expanding the definition of $\widehat{\mathsf{CGF}}$ and using Theorems \ref{t:ConformalChangeField} and \ref{t:invariance-lqg}, we find
  \begin{equation*}
    \begin{split}
      \int F\Big(\hbar',\, \mu_{g'}^{\gamma \hbar'}\Big) &\,d\widehat{\mathsf{CGF}}_{g'}(\hbar') = \int F\Big(h' + a, \, \ee^{\gamma a}\mu_{g'}^{\gamma h'}\Big)\, d a \, d\mathsf{CGF}_{g'}(h')
      \\
=& \int F\Big(h - \xi  + a, \,
\ee^{\gamma(a - \xi)} \ee^{\frac{\gamma^{2}}{2} \bar{\varphi} + n \varphi} \mu_g^{\gamma h}\Big)\, da \, d\mathsf{CGF}_{g}(h) \fstop
    \end{split} 
  \end{equation*}
  We conclude by the translation invariance of the Lebesgue measure.
\end{proof}

 \subsection{Liouville Quantum Gravity measure}\label{liouville-measure}

Recall that $k_{g}$ is 
the kernel 
of the inverse of the normalized co-polyharmonic operator $\mathsf{p}_{g}$.
Now we propose a further additive normalization in terms of the function
 \begin{equation*}
    {r}_{g}(x) =\limsup_{y\to x}\bigg[ k_{g}(x,y) - \log \frac{1}{d_g(x,y)}\bigg], \qquad \forall x \in M.
  \end{equation*}
  This function has an important quasi-invariance property under conformal changes.
  \begin{lemma}\label{t:invariance-r}
  Let $\varphi$ smooth and $g' = \ee^{2 \varphi}g$.
  Then with the notation of Theorem \ref{t:invariance-lqg},
  \begin{equation*}
    r_{g'} - r_{g} = -\bar\varphi+ \varphi \fstop
  \end{equation*}
\end{lemma}
\begin{proof}
  By Proposition \ref{t:Covariant}, for $x \ne y \in M$:
  \begin{align*}
   \bigg[ k_{g'}(x,y) + \log d_{g'}(x,y)\bigg]&-\bigg[ k_{g}(x,y)+\log d_g(x,y)\bigg]
   \\
   &= - \frac12\bar\varphi(x) - \frac12\bar\varphi(y)  + \log d_{g'}(x,y)-\log d_g(x,y)\fstop
  \end{align*}
  Thus the claim is obtained immediately by letting $y \to x$, and noting that
\begin{equation*}
\frac{d_{g'}(x,y)}{d_g(x,y)}\longrightarrow e^{\varphi(x)} \qquad \text{as } y\longrightarrow x\fstop \qedhere
\end{equation*}
\end{proof}

\begin{definition}
  We define the \emph{Liouville Quantum Gravity measure} (also called  \emph{adjusted Liouville Quantum Gravity measure}) by
  \begin{equation*}
    {\bar\mu}^{\gamma h}_g = \exp\Big(\frac{\gamma^{2}}2\, r_{g}\Big)\, \mu^{\gamma h}_g.
  \end{equation*}
\end{definition}

\begin{theorem}\label{t:invariance-adjusted}
  Let $\varphi$ be smooth, $g' = \ee^{2\varphi} g$, and $h$ and $h'$ co-polyharmonic Gaussian fields with respect to $g$ and $g'$.
  Then
  \begin{equation*}
    {\bar\mu}_{g'}^{\gamma h'} = \exp\braket{-\gamma \xi + \tparen{n + \tfrac{\gamma^{2}}{2}} \varphi} {\bar\mu}_g^{\gamma h},
  \end{equation*}
  where $\xi =\av{h}_{g'}$.
\end{theorem}

\begin{proof}
This is a direct consequence of Theorem \ref{q-in-meas} and Lemma \ref{t:invariance-r}.
\end{proof}

As before, in order to get rid of the Gaussian random variable $\xi$ in the conformal quasi-invariance formulation, we need to consider the law induced by the ungrounded co-polyharmonic field, cf.\ Corollary \ref{t:invariance-lqg-lebesgue}.

\begin{corollary}\label{t:invariance-lqg-lebesgue2}
  Assume that $\hbar\sim \widehat{\mathsf{CGF}}_{g}$ and $\hbar' \sim \widehat{\mathsf{CGF}}_{g'}$, then
  \begin{equation*}
    {\bar\mu}_{g'}^{\gamma \hbar'} \stackrel{(d)}= \ee^{(n+\frac{\gamma^2}{2})\varphi} \bar\mu_{g}^{\gamma \hbar}\comma
  \end{equation*}
or, in other words,
    \begin{equation*}
    {\bar\mu}_{g'}^{\gamma \hbar'} \stackrel{(d)}=  \bar\mu_{g}^{\gamma \, T(\hbar)}\comma
  \end{equation*}
with the shift $T:\hbar\mapsto \hbar+\left(\frac n\gamma +\frac\gamma2\right)\varphi$.

Even more, for all measurable $F \colon H^{-\varepsilon}_g \times\mathcal{M}_{b}(M) \to \mathbb{R}_{+}$:
  \begin{equation*}
    \int F\Big(\hbar',\,  \bar\mu_{g'}^{\gamma \hbar'}\Big) d\widehat{\mathsf{CGF}}_{g'}(\hbar') =
    \int F\Big(\hbar, \, \ee^{(n + \frac{\gamma^{2}}{2}){\varphi}} \bar\mu_g^{\gamma \hbar}\Big)\, d\widehat{\mathsf{CGF}}_{g}(\hbar) \fstop
  \end{equation*}
\end{corollary}

\begin{remark}
 Let us assume for the sake of discussion that the function $r_g$ is smooth.
This is known to be true in the case $n=2$; in arbitrary even dimension, according to~\cite[Lem.~2.1]{Ndiaye}, the function~$r_g$ is at least $\mathcal C^2$.
 With respect to the smooth function~$r_g$, we define the \emph{refined co-polyharmonic kernel}  by
\begin{equation*}
\tilde k_g(x,y)\coloneqq k_g(x,y)-\frac12r_g(x)-\frac12r_g(y)+c_g\comma
\end{equation*}
where $c_g\coloneqq \langle r_g\rangle_g+\frac14 \mathfrak{p}_g(r_g,r_g)$.
With $k_g$ it shares the estimate \eqref{log-div}, and in addition it satisfies
\begin{equation}
\limsup_{y\to x}\bigg[\tilde k_g(x,y)-\log\frac1{d_g(x,y)}\bigg]=c_g\comma\qquad x\in M\fstop
\end{equation}
\begin{enumerate}[$(i)$]
\item
The kernel $\tilde k_g$ is the covariance kernel associated with the \emph{refined co-polyharmonic field} given by
\begin{equation*}
\tilde h_g\coloneqq h_g-\frac12\scalar {h_g}{\mathsf p_gr_g}_g\comma
\end{equation*}
where $h_g$ denotes the co-polyharmonic field as considered before.

\item
Let $\tilde k_{g'}$ and  $\tilde h_{g'}$ denote the refined kernel and refined field associated with the metric $g'=e^{2\varphi}g$ for some $\varphi\in\C^\infty(M)$. Then
\begin{equation}
\tilde k_{g'}(x,y)=\tilde k_g(x,y)+\frac12\varphi(x)-\frac12\varphi(y)+c_{g'}-c_g\comma
\end{equation}
and
\begin{equation}
\tilde h_{g'}\stackrel{(d)}= \tilde h_g-\frac12\ttscalar{\tilde h_g}{\mathsf p_g\varphi}_g\fstop 
\end{equation}

\item The LQG measure (aka Gaussian multiplicative chaos) on $(M,g)$ associated with the refined field $\tilde h_g$ is given in terms of  the plain  LQG measure associated with $h_g$ as 
\begin{align*}
\mu_{g}^{\gamma \tilde h}&=\exp\Big(\frac{\gamma^2}2(r_g-c_g)- \frac{\gamma}2\scalar h{\mathsf p_gr_g}_g\Big)\, \mu_{g}^{\gamma h}\fstop
\end{align*}

Passing from grounded polyharmonic fields $h\sim {\mathsf{CGF}}_{g}$ to ungrounded fields $\hbar=h+a\sim {\mathsf{CGF}}_{g}\otimes \mathfrak L^1$ and making use of the translation invariance of $\mathfrak L^1$ on $\R$, the associated LGQ measures satisfy
\begin{align}
\mu_{g}^{\gamma \tilde \hbar}=e^{\gamma a}\, \mu_{g}^{\gamma \tilde h}&
\stackrel{(d)}=\exp\Big(\frac{\gamma^2}2r_g\Big)\, e^{\gamma a}\,\mu_{g}^{\gamma h}
=
e^{\gamma a}\,\bar\mu_{g}^{\gamma h}=\bar\mu_{g}^{\gamma \hbar}.
\end{align}
That is, the plain LGQ measure for the ungrounded refined co-polyharmonic field coincides in distribution with the adjusted LGQ measure for the ungrounded co-polyharmonic field.
\end{enumerate}
Working with the adjusted LGQ measure for the co-polyharmonic field (rather than with the plain LGQ measure for the refined co-polyharmonic field) allows us to avoid any smoothness assumption on $r_g$.
\end{remark}


\begin{proof}
(i)
  Straightforward calculations yield for  $u,v\in \mathring H_g^{n/2}$,
\begin{align*}
\iint\tilde k_g(x,y)\, & u(x)\, v(y)\, d\vol_g(x)\,d\vol_g(y)
\\
=&\ {\mathbf E}\tbraket{\ttscalar{\tilde h}{u}_g\cdot \ttscalar{\tilde h}{v}_g}
\\
=&\iint k_g(x,y)u(x)v(y)d\vol_g(x)\,d\vol_g(y)
\\
&-\frac12\int u\,d\vol_g\cdot {\mathcal K}_g(\mathsf p_gr_g,v) -\frac12\int v\,d\vol_g\cdot {\mathcal K}_g(\mathsf p_g r_g,u)
\\
&+\frac14\int u\,d\vol_g\cdot\int v\,d\vol_g\cdot {\mathcal K}_g(\mathsf p_gr_g,\mathsf p_gr_g)\fstop 
\end{align*}

(ii) Immediate consequences of  Lemma \ref{t:invariance-r}, Proposition \ref{conf-k}, and Theorem  \ref{t:invariance-lqg}.
\end{proof}

\begin{remark}
Most approaches to the Liouville Quantum Gravity measure in dimension 2 are formulated in terms of a (`background') metric tensor $g$ which is translation invariant.
 In these cases, $r_g\equiv const$ and thus plain and adjusted LQG measure coincide up to a multiplicative constant. Both measures  are obtained by regularizing $h$ via convolution and by normalizing $\ee^{\gamma h_{\varepsilon}(x)} \vol(dx)$ by some explicit power of $\varepsilon$, say $\varepsilon^{\gamma^2/2}$. This translation invariant re-normalization in terms of $\varepsilon^{\gamma^2/2}$ 
    is then also employed in the non-homogeneous case 
    (\cite{DKRV16}, sect.~2.3: ``We need that these cut-off approximations be
defined with respect to a fixed background metric: we consider Euclidean circle averages
of the field because they facilitate some computations \ldots'')
    whereas the `intrinsic Riemannian perspective' would suggest to re-normalize by
    $\exp\left(-\gamma^2/2 \, {\mathbf E} [h_{\varepsilon}(x)^2]\right)$.
 As observed  in \cite[Prop.~2.5]{DKRV16} and \cite[Lemma 3.2]{GuiRhoVar19}, in our notation
$${\mathbf E} [h_{\varepsilon}(x)^2]=-\log(\varepsilon)+r_g(x) +{\it o}(1),$$
and thus the LQG measure  constructed in \cite{DKRV16} and \cite{GuiRhoVar19} for arbitrary closed Riemannian surfaces coincides with the adjusted LQG measure in our sense.
\end{remark}

\section{Liouville Brownian motion and random GJMS operators}

\subsection{Support properties of LQG measures}

In the sequel, we will study support properties on $M$ for a.e. realization of LQG measures. Since the (adjusted) LQG measure $\bar\mu_g^{\gamma h}$ and the plain  LQG measure $\mu_g^{\gamma h}$ only differ by a deterministic (finite and positive) weight, these support properties will be the same for both of them. For convenience, we will state them for the plain LQG measure.

Since a typical realization of the plain Liouville Quantum Gravity measure $\mu_g^{\gamma h}$ is singular with respect to the volume measure of $M$, it gives positive mass to certain sets $E\subset M$ of vanishing volume measure. However, it does not give mass to sets of vanishing $\mathcal H^s$-capacity (for sufficiently large $s$), a classical scale of `smallness of sets' involving Green kernels and thus well-suited for our purpose.

\begin{definition}
For~$s>0$, the \emph{$\mathcal H^s$-capacity} (aka \emph{Bessel capacity}) of a  set~$E\subset \M$ is
\begin{equation}\label{eq:d:Capacity:0}
\text{\rm cap}_s(E)\coloneqq  \inf\Big\{\|f\|_{L^2}^2:\ {\sf G}_{s/2,1} f \geq 1 \ \vol_g\textrm{-a.e.\ on } E, \  f\geq 0\Big\} \fstop
\end{equation}
\end{definition}
A set with vanishing $\mathcal H^s$-capacity, also has vanishing $\mathcal H^r$-capacity for every $r\in (0,s)$,
We call a set $E$ such that $\mathrm{cap}_{s}(E) =0$, a \emph{$\mathrm{cap}_{s}$-zero} or a \emph{$\mathrm{cap}_{s}$-polar} set.

\begin{theorem}\label{t:PolarNegligible}
Consider the co-polyharmonic Gaussian field $h\sim {\sf CGF}_{\g}$ and the associated plain Liouville Quantum Gravity measure $\mu_g^{\gamma h}$  on $(M,g)$ with  $|\gamma|^2<2n$. Then for a.e.~$h$ and every $s> \gamma^2/4$, the measure $\mu^{\gamma h}_g$ does not charge  sets of vanishing $\mathcal H^s$-capacity. 
That is, 
\[
\text{\rm cap}_s(E)=0 \ \Longrightarrow \ \mu_g^{\gamma h}(E)=0\qquad\text{for every Borel } E\subset M\fstop
\]
\end{theorem}
For applications of this result in the remainder of this paper, two choices of $s$ are relevant, $s=n/2$ and $s=1$.
\begin{corollary}\label{cor:cap} Consider $h$ and $\mu_g^{\gamma h}$ as above. \begin{itemize}
\item
If $|\gamma|<\sqrt{2n}$, then $\mathbf P$-almost surely $\mu^{\gamma h}_g$ does not charge sets of vanishing $\mathcal H^{n/2}$-capacity. 

\item
If $|\gamma|<2$, then $\mathbf P$-almost surely $\mu^{\gamma h}_g$ does not charge sets of vanishing $\mathcal H^{1}$-capacity. 
\end{itemize}
In the particular case $n=2$, both assertions coincide. In general, none of the two assertions is an immediate consequence of the other one.
\end{corollary}

Our proof of the theorem relies on results on Bessel capacities and on a celebrated estimate 
for the volume of balls by J.-P.~Kahane for random measures defined in terms of covariance kernels with logarithmic divergence.

Concerning capacities, we adapt to manifolds results in~\cite{Zie89} that do not follow  from~\cite{DynKuz07}.
Denote by~${\mathcal M}_b(M)$ the space of non-negative finite Borel measures on~$\M$.
For~$\mu \in \mathcal{M}_b(M)$ we set
\begin{equation*}
\mathsf{G}_{s,\alpha} \mu(x) = \int G_{s,\alpha}(x,y) d\mu(y), \qquad  s, \alpha > 0\comma
\end{equation*}
and, for a measurable set $E \subset M$:
\begin{equation}\label{eq:MeasureCapacity}
b_s(E)\coloneqq \sup\Big\{\mu(E): \mu\in{\mathcal M}_b(M), \  \big\|{\sf G}_{s/2,1}(\mathbf{1}_E \mu)\big\|_{L^2}\leq 1 \Big\} \,.
\end{equation}
\begin{remark}
The Bessel capacities as defined above are `order~$1$ capacities' in the sense of Dirichlet forms.
The corresponding `order~$0$ capacities' would be defined by replacing the operator~${\sf G}_{s,1}$ with its grounded version~$\mathring{\sf G}_s$.
As a consequence of the compactness of~$\M$, these capacities define the same class of cap-zero subsets of~$\M$. 
\end{remark}

\begin{lemma} 
\label{Capacities}
Let $s > 0$.
The following assertions hold true:
\begin{enumerate}[$(i)$]
  \item\label{i:p:Capacities:1}
  $\mathrm{cap}_s$ is a regular Choquet capacity;
\item\label{i:p:Capacities:2} for every Suslin set~$E\subset \M$,
\begin{equation*}
b_s(E)^2=\text{\rm cap}_s(E)\,;
\end{equation*}
\item\label{i:p:Capacities:3} if~$\mu\in \mathcal{M}_b(\M)$ satisfies~$\|{\sf G}_{s/2,1}\mu\|_{L^2}<\infty$, then~$\mu$ does not charge $\text{\rm cap}_s$-zero sets;
\item\label{i:p:Capacities:4} any function in~$\mathcal H^s$ is pointwise determined (and finite) up to a $\mathrm{cap}_s$-zero set;
\item\label{i:p:Capacities:5} if~$(u_k)_k\subset \mathcal H^s$ and $u\in \mathcal H_g^s$ satisfy~$\lim_k\abs{u_k-u}_{\mathcal H^s}=0$, then there exists a subsequence~$(u_{k_j})_j\subset \mathcal H^s$ so that~$u=\lim_j u_{k_j}$ pointwise up to a $\mathrm{cap}_s$-zero set.
\end{enumerate}
\end{lemma}

\begin{proof}
Since assertions~\ref{i:p:Capacities:1} and~\ref{i:p:Capacities:2} above are set-theoretical in nature, their proof is adapted \emph{verbatim} from~\cite{Zie89}.
In particular,~\ref{i:p:Capacities:1} is concluded as in~\cite[Cor.~2.6.9]{Zie89}, and~\ref{i:p:Capacities:2} as in~\cite[Thm.~2.6.12]{Zie89}. 
These adaptations hold provided we substitute the operator~$g_\alpha\ast$ in~\cite{Zie89}, $\alpha=s/2$, with~$\mathsf{G}_{s/2,1}$, and noting that~$G_{s/2,1}(x,\emparg)$ is continuous away from~$x$.

In order to show~\ref{i:p:Capacities:3}, let~$E\subset \M$ be $\mathrm{cap}_s$-polar.
By standard facts on Choquet capacities,~$E$ can be covered by countably many Suslin $\mathrm{cap}_s$-polar sets. Thus, we may assume with no loss of generality that~$E$ be additionally Suslin.
By~\ref{i:p:Capacities:2} and definition~\eqref{eq:MeasureCapacity} of~$b_{s}$, we then have
\begin{align*}
\norm{\mathsf{G}_{s/2,1} \mu}_{L^2} \cdot \mathrm{cap}_s(E)^{1/2}\geq \norm{\mathsf{G}_{s/2,1} (\mathbf{1}_E \mu)}_{L^2} \cdot b_s(E)\geq \mu E\comma
\end{align*}
which concludes the proof by assumption on~$E$.

\ref{i:p:Capacities:4} By definition,~$u$ is in~$\mathcal H^s$ if and only if there exists~$v\in L^2$ such that~$u=\mathsf{G}_{s/2,1}v$, and~$\norm{u}_{\mathcal H^s}=\norm{v}_{L^2}$.
By density of~$\C^\infty(M)$ in both~$L^2$ and~$\mathcal H^s$, it suffices to show that, whenever~$u_n\to u$ $\vol_\g$-a.e.\ and in~$L^2$, then~$\mathsf{G}_{s/2,1} u_n\to\mathsf{G}_{s/2,1} u$ $\mathrm{cap}_s$-q.e..
This latter fact holds as in~\cite[Lem.~2.6.4]{Zie89}, with identical proof.

\ref{i:p:Capacities:5}
Firstly, let us show that
\begin{equation}\label{eq:p:Capacities:1}
\mathrm{cap}_s\tparen{\set{\abs{f}>a}}\leq \norm{f}_{\mathcal H^s}^2/a^2\comma \qquad f\in \mathcal H^s\comma a>0\fstop
\end{equation}
Indeed,
\begin{equation*}
\mathsf{G}_{s/2,1}(1-\Delta_\g)^{s/2} \abs{f}/a=\abs{f}/a\geq 1 \quad \text{ on } \set{\abs{f}>a}\comma
\end{equation*}
hence, by definition of~$\mathrm{cap}_s$ we have that
\begin{align*}
\mathrm{cap}_s\tparen{\set{\abs{f}>a}}\leq \norm{(1-\Delta_\g)^{s/2}\abs{f}/a}^2_{L^2}=\norm{\abs{f}/a}^2_{\mathcal H^s}=\norm{f}^2_{\mathcal H^s}/a^2\fstop
\end{align*}
Now, let~$\tseq{u_{k_j}}_j\subset \tseq{u_k}_k$ be so that~$\tnorm{u-u_{k_j}}_{\mathcal H^s}^2\leq 2^{-3j}$, and set $A_j\coloneqq  \ttset{\ttabs{u-u_{k_j}}>2^{-j}}$.
By~\eqref{eq:p:Capacities:1},
\begin{align*}
\mathrm{cap}_s(A_j)\leq 2^{2j}\tnorm{u-u_{k_j}}_{\mathcal H^s}^2\leq 2^{-j}\fstop
\end{align*}
Set~$A\coloneqq  \bigcap_{\ell=1}^\infty \bigcup_{j=\ell}^\infty A_j$. If~$x\notin A$, it is readily seen that~$\lim_j\ttabs{u(x)-u_{k_j}(x)}=0$ by definition of the sets~$A_j$.
Thus, it suffices to show that~$\mathrm{cap}_s(A)=0$.
Since~$\mathrm{cap}_s$ is a Choquet capacity, it is increasing and (countably) subadditive, and we have that
\begin{align*}
\mathrm{cap}_s(A)\leq \mathrm{cap}_s\paren{\bigcup_{j=\ell}^\infty A_j}\leq \sum_{j=\ell}^\infty\mathrm{cap}_s(A_j)\leq \sum_{j=\ell}^\infty 2^{-j}=2^{-\ell+1}\comma \qquad \ell\in\N\fstop
\end{align*}
Since~$\ell$ was arbitrary, the conclusion follows letting~$\ell\to\infty$.
\end{proof}

For the above-mentioned, celebrated estimate 
for the volume of balls by J.-P.~Kahane
--- concerning the so-called $R^+_\alpha$ classes ---, we refer to the survey \cite{RhoVar14} by R.~Rhodes and V.~Vargas. 
Note that $|k_g(x,y)+\log d(x,y)|\le C$ and recall Remark \ref{positive-k} concerning positivity of $k_g$.

\begin{lemma}[{\cite[Thm.~2.6]{RhoVar14}}] 
\label{ball-vol}
Take $\alpha \in (0,n)$ and $\gamma^2/2\le \alpha$.
Consider a co-polyharmonic Gaussian field $h\sim {\sf CGF}_{\g}$ and the associated plain Liouville  quantum gravity measure~$\mu_g^{\gamma h}$  on $(M,g)$.
Then almost surely, for all $\varepsilon>0$ there exists $\delta>0$, $C<\infty$, and a compact set $M_\varepsilon\subset M$ such that $\mu_g^{\gamma h}(M\setminus M_\varepsilon)<\varepsilon$ and
\begin{equation}\label{est-vol}
\mu^{\gamma h}_g\big(B_r(x)\cap M_\varepsilon\big)\le C r^{\alpha-\gamma^2/2+\delta}, \qquad \forall r > 0,\, \forall x \in M.
\end{equation}
\end{lemma}

\medskip

\begin{proof}[Proof of Theorem~\ref{t:PolarNegligible}] 
For a.e.~$h$ the following holds true. Let  numbers~$\gamma, s\in\R$ with~$\gamma^2<4s\le 2n$ be given as well as a  Borel set $E\subset M$ with $\mu_g^{\gamma h}(E)>0$.
Applying Lemma~\ref{ball-vol} with $\alpha\coloneqq n+\gamma^2/2-2s\in [\gamma^2/2,n)$ and $\varepsilon\coloneqq \frac12\mu_g^h(E)
>0$ yields the existence of $\delta>0$, $C<\infty$, and a compact set $M_\varepsilon\subset M$ such that $\mu_g^{\gamma h}(M\setminus M_\varepsilon)<\varepsilon$ and~\eqref{est-vol} holds.
Set $\mu_\varepsilon\coloneqq \mathbf{1}_{M_\varepsilon}\,\mu_g^{\gamma h}$.
Then $\mu_\varepsilon(E)\ge\varepsilon>0$.
Furthermore, with $f(r)\coloneqq r^{2s-n}$ and $R\coloneqq \text{diam}(M)$, uniformly in $y$,
\begin{align*}
\int_M f\big(d(x,y)\big)\,d\mu_\varepsilon(x)=&\ -\int_M\int_0^R \mathbf{1}_{\{r>d(x,y)\}}f'(r)dr\,d\mu_\varepsilon(x)
\\
=&\ -\int_0^R \mu_\varepsilon(B_r(y))\, f'(r)\,dr
\\
\leq&\ (n-2s)\int_0^R r^{\alpha-\gamma^2/2 +\delta}\, r^{2s-n}\frac{dr}r
\\
=&\ (n-2s)\int_0^R r^\delta\frac{dr}r\le C'<\infty\fstop
\end{align*}
Hence, according to Lemma \ref{GreenEstimates}, 
\begin{equation}\label{unif-g-est}
0\le G_{s,1}\mu_\varepsilon (y)\le C'\,.
\end{equation}
Thanks to the convolution property of the kernels $G_{r,1}$ for $r>0$ \cite[Lem.~2.3$(ii)$]{LzDSKopStu20}, 
the uniform estimate \eqref{unif-g-est}, and the fact that $\mu^{\gamma h}_g$ is a finite measure, we find, with $\mu'\coloneqq \mathbf{1}_{E}\mu_\varepsilon$:
\begin{align*}
\norm{G_{s/2,1}(\mu')}_{L^2}^2&=\iiint G_{s/2,1}(x,y)\,d\mu'(y)\, G_{s/2,1}(x,z)\,d\mu'(z)\,d\vol_g(x)\\
&=\iint G_{s,1}(y,z) \,d\mu'(y)\,d\mu'(z)\\
&\le C'\cdot \mu'(M)=:C''<\infty.
\end{align*}
Hence, by the very definition of $b_{s}$,
\begin{equation*}
b_{s}(E)\ge \frac{\mu'(E)}{\|G_{s/2,1}(\mu')\|_{L^2}}\ge \frac{\varepsilon}{ \sqrt{C''}}>0\,,
\end{equation*}
and thus in turn $\text{cap}_{s}(E)>0$ according to Lemma \ref{Capacities}.
\end{proof}


\subsection{Random Dirichlet form and Liouville Brownian motion}
For sufficiently small $|\gamma|$, the Liouville Quantum Gravity measure $\mu^{\gamma h}_g$  does not charge sets of $\mathcal H^1$-capacity zero.
Hence, a random Brownian motion can easily be constructed through time change of the standard Brownian motion $((B_t)_{t\ge0}, ({\mathbb P}_x)_{x\in M})$.

\begin{theorem}\label{t:liouville-bm} Let~$(\M,\g)$ be admissible, let 
 $h\sim {\sf CGF}_{g}$ denote the co-polyharmonic Gaussian field and  $\mu_g^{\gamma h}$  the associated Liouville Quantum Gravity measure
 with   $|\gamma|<2$. Then for ${\mathbf P}$-a.e.~$h$,
\begin{enumerate}[$(i)$]
\item\label{i:t:liouville-bm-dirichlet-form}
A regular strongly local Dirichlet form on~$L^2\big(\M,\mu^{\gamma h}_g\big)$ is given by
\begin{equation}\label{eq:LBM}
{\mathcal E}^h(f,f)\coloneqq  \int_\M |\nabla f|^2\,d\vol_\g\,, \qquad \mathcal{D}({\mathcal E}^h)\coloneqq  \Big\{f\in \mathcal H^1(\M): \ \tilde f\in L^2(M,\mu^{\gamma h}_g) \Big\}
\end{equation}
where $\tilde f$ denotes the quasi-continuous modification of $f\in \mathcal H^1(M)$.
\item\label{i:t:liouville-bm-time-change}
The associated reversible continuous Markov process $\big((X^h_t)_{t\ge0}, ({\mathbb P}^h_x)_{x\in M}\big)$, called \emph{Liouville Brownian motion on $(M,g)$}, is obtained by time change of the standard Brownian motion   on $(\M,g)$.
Namely, let ${(A_{t}^{h})}_{t \ge 0}$ be the additive functional whose Revuz measure is given by $\mu^{\gamma h}_{g}$, then
\begin{equation*}
  {\mathbb P}^h_x\coloneqq {\mathbb P}_x,\qquad X^h_t\coloneqq B_{\tau_t^h}\comma \qquad
  \tau^h_t\coloneqq \inf\{s\ge0: A^h_s>t\}
  \fstop
\end{equation*}

\item\label{i:t:liouville-bm-revuz-functional}
  Moreover, for every bounded probability density $\rho$ on $M$, the additive functional $(A^h_t)_{t\ge0}$ is ${\mathbb P}_\rho$-a.s. given by  \begin{equation}\label{eq:PCAF}
         A^h_t = \lim_{\ell\to\infty}\int_0^t \exp\left(\gamma\, h_\ell(B_s)-\frac{\gamma^2}2 k_\ell(B_s,B_s)\right)ds \comma
  \end{equation}
with $h_\ell$ and $k_\ell$ as in \eqref{eq:approx-cph-eigenfunctions} and \eqref{eq:approx-kernels-eigenfunctions}.
\end{enumerate}
\end{theorem}

\begin{remark}
Recall that the additive functional $A^{h}$ associated with the measure $\mu^{\gamma h}_g$ is the process characterized by
  \begin{equation}\label{eq:Revuz-correspondence}
    \mathbb{E}_{x} \left[ \int_{0}^{t} u(B_{s}) d A_{s}^{h} \right] = \int_{0}^{t} \int u(y)\, p_{s}(x,y)\, d\mu^{\gamma h}_g(y) ds\comma \qquad u \in \mathfrak{B}_b\comma  t \geq 0\comma
    \end{equation}
   where~$\mathfrak{B}_b$ denotes the space of real-valued bounded Borel functions on~$M$.
  For further information on additive functionals, see \cite{FukOshTak11}.
\end{remark}

  \begin{proof}
    \ref{i:t:liouville-bm-dirichlet-form} and \ref{i:t:liouville-bm-time-change} hold using standard argument in the theory of Dirichlet forms.
    Indeed, Corollary~\ref{cor:cap} and the compactness of~$\M$ imply that for ${\mathbf P}$-a.e.~$h$, the measure~$\mu^{\gamma h}_g$ is a Revuz measure of finite energy integral.
    For details see \cite[Thm.\ 1.7]{GarRhoVar14}, where this argument is carried out in the case when~$M$ is the unit disk.

    \ref{i:t:liouville-bm-revuz-functional}
    Fix $t > 0$ and $\rho$ a probability measure with bounded  density on~$M$.
    We consider the occupation measure
    \begin{equation*}
      dL_{t}(x) = \int_{0}^{t} d\delta_{B_{s}}(x) ds \fstop
    \end{equation*}
    Observe that for  $\alpha > 0$,
    \begin{equation*}
      \begin{split}
        \mathbb{E}_{\rho} &\left[\iint \frac{dL_{t}(y)\, d L_{t}(z)}{{d(y,z)}^{\alpha}} \right] = \int_{0}^{t}\!\!\! \int_{0}^{t} \mathbb{E}_{\rho} {d(B_{r}, B_{s})}^{-\alpha} dr\, ds \\
                                                                                            &= 2\int_{0}^{t}\!\!\! \int_{s}^{t} \iiint {d(y,z)}^{-\alpha} p_{s}(x,y)\, p_{r-s}(y,z) d\vol_g(y) d\vol_g(z) d\rho(x) dr ds\\
                                                                                            &\le C\,t \sup_{\ell}\int_0^t \iint {d(y,z)}^{-\alpha} p_{s}(y,z)\,d\vol_g(y) \,d\vol_g(z)\,ds\\
&\le C\,t\ee^t  \iint {d(y,z)}^{-\alpha} G_{1,1}(y,z)\,d\vol_g(y) \,d\vol_g(z)\fstop      \end{split}
    \end{equation*}
   According to the estimate for the $1$-Green kernel $G_{1,1}$, the latter integral 
    is finite for all $\alpha < 2$.
    This means that, $\mathbb{P}_{\rho}$-almost surely, $L_{t}$ satisfies \cite[Eqn.~(39)]{Kah85} for all $\alpha < 2$. Thus $L_{t}$ is, $\mathbb{P}_{\rho}$-almost surely, in the class $M^{+}_{\alpha^{+}}$ for all $\alpha < 2$.
    Arguing as in the proof of Theorem~\ref{t:existence-liouville}, we find that, for all $\gamma^{2} < 4$, having fixed the randomness with respect to $\mathbb{P}_{\rho}$, there exists a random measure $\nu^{\gamma h}_{t}$ that is the Gaussian multiplicative chaos over $(h, \gamma k)$ 
    with respect to $L_{t}$.

    Now, for all Borel sets $A \subset M$ we set
    \begin{equation*}
      \begin{split}
        \nu_{t}^{\gamma h_{\ell}}(A) &\coloneqq \int_{A} \exp\paren{ \gamma h_{\ell}(x) - \tfrac{\gamma^{2}}{2} k_{\ell}(x,x) } dL_{t}(x) \\
                              &= \int_{0}^{t} 1_{A}(B_{s}) \exp\paren{ \gamma h_{\ell}(B_{s}) - \tfrac{\gamma^{2}}{2} k_{\ell}(B_{s}, B_{s}) } ds \fstop
      \end{split}
    \end{equation*}
    Since we choose~$\seq{h_\ell}_\ell$ as in~\eqref{eq:approx-cph-eigenfunctions}, the family~$\tseq{\nu^{\gamma h_\ell}_t}_{\ell}$ is a $\mathbf{P}$-martingale for every fixed~$t\geq 0$, similarly to Theorem~\ref{t:EigenfunctionApproxGMC}.
The fact that $\nu_{t}^{\gamma h_{\ell}} \to \nu_{t}^{\gamma h}$ follows from the same uniform integrability argument for martingales as in Theorem~\ref{t:EigenfunctionApproxGMC}.

    For all $t > 0$, set $A_{t}^{h} \coloneqq \nu_{t}^{\gamma h}(M)$ and $A_{t}^{h_{\ell}} \coloneqq \nu_{t}^{\gamma h_{\ell}}(M)$ for each $\ell\in\N$.
    It is clear that $t \mapsto A_{t}^{h_{\ell}}$ is the positive continuous additive functional associated to~$\mu_g^{\gamma h_{\ell}}$ by the Revuz correspondence, that is (cf.~\eqref{eq:Revuz-correspondence}),
    \begin{align}\label{eq:t:Liouville:1}
       \mathbb{E}_\rho \braket{\int_0^t u(B_s)d A^{h_\ell}_s}  =& \int_0^t\!\! \int \!\braket{\int\!\! p_s(x,y) u(y) d\mu^{\gamma h_\ell}_{g}(y)} d\rho(x) ds \comma \qquad u\in\mathfrak{B}_b\comma t\geq 0\fstop
    \end{align}
    Now let ~$\tilde A_{t}^{h}$ denote the positive continuous additive functional associated with $\mu^{\gamma h}_g$. Then 
    applying \eqref{prec-conv} twice  --- to $\mu_g^{\gamma h_\ell}\to \mu_g^{\gamma h}$ and to $\nu_t^{h_\ell}\to \nu_t^{h}$ --- 
    we  obtain that in $\mathbf{P}$-probability:
    \begin{align*}
      \begin{split}
   \mathbb{E}_{\rho} \braket{ \int_{0}^{t} u(B_{s}) dA_{s}^{h} }&=    \mathbb{E}_{\rho} \braket{ \int_M u\,d\nu_t^h}\\
   &
   =  \lim_{\ell \to \infty}  \mathbb{E}_{\rho} \braket{ \int_M u\,d\nu_t^{h_\ell}}=
      \lim_{\ell \to \infty} \mathbb{E}_{\rho} \braket{ \int_{0}^{t} u(B_{s}) d A^{h_{\ell}}_{s} }\\ &= \lim_{\ell\to\infty}\int_0^t \int \braket{\int p_s(x,y) u(y) d\mu^{\gamma h_\ell}_{g}(y)} d\rho(x) ds
                                                                                                   \\&=\int_0^t \int \braket{\int p_s(x,y) u(y) d\mu^{\gamma h}_{g}(y)} d\rho(x) ds
                                                                                                                    = \mathbb{E}_{\rho} \braket{ \int_{0}^{t} u(B_{s}) d\tilde A_{s}^{h} } \fstop
      \end{split}
    \end{align*}
    This shows that $A^{h} = \tilde{A}^{h}$ a.s.~w.r.t.~${\mathbf P}\otimes {\mathbb P}_\rho$ and concludes the proof.
  \end{proof}

\begin{remark} The intrinsic distance associated to the Dirichlet form~\eqref{eq:LBM} vanishes identically. This can be easily verified, exactly as in~\cite[Prop.~3.1]{GarRhoVar14}.
\end{remark}

\begin{remark} The previous constructions work equally well   with  the adjusted Liouville measure ${\bar\mu}^{\gamma h}_{g}$ (or with the refined Liouville measure $\tilde\mu_g^{\gamma h}$) in the place of
the plain Liouville measure $\mu^{h}_{g}$. For a.e.~$h$, the resulting process, the adjusted  (or refined, resp.) Brownian motion, can be regarded as the plain Brownian motion with drift. 
\end{remark}

\begin{remark} In the case $n=2$, Liouville Brownian motion shares an important quasi-invariance property under conformal transformations.
In higher dimensions, no such --- or similar --- conformal quasi-invariance property holds true.
Indeed, the generator of the Brownian motion, the Laplace--Beltrami operator, is quasi-invariant under conformal transformations if and only if $n=2$. 

For the 2-dimensional counterparts of the previous theorem, see \cite{Ber15} and \cite{GarRhoVar14, GarRhoVar16}.
\end{remark}

\subsection{Random Paneitz and random GJMS operators}
In higher dimensions, from the perspective of conformal quasi-invariance, the natural random operators to study are random perturbations of the co-polyharmonic operators~$\Pol_g$.
To simplify notation, we henceforth write $\Pol$ and $\vol$ rather than~$\Pol_g$ and~$\vol_g$.

\begin{theorem}\label{t:liouville-co-polyharmonic}
  Let~$(\M,g)$ be admissible, let 
 $h\sim {\sf CGF}_{\g}$ denote the co-polyharmonic Gaussian field and  $\mu_g^{\gamma h}$  the associated plain Liouville Quantum Gravity measure
 with   $|\gamma|<\sqrt{2n}$.
 Then for ${\mathbf P}$-a.e.~$h$,
\begin{align*}
  \mathfrak{E}^h (u,v)\coloneqq  \int_\M \sqrt{\Pol} u\, \sqrt{\Pol} v \, d\vol_g \,, \qquad u,v\in \mathcal{D}(\mathfrak{E}^h)\coloneqq  \mathcal H^{n/2}\cap L^2(M,\mu^{\gamma h}_g)\comma
\end{align*}
is a well-defined non-negative closed symmetric bilinear form on~$L^2(M,\mu^{\gamma h}_g)$.
\end{theorem}

\begin{proof}
Since~$\mu^{\gamma h}_g$ does not charge $\mathrm{cap}_{n/2}$-polar sets by Theorem~\ref{t:PolarNegligible}, and since every~$f\in \mathcal H^{n/2}$ is $\mathrm{cap}_{n/2}$-q.e.\ finite by Proposition~\ref{Capacities}\ref{i:p:Capacities:4}, every~$f\in \mathcal H^{n/2}$ admits a $\mu^{\gamma h}_g$-a.e.\ finite representative (possibly depending on~$h$).
Thus,~$\mathfrak{E}^h$ is well-defined on~$\mathcal H^{n/2}\cap L^0(\mu^{\gamma h}_g)$.
In order to show that~$\mathfrak{E}^h$ is finite on~$\mathcal{D}(\mathfrak{E}^h)$, let~$u=\mathsf{G}_{n/4,1} u'$, resp.~$v=\mathsf{G}_{n/4,1} v'\in \mathcal H^{n/2}$, with~$u',v'\in L^2$, and note that
\begin{align*}
\mathfrak{E}^h(u,v)=\scalar{\mathsf{G}_{n/4,1} u'}{\Pol \mathsf{G}_{n/4,1} v'}_{L^2}=&\ \scalar{u'}{\mathsf{G}_{n/4,1}\Pol \mathsf{G}_{n/4,1} v'}_{L^2}
\\
\leq&\ \norm{u'}_{L^2}\norm{v'}_{L^2} \norm{\mathsf{G}_{n/4,1}\Pol \mathsf{G}_{n/4,1}}_{L^2\to L^2}<\infty
\end{align*}
by admissibility of~$\M$.

In order to show closedness it suffices to show that~$\mathcal{D}(\mathfrak{E}^h)$ is complete in the graph-norm
\begin{align*}
\norm{u}_{\mathcal{D}(\mathfrak{E}^h)}\coloneqq  \paren{\mathfrak{E}^h(u)+\norm{u}_{L^2(\mu^{\gamma h}_g)}^2}^{1/2}\comma \qquad u\in \mathcal{D}(\mathfrak{E}^h) \fstop
\end{align*}
Since~$\mathfrak{E}^h$ vanishes on constant functions by Theorem~\ref{Pol-basic}\ref{i:Pol-basic:2}, it suffices to show that $\mathcal{D}(\mathring{\mathfrak{E}}^h)\coloneqq  \mathring{\mathcal H}^{n/2}\cap L^2(\mu^{\gamma h}_g)$ is complete in the same norm.
To this end, let~$\seq{u_k}_k$ be $\mathcal{D}(\mathring{\mathfrak{E}}^h)$-Cauchy and note that it is in particular both $L^2(\mu^{\gamma h}_g)$- and $\mathfrak{E}^h$-Cauchy.
In particular, there exists the $L^2(\mu^{\gamma h}_g)$-limit~$u$ of~$\seq{u_k}_k$, and, up to passing to a suitable non-relabeled subsequence, we may further assume with no loss of generality that $\lim_k u_k=u$ $\mu^{\gamma h}_g$-a.e.
Furthermore, by Lemma~\ref{l:PropertiesCoPolyH}\ref{i:l:PropertiesCoPolyH:3}, $\mathfrak{E}^h$ defines a norm on~$\mathring{\mathcal H}^{n/2}$, bi-Lipschitz equivalent to the standard norm of~$\mathring{\mathcal H}^{n/2}$.
As consequence,~$\seq{u_k}_k$ is as well $\mathring{\mathcal H}^{n/2}$-Cauchy, and, by completeness of the latter, it admits an $\mathring{\mathcal H}^{n/2}$-limit~$u'$.
Up to passing to a suitable non-relabeled subsequence, by Proposition~\ref{Capacities}\ref{i:p:Capacities:5} we may further assume with no loss of generality that~$\lim_k u_k=u'$ $\mathrm{cap}_{n/2}$-q.e..
In particular, again since~$\mu^{\gamma h}_g$ does not charge~$\mathrm{cap}_{n/2}$-polar sets, we have that~$\lim_k u_k=u$ $\mu^{\gamma h}_g$-a.e., i.e.~$u'=u$ $\mu^{\gamma h}_g$-a.e., hence as elements of~$L^2(\mu^{\gamma h}_g)$.
It follows that~$L^2(\mu^{\gamma h}_g)$-$\lim_k u_k=u$, which concludes the proof of completeness.

Non-negativity is a consequence of the admissibility of~$\M$. Symmetry follows from that of~$\Pol$, Theorem~\ref{Pol-basic}\ref{i:Pol-basic:4}.
\end{proof}

\begin{corollary} Let $(M,g)$, $h$, $\gamma$, and $\mu^{\gamma h}_g$ be as above.
  Then for ${\mathbf P}$-a.e.~$h$
there exists a unique nonnegative self-adjoint operator 
 ${\sf P}^h$ on  $L^2(\M,\mu^{\gamma h}_g)$, called \emph{random co-polyharmonic operator} or  \emph{random GJMS operators},  defined by
 ${\mathcal D}(\mathsf{P}^h)\subset {\mathcal D}(\mathfrak{E}^h)$ 
 and
 \begin{equation*}
 \mathfrak{E}^h(u,v)=\int u\, \mathsf{P}^h v\,d \mu_g^{\gamma h}\comma \qquad u\in \mathcal{D}(\mathsf{P}^h)\comma v\in {\mathcal D}(\mathfrak{E}^h)\fstop
 \end{equation*}
\end{corollary}
In the case $n=4$, the operators ${\sf P}^h$ are also called \emph{random Paneitz operators}.

\begin{corollary} With $(M,g)$, $h$, $\gamma$, and $\mu^{\gamma h}_g$ as above, for a.e.~$h$ there exists a semigroup $\tseq{e^{-t{{\sf P}^h}}}_{t>0}$ of bounded symmetric operators on $L^2(\M,\mu^{\gamma h}_g)$, called \emph{
 random  co-polyharmonic heat semigroup}.
\end{corollary}
\begin{proposition}\label{t:gradient-flow}
  The random co-polyharmonic heat flow $(t,u)\mapsto e^{-t{{\sf P}^h}}u$ is the
  EDE-gradient flow for $\frac12\mathfrak{E}^h$ on $L^2(\M,\mu_g^{\gamma h})$.

Here `EDE' stands for gradient flow in the sense of `energy-dissipation-equality', see \cite[Dfn.~3.4]{ambrosio2013user}.
\end{proposition}

\begin{proof} The energy decays along the flow according to
  $$\frac d{dt} \mathfrak{E}^h(u_t)=\scalar{\sqrt{{\sf P}^h}\frac d{dt} u_t}{\sqrt{{\sf P}^h}u_t}_{L^2(\mu_g^{\gamma h})}=-\scalar{\mathsf{P}^h u_t}{ \mathsf{P}^hu_t}_{L^2(\mu_g^{\gamma h})}
=-\|{\sf P}^hu_t\|^2_{L^2(\mu_g^{\gamma h})}$$
for $u_t\coloneqq \ee^{-t{{\sf P}^h}}u_0$. Moreover for each differentiable curve $v_t$ we have
\begin{align}
  \frac{d}{dt}\mathfrak{E}^h(v_t)=\frac{d}{dt}\int \sqrt{\sf P^h}v_t\sqrt{\Pol^h}v_t\, d\mu_g^{\gamma h}
=\int\frac{d}{dt}v_t \Pol^h v_t\, d\mu_g^{\gamma h}=\scalar{\frac{d}{dt}v_t}{\Pol^h v_t}_{L^2}.
\end{align}
Consequently $\nabla \mathfrak{E}^h(v)=\Pol^h v$ which leads to 
\begin{align}
  \frac d{dt} u_t=-\nabla \mathfrak{E}^h(u_t)
\end{align}
and thus the assertion.
\end{proof}

\begin{remark} For a.e.~$h$ and every $u\in L^2(\M,\mu_g^{\gamma h})$, the solutions $u_t\coloneqq \ee^{-t{{\sf P}^h}}u$ are absolutely continuous with respect to $\mu_g^{\gamma h}$ for all $t>0$.
It is plausible to conjecture that there exists a \emph{random co-polyharmonic heat kernel} $p_t^h$ such that
\begin{equation*}
\ee^{-t{{\sf P}^h}}u(x)=\int_\M p^h_t(x,y)\, u(y)\,d\mu^{\gamma h}_g(y)\qquad \text{for a.e. }x\in M\comma \qquad u\in L^2\fstop
\end{equation*}

In the case $n=2$, such a kernel exists, and it admits sub-Gaussian upper bounds (see~\cite{MaillardRhodesVargasZeitouni,AndKaj15}), 
\begin{equation*}
p_t^h(x,y)\leq C_1 t^{-1} \log(t^{-1}) \exp\paren{-C_2\paren{\frac{d(x,y)^\beta\wedge 1}{t}}^{\frac{1}{\beta-1}}}\comma \qquad t\in \big(\tfrac{1}{2}, 1\big]\comma
\end{equation*}
for any~$\beta>\tfrac{1}{2}(\gamma+2)^2$ and constants~$C_i=C_i(\beta, \gamma, h, d(y,0))$.
\end{remark}

Now  let us address the conformal quasi-invariance of the random co-polyharmonic operators. For this purpose, of course, we have to emphasize all $g$-dependencies in the notation and thus write $\Pol_g$ and $\Pol_g^h$ rather than $\Pol$ and $\Pol^h$. 

Assume that the Riemannian manifold $(\M,\g)$ is admissible and that $|\gamma|<\sqrt{2n}$. Let $h\sim {\sf CGF}_{\g}$  denote the co-polyharmonic random field and $\mu_g^{\gamma h}$ the corresponding plain Liouville Quantum Gravity measure on $(M,g)$.

Given any $\g'=e^{2\varphi}\g$ with $\varphi\in\C^\infty(\M)$, define (a version of) the  Liouville Quantum Gravity measure on $(M,g')$ 
according to Theorem \ref{q-in-meas} by
\begin{equation}\label{co.eq.me}
\mu^{\gamma h}_{g'}\coloneqq  \ee^{F}\,\mu_g^{\gamma h}\,.
\end{equation}
with $v'=\vol_{g'}(M)$ and
\begin{equation}\label{eq:ConformalFactorPlainLQG}
F\coloneqq 
-\gamma\av{h}_{g'}+\frac{\gamma^2}{2v'}\mathcal{K}_g(e^{n\varphi},e^{n\varphi})-\bigg(\frac\gamma{v'}\bigg)^2\,{\sf k}_g(e^{n\varphi})+n\varphi \fstop
\end{equation}

\begin{theorem}\label{t:invariance-random-co-polyharmonic}
The random co-polyharmonic operator ${\sf P}_g^h$ is conformally quasi-in\-variant: if $\g'=e^{2\varphi}\g$ then
\begin{equation}
\Pol^{h'}_{g'}\stackrel{\rm (d)}= \ee^{-F}\,\Pol_g^h
\end{equation}
with $F$ as above.
\end{theorem}

\begin{proof} 
By the conformal quasi-invariance of the Liouville Quantum Gravity measure as stated in   \eqref{co.eq.me} and by the conformal invariance of the bilinear form $\mathfrak{E}_g$, 
\begin{align*}
  \int_M \Pol_g^hu\, v\, d\mu_g^{\gamma h}&=\mathfrak{E}^h_g(u,v)=\mathfrak{E}^h_{g'}(u,v)=\int_M \Pol_{g'}^hu\, v\, d\mu_{g'}^{\gamma h}
=\int_M e^{Z^h}\Pol_{g'}^hu\, v\, d\mu_{g}^{\gamma h}
\end{align*}
for all $u$ and $v$ in appropriate domains. Hence, $\Pol_g^hu=e^{Z^h}\Pol_{g'}^hu$. This proves the claim.
\end{proof}

\begin{remark}
  The above construction can also be carried out with ${\bar\mu}^{\gamma h}_{g}$ instead of $\mu^{\gamma h}_{g}$ yielding the \emph{adjusted random co-polyharmonic operator} $\bar{\mathsf{P}}_{g}^{h}$.
In that case, we  get for the conformal quasi-invariance the following formula
\begin{equation*}
  \bar{\mathsf{P}}_{g'}^{h} \stackrel{\rm (d)}= \ee^{-{\bar F}} \bar{\mathsf{P}}_{g}^{h}\comma
\end{equation*}
where $\bar {F} = \gamma \scalar{h}{\ee^{n \bar{\varphi}}} + (n + \gamma^{2}/2) \varphi$.
\end{remark}

\section{The Polyakov--Liouville measure}

Our last objective in this paper is to propose  a version of conformal field theory
 on compact manifolds of arbitrary even dimension,  an approach based on Branson's $Q$-curvature.
We provide a rigorous meaning to the  Polyakov--Liouville measure $\boldsymbol{\pi}_g$, informally  given as 
\begin{equation*}
 \exp\tparen{-S_g(h)} \, dh
\end{equation*}
with the (non-existing) uniform distribution $dh$ on the set of fields (thought as sections of  some bundles over $M$),
and the 
action
\begin{equation}\label{action-n}
  S_{\g}(h) \coloneqq   \int_M\Big( \frac{1}2\,\tabs{\sqrt{\mathsf p_g}\, h}^2 +\Theta\,Q_g h +\frac{\Theta^*}{\vol_g(M)}h+ {m} e^{\gamma h}\Big)\, d \vol_g\comma
\end{equation}
where $\mathsf p_g=a_n\Pol_g$ is the normalized co-polyharmonic operator (with the constant $a_n$ from \eqref{an}), $Q_g$ denotes Branson's curvature,   and $m,\Theta, \Theta^*,\gamma$ are parameters --- subjected to some restrictions specified below, in particular,~$0<|\gamma|<\sqrt{2n}$.

{We note that the factor~$\Theta^*/\vol_g(M)$ does not appear elsewhere in the literature.
Its presence is needed to address constructions of \emph{plain} objects, see~\S\ref{ss:PlainPL}, while the factor is not present when addressing \emph{}adjusted objects, see~\S\ref{ss:AdjustedPL}.}

\subsection{Heuristics and motivations}
Before going into the details of our approach, let us briefly recall the longstanding challenge of conformal field theory and some recent breakthroughs in the two-dimensional case. 
Here~\eqref{action-n} becomes  the celebrated Polyakov--Liouville action
\begin{equation}\label{pol-lio-act}
  S_g(h) = \int_M \left( \frac1{4\pi}{|\nabla h|}^{2} +\frac\Theta2 \, R_g \, h +\frac{\Theta^*}{\vol_g(M)}h+  m \,e^{\gamma h}\right) d \vol_g\comma
\end{equation}
where $R_g$ is the scalar curvature  and $m,\Theta,\Theta^*,\gamma$ are parameters. (Instead of $m$ and $\Theta$ mostly in the literature $\bar\mu$ and $Q$ are used. However, in this paper the latter symbols are already reserved for the Liouville Quantum Gravity measure and Branson's curvature.)
With the Polyakov--Liouville action, this ansatz for the measure 
${\boldsymbol{\pi}_g}(dh) ={ \frac1{Z_g^*}}e^{-S_g(h)} dh$
reflects the coupling of the gravitational field with a matter field. 
It can be regarded as quantization of the  the classical Einstein--Hilbert action
$S_g^{EH}(h)=\frac1{2\kappa}\int_M \big(R_g-2\Lambda\big)\,dx$
or, more precisely, of its coupling with a matter field
\begin{equation*}
  S_g^{EH}(h)=\int_M\bigg[ \frac1{2\kappa}\big(R_g-2\Lambda\big)+{\mathcal L}_M\bigg]\,dx \fstop
\end{equation*}

\smallskip

In the case $n=2$ and $Q^*=0$, based on the concepts of Gaussian Free Fields and Liouville Quantum Gravity measures, the rigorous construction of such a Polyakov--Liouville measure~${\boldsymbol{\pi}_g}$ has been carried out recently in \cite{DKRV16} for surfaces of genus~$0$, \cite{DRV16} for surfaces of genus 1 (see also \cite{HRV18} for the disk), and in  \cite{GuiRhoVar19} for surfaces of higher genus. For related constructions, see \cite{DuplantierMillerSheffield}.
The approach of  \cite{GuiRhoVar19} 
 gives a rigorous meaning to
\begin{equation*}
\boldsymbol{\pi}_g(dh) = \exp\braket{-\int \paren{\frac\Theta2 R_g\,h + m\, e^{\gamma h}} d \vol_g}\exp\left(- \frac{1}{4\pi} \norm{\nabla h}_{L_g^2}^{2}\right) dh
\end{equation*}
by setting 
\begin{equation*}
    \boldsymbol{\nu}_g^*(dh)\coloneqq \exp\paren{-\frac\Theta2\scalar{h}{R_g}_g- m \,{\bar\mu}_{g}^{\gamma h}(M)} \, \widehat{\boldsymbol{\nu}}_g(dh) \comma
  \end{equation*}
   where ${\bar\mu}^{\gamma h}_{g}$ denotes the adjusted Liouville Quantum Gravity measure on~$M$ with parameter~$\gamma\in(0,2)$, 
and 
  \begin{equation*}
  \boldsymbol{\pi}_g(dh)\coloneqq \sqrt{ \frac{\vol_{g}(M)}{{\det}' (-\frac1{4\pi^2}\Delta_{g})}} \
    \boldsymbol{\nu}_g^*(dh) \fstop
  \end{equation*}
  This provides a complete solution to the above mentioned challenge in dimension~2.
  
{In dimension greater than~$2$, the same constructions are considered on spheres by B.~Cercl\'e in~\cite{Cercle}, and are anticipated in the physics literature by T.~Levy and Y.~Oz in~\cite{LevyOz}.
Here, we carry out the same programme {with mathematical rigor and}  in full generality on arbitrary even-dimensional admissible manifolds.
}

\begin{remark}\label{normalized-heuristic}
The previous argumentation is based on the interpretation  of  $\widehat{\boldsymbol{\nu}}_g\coloneqq \widehat{{\sf GFF}}_{g}$, the law of the ungrounded Gaussian Free Field, as a rigorous definition for the measure
\begin{equation*}
 \sqrt{ \frac{{\det}' (-\frac1{4\pi^2}\Delta_{g})}{\vol_{g}(M)}} \ \exp\paren{-{\tfrac1{4\pi}}\langle h|-\Delta h\rangle} dh \fstop
\end{equation*}
To justify the latter, recall that for a nonnegative symmetric operator $A$  on a finite dimensional Hilbert space $H$ with orthonormal eigenbasis
$\{e_1,\ldots,e_n\}$ and eigenvalues $\{\lambda_1,\ldots,\lambda_n\}$ we put
$$\int_H f(h)\exp\Big(-\pi\,\langle h,Ah\rangle \Big)dh\coloneqq \int_{\R^n} f\Big(\sum_i x_ie_i\Big)\exp\Big(-\pi\,\sum_i \lambda_i x_i^2 \Big)dx
$$
and thus in particular
$$\int_H \exp\Big(-\pi\,\langle h,Ah\rangle \Big)dh=\frac1{\sqrt{\det(A)}}.$$

Therefore, for a positive self-adjoint operator $A$ with discrete spectrum $\{\lambda_j\}_{j\in\N}$ on a Hilbert space $H$ in analogy we  put
$$\int_H \exp\Big(-\pi\,\langle h,Ah\rangle \Big)dh=:\frac1{\sqrt{{\det}'(A)}}$$
where 
${\det}'(A)\coloneqq \exp\big(-\zeta_A'(0)\big)$ denotes the \emph{regularized determinant}, defined in terms of the meromorphic continuation of the function
$$\zeta_A(s)\coloneqq \sum_{j\ge 1} \lambda_j^{-s},$$
initially defined for $s$ with large enough real part.
For $A=-\frac1{4\pi^2}\Delta$ 
this leads to the identification
$$ \exp\Big(-{\tfrac1{4\pi^2}}\langle h|-\Delta h\rangle\Big) dh=\frac1{\sqrt{{\det}'(-\frac1{4\pi^2}\Delta)}}\,{\boldsymbol{\nu}}_g(dh)\qquad\text{on }\mathring H_g^{-\varepsilon}.$$

Furthermore,  the orthogonal decomposition $\hbar=h+a\,e_0$ of $\hbar\in H_g^{-\varepsilon}$ into $h\in \mathring H_g^{-\varepsilon}$ and a multiple of the normalized eigenfunction 
$e_0\coloneqq\frac1{\sqrt{\vol_g(M)}}$  for the (single) eigenvalue 0 of the Laplacian leads  to
\begin{align*}
\int_{H_g^{-\varepsilon}}f(\hbar)\exp\Big(-{\tfrac1{4\pi^2}}\langle \hbar|-\Delta \hbar\rangle\Big) d\hbar
&=
\int_\R\left(\int_{\mathring H_g^{-\varepsilon}}f(h+a\,e_0)\exp\Big(-{\tfrac1{4\pi^2}}\langle h-\Delta h\rangle\Big) dh\right)da\\
&={\sqrt{\vol_g(M)}}\,
\int_\R\left(\int_{\mathring H_g^{-\varepsilon}}
\frac{f(h+a)}{\sqrt{{\det}'(-\frac1{4\pi^2}\Delta)}}\,{\boldsymbol{\nu}}_g(dh)
\right)da\\
&=\frac{\sqrt{\vol_g(M)}}{\sqrt{{\det}'(-\frac1{4\pi^2}\Delta)}}\,\int_{H_g^{-\varepsilon}}f(\hbar)\,d\widehat{\boldsymbol{\nu}}_g(d\hbar)
\end{align*}
for any $f: H_g^{-\varepsilon}\to\R_+$.
This yields the identification 
\[
\exp\Big(-{\tfrac1{4\pi^2}}\langle \hbar|-\Delta \hbar\rangle\Big) d\hbar=\frac{\sqrt{\vol_g(M)}}{\sqrt{{\det}'(-\frac1{4\pi^2}\Delta)}}\,\widehat{\boldsymbol{\nu}}_g(d\hbar)\qquad\text{on }H_g^{-\varepsilon}.
\]
\end{remark}

\medskip

A remarkable property 
of the Polyakov--Liouville action is that it quantifies the conformal quasi-invariance of the functional determinant in dimension $2$.
Namely, we have that \cite[Eq. (1.13)]{OPS88}:
\begin{equation}\label{ops}
\log \frac{{\det}' (-\frac1{4\pi^2}\Delta_{g'})}{\vol_{g'}(M)} - \log \frac{{\det}' (-\frac1{4\pi^2}\Delta_{g})}{\vol_{g}(M)} = - \frac{1}{12 \pi} \int 2 \varphi\, \mathsf{scal}_g + \abs{\nabla \varphi}^{2}_{g} \, d \vol_{g}.
  \end{equation}
    
  Thus we can see the Polyakov--Liouville action as a potential accounting for the variation of the  functional determinant of the Laplacian coupled with the volume.
It is conjectured  (see \cite[Equ. (6)]{BransonGover}, \cite[Equ. (5.9b)]{Diaz08} and the references therein) that a physically relevant Polyakov formula for $n > 2$ should involve the (normalized) co-polyharmonic operators.
  Under our admissibility, 
    it should take the form:
  \begin{equation*}
    \log {\det}' \mathsf{p}_{g} - \log {\det}' \mathsf p_{g'} = \Theta \int \left[\frac{1}{2} \varphi\, \mathsf{P}_{g} \varphi + \varphi\, Q_{g} \right] \dd \vol_{g} + \int F_{g'}\, d \vol_{g'} - \int F_{g}\, d \vol_{g}+G \comma
  \end{equation*}
where $\Theta$ is a constant, $F_g$ and $F_{g'}$ are local scalar invariants, and $G$ is a global term. The first integral on the right-hand side here is regarded as the `universal part'.
In view of this formula, let us define a higher dimensional equivalent of the $2d$ Polyakov--Liouville action:
\begin{equation*}
  S_{g}(h) = \Theta \int h\, Q_{g} \, d \vol_{g} + 
  \frac{\Theta^*}{\vol_g(M)} \int h\,d\vol_g
  +m \int \ee^{\gamma h} d \vol_{g} + \frac{1}{2} \mathfrak{p}_{g}(h,h)\fstop 
\end{equation*}

\begin{remark}
Minimizers of $S_g$ satisfy
$$\mathsf p_g h+ \Theta Q_g+\frac{\Theta^*}{\vol_g(M)}+m\gamma \ee^{\gamma h}=0.$$
If we 
choose $\Theta^*=0$, $\Theta=\frac{n a_n}\gamma$, $m=-\frac{n a_n}{\gamma^2}\bar Q$ for some $\bar Q\in\R$ and put $\varphi=\frac\gamma n h$, then this reads as
$$\frac1{a_n}\,\mathsf{p}_g\varphi+{Q}_g=e^{n\varphi}\bar Q.$$
In other words,
$g'=e^{2\varphi}g$ is a metric of constant Branson curvature $Q_{g'}=\bar Q$.
\end{remark}

The remainder of this section is devoted to give a rigorous meaning to the measure
\begin{equation*}
  \boldsymbol{\pi}_{g}(d h) = \exp(-S_{g}(h)) dh.
\end{equation*}
As an ansatz, we regard the quantity $\frac1{Z_g}\exp(-\frac{1}{2} \mathfrak{p}_{g}(h,h))dh$ as an informal definition of the law of the ungrounded co-polyharmonic field.
With this interpretation, we regard $\int \ee^{\gamma h} d \vol_{g}$ as the volume of $M$ with respect to the Liouville Quantum Gravity measure.
Since the latter comes in two versions ---the plain and the adjusted Liouville measure---
we obtain two conformally quasi-invariant rigorous definitions of the above  measure, denoted henceforth by  $\boldsymbol{\nu}_{g}^{*}$ and $\boldsymbol{\bar\nu}_{g}^{*}$.

\begin{remark}
Before going into further details, let us have a naive look on the transformation property of our action functional under conformal changes which would apply if the random field $h$ were smooth.
Choose $\Theta^*=0$. Then by a direct computation, we have that for all $\varphi$ smooth and all $h \in H^{n/2}_g$:
\begin{align*}
S_{\ee^{2\varphi}g}\left(h - \frac{n}{\gamma} \varphi\right) =&\ S_{g}(h) + \left(\frac{\Theta}{a_n} -  \frac{n}{\gamma}\right) \mathfrak{p}_{g}(h, \varphi)
\\
&\ + \left(\frac{n^{2}}{2\gamma^{2}} -\frac{ \Theta}{a_n} \frac{n}{\gamma}\right) \mathfrak{p}_{g}(\varphi, \varphi) - \Theta \frac{n}{\gamma} \int \varphi\, Q_{g}\, d \vol_{g}\comma
\end{align*}
where we used that $Q_{\ee^{2\varphi}g} = \ee^{-n\varphi}(Q_{g} + \frac1{a_n}\mathsf p_{g} \varphi)$.
In particular, when selecting the special value $\Theta = a_{n} \frac{n}{\gamma}$ the above expression simplifies to
\begin{equation*}
  S_{\ee^{2\varphi} g}\left(h - \frac{n}{\gamma} \varphi\right) = S_{g}(h) - \frac{n^{2}}{\gamma^{2}} \left[\frac 12\mathfrak{p}_{g}(\varphi,\varphi) + a_n \int \varphi\, Q_{g}\, d \vol_{g}\right].
\end{equation*}
Therefore, writing $T$ for the shift by $h\mapsto h+\frac{n}{\gamma} \varphi$, we expect the following quasi-conformal invariance:
\begin{equation*}
  \log \frac{d T_* \boldsymbol{\pi}_{\ee^{2\varphi}g}}{d \boldsymbol{\pi}_{g}}(h) =  \frac{n^{2}}{\gamma^{2}} \left(\frac12\mathfrak{p}_{g}(\varphi,\varphi) + a_n \int \varphi\, Q_{g}\, d \vol_{g}\right) \fstop
\end{equation*}
However, due to the choice of the measure $\boldsymbol{\nu}_{*}$ the random field $h$ will not be smooth. 
As a consequence, the quasi-conformal invariance arises at a different value of $\Theta$ (and/or $\Theta^*$). Indeed, the renormalization of the adjusted (or plain) Liouville Quantum Gravity measure $\bar\mu_g^{\gamma h}$ (or $\mu_g^{\gamma h}$, resp.) produces in an additional term which corresponds to quasi-invariance under the shift
\begin{equation*}
h\mapsto h+\Big(\frac{n}{\gamma}+\frac\gamma2\Big) \varphi
\end{equation*}
(or $h\mapsto h+\frac{n}{\gamma}\varphi+\frac\gamma2 \bar\varphi$ for some function $\bar\varphi$ given in terms of $\varphi$).
We derive rigorous statements below.
For convenience, we treat the two procedures --in spite of their similarity-- for the `plain' and `adjusted' cases separately.
\end{remark}
\begin{remark} 
  The approach involving the adjusted measure $\boldsymbol{{\bar\nu}}^{*}_g$  is similar to (and inspired by) that of \cite{GuiRhoVar19} in the case $n=2$.
 Results concerning the plain measure~$\boldsymbol{\nu}^{*}_g$ seem to be new even in the two-dimensional case.
\end{remark}

\subsection{The plain Polyakov--Liouville measure}\label{ss:PlainPL}
Let us address  the challenge of giving
a rigorous meaning to 
 \begin{equation*}
   d\boldsymbol{\pi}_g(h) =
   \exp\left(- \int \tparen{\Theta\,Q_g h + \Theta^*\langle h\rangle_g+m e^{\gamma h}} d \vol_g\right) \exp\left(-\frac{1}{2} \mathfrak{p}_g(h,h)\right) dh
 \end{equation*} 
on {admissible} manifolds of arbitrary even dimension.
Assume for the sequel that $|\gamma|<\sqrt{2n}$,
and let 
\[
\boldsymbol{\nu}_g\coloneqq {\sf CGF}_{g}
\]
denote the law of the the co-polyharmonic field,    a (rigorously defined) probability measure on $\mathring H_g^{-\varepsilon}$ for some/any $\varepsilon>0$. For convenience, we choose $\varepsilon=n/2$.
Furthermore, let 
\begin{equation}\widehat{\boldsymbol{\nu}}_g\coloneqq \widehat{{\sf CGF}}_{g}
\end{equation}
denote the (infinite) measure on   $H_g^{-n/2}$ introduced in Proposition \ref{CGF-factorized} as the distribution of the ungrounded co-polyharmonic field on $(M,g)$.
 As outlined in Section \ref{s:copolyharmonic-field} and Remark \ref{normalized-heuristic} the latter admits a heuristic characterization as
\begin{equation*}
  d\widehat{\boldsymbol{\nu}}_g(h)=\frac1{Z_g}\exp\Big(-\frac{1}{2} \mathfrak{p}_g(h,h)\Big)dh
\end{equation*}
with $$Z_g=\sqrt{\frac{\vol_g(M)}{\det'(\frac1{2\pi}\mathsf p_g)}}.$$

  Proceeding as in the two-dimensional case, in terms of this measure, we define the measure 
  \begin{equation}
d{\boldsymbol{\nu}}^*_g(h)\coloneqq \exp\Big(-\Theta\scalar{h}{Q_g}_g- \Theta^*\langle h\rangle_g -m \,\mu^{\gamma h}_{g}(M)\Big)\, d\widehat{\boldsymbol{\nu}}_g(h)
\end{equation}
on $H_g^{-n/2}$
 with  associated \emph{partition function} 
\begin{align*}
  Z^*_g\coloneqq \int_{H^{-n/2}}d\,\boldsymbol{\nu}^*_g(h)\ \in (0,\infty]\comma
\end{align*}
 where $\Theta, \Theta^*, m, \gamma\in\R$ are parameters with $m>0$, $0<|\gamma|<\sqrt{2n}$, and 
where~$\mu^{\gamma h}_{g}$ denotes the plain Liouville Quantum Gravity measure on the $n$-dimensional manifold $M$.
Moreover,~$\scalar{h}{Q_g}_g$ denotes the pairing between the random field $h$ and the (scalar valued) $Q$-curvature, and $\langle h\rangle_g=\frac1{\vol_g(M)}\scalar{h}{\mathbf 1}$ denotes the pairing between $h$ and the constant function~$\mathbf 1$, normalized by the volume of $M$.
Set $Q(M)\coloneqq Q(M,g)$. 

\begin{theorem}\label{t:finiteness-polyakov-liouville}
  Assume that $0<\gamma<\sqrt{2n}$ and $\Theta\,Q(M)+\Theta^*<0$. Then ${\boldsymbol{\nu}}^*_g$ is a finite measure
  and so is
  $$\boldsymbol{\pi}_{g}\coloneqq \sqrt{\frac{\vol_g(M)}{\det'(\frac1{2\pi}\mathsf p_g)}}\cdot  \boldsymbol{\nu}^{*}_{g}.$$ 
  The normalizations of them coincide and provide a well-defined probability measure on $H_g^{-n/2}$,
  $${\boldsymbol{\nu}}_g^\sharp\coloneqq \frac1{Z^*_g}{\boldsymbol{\nu}}^*_g= \frac1{Z^\pi_g}{\boldsymbol{\pi}}_g
  $$
with   $Z^\pi_g\coloneqq \int_{H^{-n/2}}d\,\boldsymbol{\pi}_g(h)$.
\end{theorem}

\begin{proof} 
\begin{align*}
  Z^*_g&=\int_{H^{-n/2}}\exp\Big(-\Theta\scalar{h}{Q_g}_g- \Theta^*\langle h\rangle_g-m \,\mu^{\gamma h}_{g}(M)\Big)\, d\widehat{\boldsymbol{\nu}}_g(h)\\
       &=\int_{H^{-n/2}}\int_\R\exp\Big(-\Theta\scalar{h}{Q_g}_g-a\big(\Theta  Q(M)+\Theta^*\big)- m \,e^{\gamma a}\mu^{\gamma h}_{g}(M)\Big)\, da\,d{\boldsymbol{\nu}}_g(h)\\
&\stackrel{(a)}=
\int_{H^{-n/2}}e^{-\Theta\scalar{h}{Q_g}_g}\,\int_0^\infty\bigg(\frac t{m\, \mu^{\gamma h}_{g}(M)}\bigg)^{-\frac{\Theta\,Q(M)+\Theta^*}\gamma}\, e^{- t}\, \frac{dt}{\gamma t}\,d{\boldsymbol{\nu}}_g(h)
\\
&\stackrel{(b)}=\frac1\gamma\,\Gamma\Big(-\frac{\Theta\,Q(M)+\Theta^*}\gamma\Big)\cdot
\int_{H^{-n/2}}e^{-\Theta\scalar{h}{Q_g}_g}\,
\big({m\, \mu^{\gamma h}_{g}(M)}\big)^{\frac{\Theta\,Q(M)+\Theta^*}\gamma}\,
d{\boldsymbol{\nu}}_g(h)\fstop
\end{align*}
Here $(a)$ follows by change of variables $a \mapsto t\coloneqq m \,e^{\gamma a}\mu^{h}_{g,\gamma}(M)$, and $(b)$ by the very definition of Euler's $\Gamma$ function.
The final integral then can be estimated according to 
\begin{align*}
  \int_{H^{-n/2}} & e^{-\Theta\scalar{h}{Q_g}_g}\,
\big({m\, \mu^{\gamma h}_{g}(M)}\big)^{\frac{\Theta\,Q(M)+\Theta^*}\gamma}\,
d{\boldsymbol{\nu}}_g(h)
\\
  \leq&\ \bigg(\int_{H^{-n/2}}e^{-2\Theta\scalar{h}{Q_g}_g}\,
d{\boldsymbol{\nu}}_g(h)\bigg)^{1/2}
\cdot
\bigg(\int_{H^{-n/2}}
\big({m\, \mu^{\gamma h}_{g}(M)}\big)^{\frac{2(\Theta\,Q(M)+\Theta^*)}\gamma}\,
d{\boldsymbol{\nu}}_g(h)\bigg)^{1/2}\fstop
\end{align*}
The finiteness of the first term on the right-hand side is obvious by the defining property of~${\boldsymbol{\nu}}_g$:
\begin{equation*}
  \int_{H^{-n/2}}e^{-2\Theta\scalar{h}{Q_g}_g}\,
  d{\boldsymbol{\nu}}_g(h)=e^{2\Theta^2\,\mathcal{K}_g(Q_g,Q_g)} \fstop
\end{equation*}
The finiteness of $\int_{H^{-n/2}}
{\mu^{\gamma h}_{g}(M)}^{\frac{2(\Theta\,Q(M)+\Theta^*)}\gamma}\,
d{\boldsymbol{\nu}}_g(h)$ for $\frac{\Theta\,Q(M)+\Theta^*}\gamma<0$ follows from Theorem~\ref{t:existence-liouville}~\ref{i:liouville-moments}. 
\end{proof}

\begin{remark} 
Assuming that $\Theta$ is positive, the finiteness assumption $\Theta\,Q(M)+\Theta^*<0$ in the above theorem is equivalent to saying that the
constant $-\Theta^*/\Theta$  is larger than the total $Q$-curvature.
\end{remark}


\begin{theorem}\label{t:invariance-polyakov-liouville}
  Assume   that $0<\gamma<\sqrt{2n}$, $\Theta=a_n\,\frac n\gamma$, and $\Theta^*=  \gamma$. 
  Then ${\boldsymbol{\nu}}^*_g$ is conformally quasi-invariant modulo shift in the  sense that for all $\varphi\in\C^\infty(M)$ and $g'=e^{2\varphi}g$,
\begin{equation}
{\boldsymbol{\nu}}^*_{g'}= Z(g,\varphi) \cdot T_*
{\boldsymbol{\nu}}^*_{g}\comma
\end{equation}
where $T_*$ denotes the push forward under the shift 
$T: h\mapsto h-\frac n\gamma \varphi-\frac\gamma2\bar\varphi$ on~$H_g^{-n/2}$   with  $\bar\varphi$ defined as in~\eqref{def-bar-phi}. 
The \emph{A-type conformal anomaly}  $Z(g,\varphi)$ is given as 
  \begin{equation}
   Z(g,\varphi)\coloneqq     \exp\braket{ \Theta \int \left(\frac{n}{\gamma} \varphi + \frac{\gamma}{2} \bar{\varphi}\right) \, Q_{g} d \vol_{g} + n\av{\varphi}_{g'}+ \frac{n^{2}}{2\gamma^{2}} \mathfrak{p}_{g}(\varphi, \varphi) } \fstop
 \end{equation}
 In particular,
$Z(g,\varphi)=Z_{e^{2\varphi}g}^*/Z_g^*$ provided $Z_g^*<\infty$.
 \end{theorem}

\begin{proof}
For the sake of brevity let us write
\begin{equation*}
  S_{g}(h) = \Theta \scalar{h}{Q_{g}}_g + \Theta^{*}\av{h}_g + m \mu^{\gamma h}_{g}(M) \fstop
\end{equation*}
We also set $\Phi \coloneqq (\frac{n}{\gamma} \varphi + \frac{\gamma}{2} \bar{\varphi}) \in \C^\infty(M)$.
Let $F \colon H_g^{-n/2} \to \mathbb{R}_{+}$ be measurable.
  Then by Girsanov Theorem (Corollary \ref{girsanov2}) for $\boldsymbol{\hat{\nu}}_{g'}$, we find that
\begin{align*}
\int_{H^{-n/2}} F d \boldsymbol{{\nu}}^{*}_{g'} =& \int F\left(h - \Phi  \right) \exp\left(-S_{g'}\left(h - \Phi\right)\right)
\\
&\cdot \exp\left(  
 \boldsymbol{\langle\!\!\langle} h|\mathsf p_{g'}\Phi\boldsymbol{\rangle\!\!\rangle}_{g'}
- \frac{1}{2}\mathfrak{p}_{g'}(\Phi, \Phi)\right) d\boldsymbol{\hat{\nu}}_{g'}(h) \fstop
\end{align*} 
By Corollary \ref{t:invariance-lqg-lebesgue} and Theorem \ref{t:existence-liouville}\ref{i:liouville-Cameron--Martin-shift}, we have that
\begin{align*}
\int_{H^{-n/2}} F d \boldsymbol{{\nu}}^{*}_{g'} =& \int F(h - \Phi ) \exp\left(- \Theta\scalar{h - \Phi}{Q_{g'}}_{g'} - m \mu_{g}^{\gamma h}(M)\right)
\\
&\cdot \exp\left(-{\Theta^{*}} \scalar{h - \Phi}{{\mathbf 1}}_{g'}\right)
\\
&\cdot \exp\left(  \scalar{h}{\mathsf p_{g'} \Phi}_{g'} - \frac{1}{2}\mathfrak{p}_{g'}(\Phi, \Phi)\right) d\boldsymbol{\hat{\nu}}_{g}( h) \fstop
\end{align*}
Now recall that $\mathsf p_{g'} u = \ee^{-n\varphi} \mathsf p_{g}u$, that $\mathfrak{p}_{g}$ is conformally invariant, and that, by Proposition~\ref{q-transform}, $Q_{g'} = \ee^{-n\varphi}(Q_{g} +\frac1{a_n} \mathsf p_{g} \varphi)$.
Thus, we obtain
\begin{align*}
\int_{H^{-n/2}} F d \boldsymbol{{\nu}}^{*}_{g'} =& \int F(h - \Phi) \exp\left(- \Theta\scalar{h - \Phi}{Q_{g}} _{g}- m \mu_{g}^{\gamma h}(M)\right)
\\
&\cdot\exp\left(- \Theta^{*}\scalar{h - \Phi}{{\mathbf 1}}_{g'}\right)
\\
&\cdot\exp\Bigg(-\frac{\Theta}{a_n} \scalar{h - \Phi}{\mathsf p_{g} \varphi}_{g} +  \scalar{h}{\mathsf p_{g} \Phi}_{g}
 -\frac{1}{2} \mathfrak{p}_{g}(\Phi, \Phi)\Bigg) d\boldsymbol{\hat{\nu}}_{g}(h) \fstop
\end{align*}
Since we have chosen $\Theta = a_{n} \frac{n}{\gamma}$ some of the terms in the last line cancel out and we get:
\begin{equation*}
    \begin{split}
      \int_{H^{-n/2}} F d \boldsymbol{{\nu}}^{*}_{g'} = \int & F(h - \Phi) \exp\left(- \Theta\scalar{h - \Phi}{Q_{g}}_g - m \mu_{g}^{\gamma h}(M)\right)
      \\
                                                                      & \exp\left(- \Theta^* \scalar{h - \Phi}{{\mathbf 1}}_{g'}\right) \\
                                                                      & \exp\left( \frac{\gamma}{2} \scalar{h}{\mathsf p_{g} \bar{\varphi}}_g + \frac{n^{2}}{2\gamma^{2}} \mathfrak{p}_{g}(\varphi, \varphi) - \frac{\gamma^{2}}{8} \mathfrak{p}_{g}\left(\bar{\varphi},\bar{\varphi}\right)\right) d\boldsymbol{\hat{\nu}}_{g}(h) \fstop
  \end{split}
\end{equation*}
Now by definition of $\bar{\varphi}$, we get 
\begin{align*}
\mathsf p_{g} \bar{\varphi} = \frac{2}{\vol_{g'}(M)} \pi_{g}(\ee^{n\varphi}) = \frac{2}{\vol_{g'}(M)} \ee^{n\varphi} - \frac{2}{\vol_{g}(M)} \fstop
\end{align*}
In particular, for every~$h\in H_g^{-n/2}$,
\begin{align}\label{eq:t:PolyakovLioville:1}
{\scalar{h}{\mathsf p_{g} \bar{\varphi}} = 2 \av{h}_{g'}-2\av{h}_g} \fstop
\end{align}

As a consequence,
\begin{equation*}
\begin{split}
\int_{H^{-n/2}} F d \boldsymbol{{\nu}}^{*}_{g'} = \int & F(h - \Phi) \exp\left(- \Theta\scalar{h}{Q_{g}}_g +\Theta \int \Phi Q_g d\vol_g  - m \mu_{g}^{\gamma h}(M)\right)
\\
& \exp\left(- \Theta^{*}\av{h}_{g'}+\Theta^*\av{\Phi}_{g'} \right)
\\
&\exp\left(\gamma\big(\av{h}_{g'}-\av{h}_g\big) + \frac{n^{2}}{2\gamma^{2}} \mathfrak{p}_{g}(\varphi, \varphi) -  \frac{\gamma^{2}}{8} \mathfrak{p}_{g}\left(\bar{\varphi},\bar{\varphi}\right)\right) d\boldsymbol{\hat{\nu}}_{g}(h) \fstop
  \end{split}
\end{equation*}

Thus the choice of $\Theta^{*} = \gamma$, after cancellations and rearrangement, yields 
\begin{align}
\nonumber
\int_{H^{-n/2}} F d \boldsymbol{\nu}^{*}_{g'} = & \int  F(h - \Phi ) \exp\left(- \Theta\scalar{h}{Q_{g}}_g - \Theta^{*} \av{h}_g - m \mu_{g}^{\gamma h}(M)\right) d\boldsymbol{\hat{\nu}}_{{g}}(h)
\\
\label{eq:T:PolyakovLiouville:2}
& \cdot \exp\left(\int \Phi \left(\Theta Q_{g} + \Theta^{*}\frac{e^{n\varphi}}{\vol_{g'}(M)}\right) d \vol_{g} + \frac{n^{2}}{2\gamma^{2}} \mathfrak{p}_{g}(\varphi, \varphi) -  \frac{\gamma^{2}}{8} \mathfrak{p}_{g}(\bar{\varphi}, \bar{\varphi}) \right) \fstop
\end{align}
Again in light of~\eqref{eq:t:PolyakovLioville:1} we further have that
\begin{align*}
\frac{1}{2}\mathfrak{p}_g(\bar\varphi,\bar\varphi)= \av{\bar\varphi}_{g'}-\av{\bar\varphi}_g\fstop
\end{align*}
Substituting the definitions of~$\Theta^*\coloneqq\gamma$ and~$\Phi\coloneqq (\frac{n}{\gamma}\varphi+\frac{\gamma}{2}\bar\varphi)$ then yields
\begin{align*}
\int \Phi& \left(\Theta Q_{g} + \Theta^{*}\frac{e^{n\varphi}}{\vol_{g'}(M)}\right) d \vol_{g} + \frac{n^{2}}{2\gamma^{2}} \mathfrak{p}_{g}(\varphi, \varphi) -  \frac{\gamma^{2}}{8} \mathfrak{p}_{g}(\bar{\varphi}, \bar{\varphi}) =
\\
=&\ \Theta \int \Phi Q_g d \vol_g +n\av{\varphi}_{g'}+\frac{\gamma^2}{4}\paren{\av{\bar\varphi}_{g'} + \av{\bar\varphi}_g} +  \frac{n^{2}}{2\gamma^{2}} \mathfrak{p}_{g}(\varphi, \varphi) \fstop
\end{align*}
Finally, by definition~\eqref{def-bar-phi} of~$\bar\varphi$, and since~$\mathsf{k}_g(e^{n\varphi})\in \mathring{L}^2_g$,
\begin{align*}
\av{\bar\varphi}_{g'}+\av{\bar\varphi}_g=&\ \frac{2}{\vol_{g'}(M)^2}\int_M e^{n\varphi} \mathsf{k}_g(e^{n\varphi}) d\vol_{g}-\frac{1}{\vol_{g'}(M)^2} \mathcal{K}_g(e^{n\varphi},e^{n\varphi})
\\
&\ + \frac{2}{\vol_{g'}(M)\vol_g(M)}\int_M \mathsf{k}_g(e^{n\varphi}) d\vol_g - \frac{1}{\vol_{g'}(M)^2} \mathcal{K}_g(e^{n\varphi},e^{n\varphi})
\\
=&\ \frac{2}{\vol_{g'}(M)^2} \mathcal{K}_g(e^{n\varphi}, e^{n\varphi})-\frac{1}{\vol_{g'}(M)^2} \mathcal{K}_g(e^{n\varphi},e^{n\varphi})
\\
&\ +0 -\frac{1}{\vol_{g'}(M)^2} \mathcal{K}_g(e^{n\varphi},e^{n\varphi})
\\
=&\ 0
\comma
\end{align*}
and therefore
\begin{align}
\nonumber
\int \Phi& \left(\Theta Q_{g} + \Theta^{*}\frac{e^{n\varphi}}{\vol_{g'}(M)}\right) d \vol_{g} + \frac{n^{2}}{2\gamma^{2}} \mathfrak{p}_{g}(\varphi, \varphi) -  \frac{\gamma^{2}}{8} \mathfrak{p}_{g}(\bar{\varphi}, \bar{\varphi}) =
\\
\label{eq:T:PolyakovLiouville:3}
=&\ \Theta\int \Phi Q_g d \vol_g +n\av{\varphi}_{g'}+ \frac{n^{2}}{2\gamma^{2}} \mathfrak{p}_{g}(\varphi, \varphi) \fstop
\end{align}
Substituting~\eqref{eq:T:PolyakovLiouville:3} into~\eqref{eq:T:PolyakovLiouville:2}, we finally have that
\begin{align*}
\int_{H^{-n/2}} F d \boldsymbol{\nu}^{*}_{g'} = & \int  F(h - \Phi) \exp\left(- \Theta\scalar{h}{Q_{g}}_g - \Theta^{*} \av{h}_g - m \mu_{g}^{\gamma h}{(M)}\right) d\boldsymbol{\hat{\nu}}_{{g}}(h)
\\
& \cdot \exp\left(\Theta\int \Phi Q_g d \vol_g +n\av{\varphi}_{g'}+ \frac{n^{2}}{2\gamma^{2}} \mathfrak{p}_{g}(\varphi, \varphi) \right) \fstop
\end{align*}
This concludes the proof of the conformal quasi-invariance.
To conclude for the expression of $Z(g,\varphi)$ we take $F=\mathbf 1$.
\end{proof}

\begin{corollary}\label{polyakov-inv} Assume that $\Theta=a_n\,\frac n\gamma$,  $\Theta^*=\gamma$, and $\gamma^2<-n\,a_n\,Q(M)$.
  Then $Z(g,\varphi)=\frac{Z^*_{g'}}{Z^*_g}$, and
${\boldsymbol{\nu}}_g^\sharp$ is conformally invariant modulo shift:
    \begin{equation}
      {\boldsymbol{\nu}}_{e^{2\varphi}g}^\sharp=T_*
      {\boldsymbol{\nu}}^\sharp_{g}
  \end{equation}
with
  $T: h\mapsto h-n \varphi/\gamma-\gamma \bar\varphi/2$.
\end{corollary}
\begin{proof}
  We have
  \begin{equation*}
    \boldsymbol{\nu}^{\sharp}_{g'} = \frac{\boldsymbol{\nu}^{*}_{g'}}{{Z}^{*}_{g'}} = \frac{Z(g,\varphi)}{{Z}^{*}_{g'}} \cdot T_{*} \boldsymbol{\nu}_{g}^{*} = \frac{{Z}^{*}_{g}}{{Z}^{*}_{g'}} Z(g,\varphi) \cdot T_{*} \boldsymbol{\nu}^{\sharp} = T_{*} \boldsymbol{\nu}^{\sharp} \fstop \qedhere
  \end{equation*}
\end{proof}

\begin{remark} With the choices 
$\Theta\coloneqq \frac{na_n}\gamma$ and $\Theta^*\coloneqq \gamma$
from above, the condition 
$\Theta\,Q(M)+\Theta^*<0$ reads as
$\frac{a_nn}{\gamma^2}Q(M)+1<0$
or, in other words,
\begin{equation}\gamma^2<-n\,a_n\,Q(M) \fstop
\end{equation}
\end{remark}

\begin{remark}
  In view of Corollaries \ref{extend-field} and \ref{extend-conf-meas}, the quasi-invariance assertion in the previous Theorem \ref{t:invariance-polyakov-liouville} and Corollary \ref{polyakov-inv} also holds under the more general class of conformal transformations in the sense of Definition \ref{d:ConformalEquiv} \ref{i:d:ConformalEquiv:2}.
  In particular, the plain Polyakov--Liouville measure is quasi-invariant under isometric transformations $\Phi: M\to M'$.
\end{remark}

\subsection{The adjusted Polyakov--Liouville measures}\label{ss:AdjustedPL}

As anticipated, our results for the plain Polyakov--Liouville measure can also be recast in the setting of the adjusted Polyakov--Liouville measure.
Let us set
\begin{equation*}
d \boldsymbol{{\bar\nu}}_{g}^{*}(h) = \exp\left(-\Theta \scalar{h}{Q_{g}}_g - m \bar\mu_{g}^{\gamma h}(M)\right) d\boldsymbol{\widehat{\nu}}_{g}(h) \comma
\end{equation*}
which corresponds to the \emph{adjusted Polyakov--Liouville measure}.
The associated \emph{partition function} is
\begin{align*}
  &   {\bar Z}^*_g\coloneqq \int_{H^{-n/2}}d\,{\boldsymbol{{\bar\nu}}}^*_g(h)\fstop
\end{align*}

As for the plain measure, we have the following result.
\begin{theorem}\label{fin-adj-pol-liou}
  Assume that $0<\gamma<\sqrt{2n}$ and $\Theta\,Q(M)<0$.
  Then ${\boldsymbol{\bar\nu}}^*_g$ is a finite measure and so is
  \[
  \boldsymbol{\bar\pi}_{g}\coloneqq \sqrt{\frac{\vol_g(M)}{\det'(\frac1{2\pi}\mathsf p_g)}}\cdot  \boldsymbol{\bar\nu}^{*}_{g}.
  \]
\end{theorem}

{The measure~$ \boldsymbol{\bar\pi}_{g}$ is in fact the `standard' Polyakov--Liouville measure considered in the two-dimensional case; see e.g.~\cite[\S3.2]{DKRV16} and \cite[Prop.~4.1]{GuiRhoVar19}}

\begin{proof}
  The proof is the same as in Theorem~\ref{t:finiteness-polyakov-liouville}, simply remarking that Theorem~\ref{t:existence-liouville}~\ref{i:liouville-moments} also applies for $\bar\mu^{h}$ instead of $\mu^{h}$.
\end{proof}

{The next result extends to admissible manifolds the same assertion for spheres in~\cite[Thm.~3.10]{Cercle}.}

\begin{theorem}\label{conf-adj-pol-liou}
 Assume that $0<\gamma<\sqrt{2n}$ and  that $\Theta = a_{n} (\frac{n}{\gamma} + \frac{\gamma}{2})$.
  Let $\varphi$ be smooth and $g' = \ee^{2\varphi} g$.
  Then $\boldsymbol{\bar\nu}^{*}$ is conformally quasi-invariant under the shift $T \colon h \mapsto h - (\frac{n}\gamma+\frac\gamma2) \varphi$, viz.
  \begin{equation*}
    \boldsymbol{\bar\nu}^{*}_{g'} =  \bar Z(g,\varphi)\cdot T_{*} \boldsymbol{\bar\nu}^{*}_{g} \comma
  \end{equation*}
  where
  \begin{equation*}
    \bar Z(g,\varphi) = \exp\left(\Big(\frac{n}{\gamma} + \frac{\gamma}{2}\Big)^{2} \left[\frac12\mathfrak{p}_{g}(\varphi,\varphi) + {a_n} \int \varphi\, Q_{g}\, d \vol_{g} \right] \right) \fstop
  \end{equation*}
\end{theorem}

\begin{proof}
  Let $F \colon H^{-n/2}_g \to \mathbb{R}_{+}$ be measurable.
  Write $\Phi = (\frac{n}{\gamma} + \frac{\gamma}{2}) \varphi \in\C^\infty(M)$.
  By Girsanov's theorem for $\boldsymbol{\hat \nu}_{g'}$ (Corollary \ref{girsanov2}), we have:
  \begin{equation*}
    \begin{split}
      \int_{H^{-n/2}} F d \boldsymbol{\bar \nu}^{*}_{g'} = \int & F(h - \Phi) \exp\left(-\Theta \scalar{h - \Phi}{Q_{g'}}_{g'} - m \bar\mu_{g'}^{\gamma(h - \Phi)}(M)\right)
      \\
      &\cdot\exp\left(  \scalar{h}{\mathsf p_{g'} \frac{\Theta}{a_{n}}\varphi}_{g'} - \frac{1}{2} \mathfrak{p}_{g'}\left(\frac{\Theta}{a_{n}}\varphi, \frac{\Theta}{a_{n}}\varphi\right)\right) d\boldsymbol{\widehat \nu}_{g'}(h) \fstop
    \end{split}
  \end{equation*}
  In view of Theorems \ref{CGF-factorized} and \ref{t:invariance-adjusted}, we thus get
  \begin{equation*}
    \begin{split}
      \int_{H^{-n/2}} F d \boldsymbol{\bar \nu}^{*}_{g'} = \int & F(h - \Phi) \exp\left(-\Theta \scalar{h - \Phi}{Q_{g} +\frac1{a_n}\mathsf p_{g} \varphi}_g - m \bar\mu_{g}^{\gamma h}{(M)}\right) \\
                                                                         & \cdot\exp\left(\frac{\Theta}{a_n} \scalar{ h}{ \mathsf p_{g}\varphi}_g - \frac{\Theta^{2}}{2a^2_{n}} \mathfrak{p}_{g}(\varphi,\varphi)\right) d\boldsymbol{\widehat \nu}_{g}(h) \fstop
    \end{split}
  \end{equation*}
  Expanding $\scalar{h - \Phi_{g}}{Q_{g} + \frac1{a_n}\mathsf p_{g} \varphi}_g$ cancels out with some term on the second line and we obtain the announced result.
\end{proof}
\begin{corollary} Assume $Q(M)<0$ and set $ \boldsymbol{\bar\nu}^{\sharp}_{g} \coloneqq \frac1{\bar Z^*} \boldsymbol{\bar\nu}^*_{g}$.
  Then with $\Theta$ and $T$ 
  as above,
  \begin{equation*}
    \boldsymbol{\bar\nu}^{\sharp}_{g'} = T_{*} \boldsymbol{\bar\nu}^{\sharp}_{g} \fstop
  \end{equation*}
\end{corollary}

\begin{remark} In dimension $n=2$, the result of the previous Theorem is consistent with the results in \cite{GuiRhoVar19}. Indeed, the main result there (Thm.~1) deals with  conformal changes of the Polyakov-Liouville measure
$$\boldsymbol{\bar\pi}_{g}\coloneqq \sqrt{\frac{\vol_g(M)}{\det'(-\frac1{4\pi^2}\Delta_g)}}\cdot  \boldsymbol{\bar\nu}^{*}_{g}.$$
In our notation and with $\Theta$ and $T$ as in Theorem \ref{conf-adj-pol-liou}, the conformal anomaly for this measure in the case  $n=2$ is
  \begin{equation}\label{cencha}
\frac{d\boldsymbol{\bar\pi}_{g'}}{d T_*\boldsymbol{\bar\pi}_{g}} \ = \    \exp\left(\left[\frac16+\Big(\frac{2}{\gamma} + \frac{\gamma}{2}\Big)^{2}\right]\cdot \left[\frac12\mathfrak{p}_{g}(\varphi,\varphi) + {a_n} \int \varphi\, Q_{g}\, d \vol_{g} \right] \right) \fstop
  \end{equation}
  

 {The factor $\exp(\frac16\left[\frac12\mathfrak{p}_{g}(\varphi,\varphi) + {a_n} \int \varphi\, Q_{g}\, d \vol_{g} \right] $ here accounts for} the conformal transformation of the pre-factor $\sqrt{\frac{\vol_g(M)}{\det'(-\frac1{2\pi}\Delta_g)}}$ of the measure $\boldsymbol{\bar\pi}_{g}$.
  In dimension $n=2$, 
   $$\log{{\det}'(-t\Delta_g)}=\log{{\det}'(-\Delta_g)}+\log(t)\, \Big(1-\frac{\chi(M)}6\Big)\qquad\forall t>0$$
   with the latter term on the right-hand side here being conformally invariant,
   and the conformal change of   $$\log{\frac{\vol_{g'}(M)}{\det'(-\Delta_{g'})}}-\log{\frac{\vol_g(M)}{\det'(-\Delta_g)}}$$ 
   is given explicitly  in terms of the  so-called 
     \emph{Polyakov formula}, \cite[Equ. (1.13)]{OPS88}, {cf.~\eqref{ops}. } 
     The quantity ${\sf c}:=1+6\big(\frac{2}{\gamma} + \frac{\gamma}{2}\big)^{2}$ appearing in \eqref{cencha} is the celebrated \emph{central charge}, taking values in the interval $(25,\infty)$.
     
{Indeed, with this choice of $\sf c$, with $\omega=2\varphi$,  $a_2=\frac1{2\pi}$, $K_g=2 Q_g$, and $\sf p_g(\omega,\omega)=\frac1{2\pi}\int |d\omega|_g^2d \vol_{g}$, 
     the right hand side of 
     \eqref{cencha}
      reads  as
        \begin{equation}
    \exp\left(\frac{\sf c}{96\pi}\cdot \int\left[|d\omega|_g^2
 +2 K_g\omega \right] d \vol_{g}  \right) 
  \end{equation}
  exactly as in  \cite{GuiRhoVar19}.}
\end{remark}

\begin{corollary} Assume that $n=4$ and put similarly as before 
$$\boldsymbol{\bar\pi}_{g}\coloneqq \sqrt{\frac{\vol_g(M)}{\det'(\frac1{2\pi}\mathsf p_g)}}\cdot  \boldsymbol{\bar\nu}^{*}_{g}.$$ Then with $\Theta$ and $T$ as in Theorem \ref{conf-adj-pol-liou},
  \begin{align*}
\frac{d\boldsymbol{\bar\pi}_{g'}}{dT_*\boldsymbol{\bar\pi}_{g}} \ = \ &  \exp\left(\left[\frac{7}{45}+\Big(\frac{4}{\gamma} + \frac{\gamma}{2}\Big)^{2}\right]\cdot \left[\frac12\mathfrak{p}_{g}(\varphi,\varphi) + {a_n} \int \varphi\, Q_{g}\, d \vol_{g} \right] \right)\\
 \cdot&\exp\left( \frac1{45 \pi^2} \left[-\int \scal_{g'}^2 d\vol_{g'}+\int \scal_{g}^2 d\vol_{g}\right]\right)
 \cdot \exp\left(\frac1{1440 \pi^2}\int \varphi |W|^2d\vol_g\right)
  \end{align*}
  where $W$ denotes the Weyl tensor.
\end{corollary}

\begin{proof} Let $\zeta_g(s)=\sum_{j\ge 1} (\nu_j/a_n)^{-s}$. Then $-\log{\det}'(\mathsf P_g)=\zeta_g'(0)$ and thus  $-\log{\det}'(t\mathsf P_g)=\zeta_g'(0)-\log(t)\,\zeta_g(0)$ for every $t>0$. The value 
$$\zeta_g(0)=-1+\int U_4(\mathsf P_g)d\vol_g$$
is a conformal invariant, \cite[Lemma 2]{Bra96}.
Therefore,
%
  $$\log{\frac{\vol_{g'}(M)}{\det'(\frac1{2\pi}\mathsf p_{g'})}}-\log{\frac{\vol_g(M)}{\det'(\frac1{2\pi}\mathsf p_g)}}=
  \log{\frac{\vol_{g'}(M)}{\det'(\mathsf P_{g'})}}-\log{\frac{\vol_{g}(M)}{\det'(\mathsf P_{g})}}.
$$ 
  The right-hand side is calculated explicitly in \cite[Thm.~4]{Bra96}.
  Together with Theorem \ref{conf-adj-pol-liou} above this yields the claim.
\end{proof}

\subsection{Some examples}

The  above assertions impose 
two conditions  on a given manifold $(M,g)$:
positivity of the co-polyharmonic operator $\Pol_g$ and negativity of the total $Q$-curvature $Q(M)$.

Let us  present some examples of such manifolds.

\begin{example}[$n=2$] Every compact Riemannian surface of genus $\ge2$  satisfies both of these conditions. 
\end{example}

\begin{example}[$n=2,6,10,\ldots$] Every compact hyperbolic Riemannian manifold of dimension $n=4\ell+2$ for some $\ell\in\N$ and with
$\lambda_1>\frac{n(n-2)}4$ satisfies both of these conditions. 
\end{example}
\begin{proof}
Combine Proposition \ref{t:admissible-hyperbolic} and Example \ref{ex: q-curvature}.
\end{proof}

\begin{example}[$n=4$] Let $M=M_1\times M_2$ where $M_1$ and $M_2$ are compact Riemannian surfaces  of  constant curvature $k_1$ and $k_2$, resp.
\begin{enumerate}[$(i)$]
\item\label{i:ex:Dim4:1} Then $Q_g<0$ if and only if 
\[
|k_1+k_2|<\sqrt3\cdot |k_1-k_2| \fstop
\]
\item\label{i:ex:Dim4:2} Furthermore, $\Pol_g>0$ on $\mathring H$ if  $k_1+k_2\ge0$.
\end{enumerate}
\end{example}
\begin{proof}
\ref{i:ex:Dim4:1} According to Example \ref{ex: q-curvature}~\ref{i:ex: q-curvature:1}, 
\begin{align*}
Q_g=-k_1^2-k_2^2+\frac23(k_1+k_2)^2=-\frac12(k_1-k_2)^2+\frac16(k_1+k_2)^2\fstop
\end{align*}

\ref{i:ex:Dim4:2} For $i=1,2$, denote by $\Pol_i=-\Delta_i$ the negative of the Laplacian on the manifold~$M_i$. 
Then by Proposition~\ref{einstein-poly}~\ref{i:einstein-poly:2},
\begin{align*}\Pol_g&=(\Pol_1+\Pol_2)^2-2k_1\Pol_1-2k_2\Pol_2+\frac43(k_1+k_2)(\Pol_1+\Pol_2)\\
&=\Pol_1(\Pol_1-2k_1)+\Pol_2(\Pol_2-2k_2)+2\Pol_1\Pol_2+\frac43(k_1+k_2)(\Pol_1+\Pol_2)
\ge0\end{align*}
according to the Lichnerowicz estimate $\Pol_i\ge 2k_i$ for $i=1,2$ (which is valid independent of the sign of $k_i$).
Indeed,  $\Pol_g$ is positive since the term $2\Pol_1\Pol_2$ is positive on the grounded $L^2$-space. 
\end{proof}

\section{Outlook: Discrete models and scaling limits}
In \cite{LDS-RH-EK-KTS2}, we study discrete approximations of the  Co-polyharmonic Gaussian Field $h$ on the continuous torus $\mathbb{T}^n = \mathbb{R}^{n} / \mathbb{Z}^{n}$ and of the associated  LQG measure $\mu=\mu_g^{\gamma h}$ on $\mathbb{T}^n$ for arbitrary even $n$ and  $\abs{\gamma} < \sqrt{2n}$.
The Co-polyharmonic Gaussian Field $h_L$ on the discrete torus $\mathbb{T}^n_L = \frac{1}{L} \mathbb{Z}^{n} / \mathbb{Z}^{n}$ for $L\in\N$ will be defined in complete analogy to the continuous case. Indeed, however, it also admits an instructive characterization as 
 random field on $\mathbb{R}^{\mathbb{T}^{n}_{L}}$   with law explicitly given by
\begin{equation*}
  c_n\, e^{-b_n\norm{(-\Delta_L)^{n/4}h}^2} dh,
\end{equation*}
where $dh$ is the Lebesgue measure and $\Delta_{L}$ is the discrete Laplacian. The 
associated discrete Liouville Quantum Gravity measure is the random measure on $\mathbb{T}^{n}_{L}$  given as
  \begin{equation*}
    \mu_{L}(dz) = \exp \paren{ \gamma h_L(z) - \frac{\gamma^{2}}{2} \mathbf{E} h_{L}(z) } dz.
  \end{equation*}

As $L\to\infty$, we prove convergence of the fields $h_L$ to the Co-polyharmonic Gaussian Field $h$ on the continuous torus $\mathbb{T}^n = \mathbb{R}^{n} / \mathbb{Z}^{n}$, as well as convergence of the random measures $\mu_L$ to the LQG measure $\mu$ on $\mathbb{T}^n$ for all $\abs{\gamma} < \sqrt{2n}$.

{\footnotesize


\end{document}